\documentclass
{article}
\usepackage{amsmath, amssymb, amsthm, graphicx,bm,
cite, 
comment}\usepackage{tikz-cd}
\usepackage[all]{xy}
\usepackage{titlesec}\usepackage{calligra,mathrsfs}
\DeclareMathOperator{\der}{Der}\DeclareMathOperator{\gr}{gr}
\DeclareMathOperator{\SU}{SU}

\DeclareMathOperator{\spec}{Spec}\DeclareMathOperator{\Mod}{mod}
\DeclareMathOperator{\Def}{Def}  
 
\theoremstyle{plain}
\newtheorem{thm}{Theorem}[section]
\newtheorem{lem}[thm]{Lemma}
\newtheorem{cor}[thm]{Corollary}
\newtheorem{prop}[thm]{Proposition}

\newtheorem*{thm*}{Theorem}
\newtheorem*{lem*}{Lemma}
\newtheorem*{cor*}{Corollary}

\theoremstyle{definition}
\newtheorem{dfn}[thm]{Definition}
\newtheorem{ex}[thm]{Example}
\newtheorem{rmk}[thm]{Remark}

\newtheorem*{dfn*}{Definition}

\newtheorem*{hp*}{Hypothesis}

\numberwithin{equation}{section}\numberwithin{figure}{section}

\def\fm{{\mathfrak m}}\def\sets{{\rm sets}}
 \def\lk{{ {\rm lk } }}\def\reg{{ { \rm reg}  }}

\def\e#1\e{\begin{equation}#1\end{equation}}
\def\iz#1\iz{\begin{itemize}#1\end{itemize}}
\def\ea#1\ea{\e{\begin{split}#1\end{split}}\e}
\def\eq{\eqref}
\def\l{\label}
\def\0{\hspace{0pt}}

\def\ph{\phi}

\def\Uh_#1{\,\widehat{\!U}_{\!#1}}

\def\Ph{\Phi}
\def\r{\rho}

\def\et{\eta}
\def\ps{\psi}
\def\vx{{\rm vx}}

\def\ind{\mathop{\rm ind}\nolimits}
\def\dim{\mathop{\rm dim}\nolimits}

\def\ker{\mathop{\rm ker}}
\def\im{\mathop{\rm im}}
\def\coker{\mathop{\rm coker}}
\def\Re{\mathop{\rm Re}}

\def\Ext{\mathop{\rm Ext}\nolimits}

\def\End{\mathop{\rm End}}

\def\ge{\geqslant}
\def\le{\leqslant\nobreak}

\def\cA{{{\mathcal A}}}
\def\cB{{{\mathcal B}}}

\def\cD{{{\mathcal D}}}
\def\cE{{{\mathcal E}}}
\def\cF{{{\mathcal F}}}
\def\cG{{{\mathcal G}}}

\def\cI{{{\mathcal I}}}
\def\cJ{{{\mathcal J}}}
\def\cK{{{\mathcal K}}}

\def\O{{{\mathcal O}}}

\def\cU{{{\mathcal U}}}
\def\cV{{{\mathcal V}}}
\def\={\equiv}

\def\cQ{{{\mathcal Q}}}

\def\cX{{{\mathcal X}}}
\def\cY{{{\mathcal Y}}}

\def\C{{{\mathbb C}}}

\def\P{{\mathbb{P}}}

\def\K{{{\mathbb K}}}
\def\Q{{{\mathbb Q}}}
\def\R{{{\mathbb R}}}
\def\Z{{{\mathbb Z}}}

\def\al{\alpha}
\def\be{\beta}
\def\ga{\gamma}
\def\de{\delta}
\def\io{\iota}
\def\ep{\epsilon}
\def\la{\lambda}

\def\th{\theta}
\def\ta{{\tau}}
\def\ze{\zeta}

\def\si{\sigma}
\def\om{\omega}
\def\De{\Delta}
\def\La{\Lambda}

\def\Th{\Theta}
\def\Om{\Omega}
\def\Ga{\Gamma}

\def\id{{\rm id}}

\def\d{{\rm d}}
\def\db{{\bar\bd}}
\def\ts{\textstyle}

\def\w{\wedge}
\def\-{\setminus}
\def\bt{{\bullet}}

\def\bop{\bigoplus}

\def\ov{\overline}

\def\ul{\underline}
\def\iy{\infty}

\def\cm{\circ}

\def\nb{\nabla}\def\sb{\subseteq}\def\bd{\partial}\def\sing{{\rm sing}}
\def\loc{{\rm loc}}\def\Art{{\rm Art}}
\begin{document}
\title{Deformations of Compact Calabi--Yau Conifolds}
\date{}
\author{Yohsuke Imagi}
\maketitle

\begin{abstract}
Let $X$ be a compact normal K\"ahler space whose canonical sheaf is a rank-one free $\O_X$ module and whose singularities are isolated, rational and quasi-homogeneous. We prove then that under a topological hypothesis the obstruction to deforming $X$ concentrates upon its singularities, generalizing partially the results of \cite{Nam1,Gross}. We prove also that under a certain hypothesis the locally trivial deformations of $X$ are unobstructed.  
\end{abstract}

\section{Introduction}
In this paper we generalize the results of \cite{Bog, Kaw, Ran,Ran1,Tian,Tian2,Tod}. We deal with the following class of complex analytic spaces (which we call complex spaces for short).
\begin{dfn}\l{dfn: CY conifolds}
A {\it compact Calabi--Yau $n$-conifold} is a compact normal K\"ahler space $X$ of dimension $n$ whose canonical sheaf is a rank-one free $\O_X$ module and whose singularities are isolated, rational and quasi-homogeneous.
\end{dfn}
The deformations of a compact Calabi--Yau $n$-fold $X$ may in general be obstructed \cite{Gross2}. But by \cite[Theorem 2.2]{Gross}, for $n=3$ the obstruction concentrates (in the sense of Definition \ref{dfn: concentrate}) upon the singular set $X^\sing.$ In particular, if $X^\sing$ is isolated and if the germ $(X,x)$ at every $x\in X^\sing$ has unobstructed deformations then $X$ itself has unobstructed deformations too. 

In this paper we prove the following theorem. As in \S\ref{sect: Kahler cones}, if $X$ is a compact Calabi--Yau conifold then for each $x\in X^\sing$ there exist a K\"ahler cone $C_x$ with vertex $\vx$ and a germ isomorphism $(X,x)\cong (C_x,\vx).$ We can define in particular, for $q=0,1,2,\dots,$ the $q^{\rm th}$ Betti number $b_q(C_x^\reg)$ of the regular locus $C_x^\reg:=C_x\-\{\vx\}.$ These are independent of the choice of $(C_x,\vx)$ and depend only on the topology of $(X,x).$
\begin{thm}\l{main thm1}
Let $X$ be a compact Calabi--Yau $n$-conifold such that for each $x\in X^\sing$ we have $b_{n-2}(C_x^\reg)=0$ or $b_{n-1}(C_x^\reg)=0.$ The obstruction to deforming $X$ then concentrate upon its singularities.
\end{thm}
For $n=3,$ by \cite[Theorem 2.2]{Gross} the same statement holds without the hypothesis that the singularities should be isolated and quasi-homogeneous. The Betti number condition is not necessary either. That proof however is based upon the method of \cite{Nam1} which applies only to $n=3.$ Theorem \ref{main thm1} on the other hand applies to every $n.$ We prove also
\begin{thm}\l{main thm2}
Let $X$ be a compact Calabi--Yau $n$-conifold with $n\ge3$ and $H^2_{X^\sing}(X,\Om^{n-2}_X)=0.$ The locally trivial deformations of $X$ are then unobstructed.
\end{thm}

For $n=2$ the singularities are rational double points (whose deformations are unobstructed) and Theorem \ref{main thm1} is proved in \cite{Ran1,Ran2}. The locally trivial deformations are also unobstructed as we recall now. Denote by $\Def(X)$ the Kuranishi space for deformations of $X,$ which is smooth by Theorem \ref{main thm1}. For $x\in X^\sing$ denote by $\Def(X,x)$ the Kuranishi space for deformations of the germ $(X,x),$ which is smooth because $(X,x)$ is a hypersurface singularity. As in \cite{BW} the map $\Def(X)\to \prod_{x\in X^\sing}\Def(X,x)$ defined by taking germs at $X^\sing$ is a submersion. Its fibre over the reference point of $\prod_{x\in X^\sing}\Def(X,x)$ defines therefore a smooth Kuranishi space for locally trivial deformations of $X.$

As in \cite{BGL} every Calabi--Yau conifold $X$ has a quasi-\'etale cover which is the product of a torus, irreducible Calabi--Yau varieties and irreducible holomorphic symplectic varieties. But as $X^\sing$ is isolated the product has only one factor. Let this be an irreducible symplectic variety. Its deformations are then unobstructed \cite[Theorem 2.5]{Nam}. Although $X$ is supposed projective in the formal statement of this result, that hypothesis is unnecessary as is clear from its proof; all we need is \cite[Theorem 1]{Ohs}, which applies to K\"ahler spaces. It is known also that the locally trivial deformations of $X$ are unobstructed \cite[Theorem 4.7]{BL2}.

If $X$ is not symplectic, less is known. By \cite[Corollary 1.5 and Remark 4.5]{FL}, if the singularities of $X$ are complete intersections and satisfy the $1$ Du Bois condition then the deformations of $X$ are unobstructed. The proof shows also that if $X^\sing$ is in addition $n-1$ Du Bois then the locally trivial deformations of $X$ are unobstructed. For instance, if $(X,x)$ is a quasi-homogeneous hypersurface singularity then the higher Du Bois conditions are conditions about its minimal exponent $\al$ (introduced in \cite{Sait}). More precisely, let $\C^{n+1}$ have a $\C^*$ action of weights $w_1,\dots,w_{n+1}\in\{1,2,3,\dots\}$ with greatest common divisor $1$ and let $(X,x)$ be defined in $\C^{n+1}$ by a weighted homogeneous polynomial of degree $d;$ then $\al=(w_1+\cdots+w_{n+1})/d.$ Moreover, for $k=0,1,2,\dots$ the germ $(X,x)$ is $k$ Du Bois if and only if $\al\ge k+1.$ So the major advantage of Theorems \ref{main thm1} and \ref{main thm2} is that we do not need such restrictions. Also we do not need the singularities to be complete intersections either.

We explain now how we prove Theorems \ref{main thm1} and \ref{main thm2}. The formal structure of the proof is similar to that of \cite[Theorem 2.5]{Nam} above. In that theorem we introduce on the regular locus $X^\reg$ a complete K\"ahler metric and show that part of the Hodge spectral sequence of $X^\reg$ degenerates (as in \cite[Theorem 1]{Ohs}). This with the $T^1$ lift theorem implies readily that the deformation functors are unobstructed. 

On the other hand, in the proof of Theorems \ref{main thm1} and \ref{main thm2} we introduce K\"ahler metrics called conifolds metrics in the sense of \cite[Definition 4.6]{Chan}, \cite[Definition 2.2]{HS}, \cite[Definition 2.1]{J1} and \cite[Definition 3.24]{KL}; see also Definition \ref{dfn: Riem conifolds} and Remark \ref{rmk: Riem conifolds}. Such metrics are locally expressible as Riemannian cone metrics, which are incomplete metrics on $X^\reg$ whereas those used in \cite{Ohs} are complete metrics on $X^\reg.$ 

Also we cannot make such a simple statement as the Hodge spectral sequence degeneration \cite[Theorem 1]{Ohs}. Theorems \ref{main thm1} and \ref{main thm2} are concerned with certain maps of the form $\cA\to \cB$ where $\cA,\cB$ are sheaf cohomology groups of infinitesimal deformations of $X.$ The heart of the proofs will be to show that the map $\cA\to \cB$ is surjective. We do this as follows:
\begin{equation}\l{lift}\parbox{10cm}{
We take an element of $\cB$ and represent it by a relative harmonic form on $X^\reg.$ We lift it to a relative differential form which will represent an element of $\cA.$ We then seek a relative harmonic form whose restriction to $\spec\C$ will be the original harmonic form.
}\end{equation}
For more precise statements see Theorem \ref{thm: cH^n-2 intro} and \eq{T1Th} below.

We begin in \S\S\ref{sect: Riem}--\ref{sect: harm n-1} with the study of harmonic forms with respect to K\"ahler conifold metrics. In \S\ref{sect: Riem} we deal with Riemannian cones, which are the model at singularities of the K\"ahler conifold metrics. The main result of the section is as follows:
\begin{thm}[Theorem \ref{thm: no log}]\l{thm: no log intro}
Let $C$ be a Riemannian cone and $\ph$ a harmonic form on $C^\reg.$ Then $\ph$ may be written as an infinite sum \eq{expansion} of homogeneous harmonic forms without logarithm terms. 
\end{thm}
\noindent
We prove this partly because it is itself of interest. It implies the other known results \cite[Remark B.3]{HS}, \cite[Proposition 2.4]{J1} and \cite[Lemma 3.15]{KL} to the effect that no logarithm terms exist. 

In \S\ref{sect: Kahler cones} we deal with K\"ahler cones, which are the cones on compact Sasakian manifolds. In \S\ref{sect: K conifolds} we show that every compact Calabi--Yau conifold has K\"ahler conifold metrics. The main result is Lemma \ref{lem: Kahl pot}. It says that we can glue in the K\"ahler cone metric without making any change at the points far from $X^\sing.$ More precisely, we start with any K\"ahler form on $X$ which is defined at every point including $X^\sing.$ We modify it only near $X^\sing$ by changing the K\"aher potential near $X^\sing.$ For this we use Lemma \ref{lem: Kahl pot}. The result, the conifold metric, is defined only on $X^\reg.$ Here we do not need to think of Ricci-flat K\"ahler metrics as in \cite{HS}. Note that even for $X$ non-singular we do not need to choose Ricci-flat K\"ahler metrics to prove the original statement of Bomologov, Tian and Todorov.

In Proposition \ref{prop: hol 1-forms} we show also that $H^1(X,\O_X)=0$ unless $X^\sing$ is empty. For this we use again the result of \cite{BGL}.

In \S\S\ref{sect: harm}--\ref{sect: harm n-1} we prove the results we shall need about harmonic forms on compact K\"ahler conifolds. The key issue is as follows. Let $\ph$ be a harmonic $(p,q)$ form (which satisfies the Laplace equation $\De\ph=0).$ For $X$ non-singular the integration by parts formula implies immediately that $\d\ph=\d^*\ph=0.$ For $X$ singular however we do not know whether $\d\ph,\d^*\ph$ decay so fast that the integration by parts formula will hold. This is true in the circumstances of Theorem \ref{main thm1} (in Lemma \ref{lem: c harm n-1}). We can verify it by expanding $\d\ph$ into the sum of homogeneous harmonic forms on the K\"ahler cones. The results of \S\ref{sect: Riem} and the Betti number condition imply then the vanishing of those terms which will prevent us from getting the integration by parts formula.

There is also a similar step in the proof of Theorem \ref{main thm2}. The cohomology group $H^1(X,\Th_X)$ relevant to the locally trivial deformations are not directly related to harmonic forms. The condition $H^2_{X^\sing}(X,\Om^{n-2}_X)$ implies that it is isomorphic to a certain vector space of harmonic forms; for more details see Theorem \ref{thm: harm n-1 1}.
 
In \S\ref{sect: deform} we recall the standard algebraic geometry facts we will use. In particular, we give the more precise meaning to the conclusion of Theorem \ref{main thm1}. We recall also the versions we will use of $T^1$ lift theorems. In \S\ref{sect: rel forms} we collect the facts we will use about K\"ahler forms and its relative versions. For this purpose we will work with $\R$-algebras, not with $\C$-algebras as we normally do in deformation theory; for more details see Definitions \ref{dfn: real para}, \ref{dfn: glue} and \ref{dfn: Kahl}. In \S\ref{sect: tens calc} we show that the standard tensor calculus on K\"ahler manifolds extend to their infinitesimal deformations. In \S\ref{sect: C^iy deform} we study the notion of $C^\iy$ deformations we shall need for \eq{lift}.

In \S\ref{sect: proof1} we prove Theorem \ref{main thm1}. The key result may be stated as follows.
\begin{thm}[Theorem \ref{thm: cH^n-2}]\l{thm: cH^n-2 intro}
Let $X$ be a compact K\"ahler $n$-conifold whose singularities are rational and of depth $\ge n,$ satisfying the Betti number condition of Theorem \ref{main thm1}.  Let $A$ be an Artin local $\C$-algebra and $\cX/A$ a deformation of $X.$ The natural map ${}_c H^{n-2}(X^\reg,\Om^1_{\cX/A})\to{}_c H^{n-2}(X^\reg,\Om^1_X)$ is then surjective.
\end{thm}
\noindent
The cohomology groups with left lower index c are defined in Definition \ref{dfn: c}. Theorem \ref{thm: cH^n-2 intro} implies that we can in principle do the same computation of cohomology groups as in \cite[Theorem 2.2]{Gross}. As a result we can apply the $T^1$ lift theorem to deduce Theorem \ref{main thm1}.

In \S\ref{sect: rel harm} and \S\ref{sect: proof2} we prove Theorem \ref{main thm2}. To do so we go back to the little steps in \S\ref{sect: harm n-1 1} and make sure that they generalize to infinitesimal deformations. As a consequence the two sections are rather long. The key result is the following, Theorem \ref{thm: harm n-1 1 version2A}: let $k\ge0$ be an integer and put $A_k:=\C[t]/(t^{k+1}),$ and let $X_k/A_k$ be a deformation of $X;$ then the $A_k$ module $H^1(X,\Th_{X_k/A_k})$ is isomorphic to the space of relative harmonic $(n-1,1)$ forms. We then carry out \eq{lift} and prove that the following holds. 
\begin{equation}\l{T1Th}\parbox{10cm}{
Denote by $X_{k-1}/A_{k-1}$ the deformation of $X$ defined by $\O_{X_{k-1}}:=\O_{X_k}\otimes_{A_k} A_{k-1}.$ The natural map $H^1(X,\Th_{X_k/A_k})\to H^1(X,\Th_{X_{k-1}/A_{k-1}})$ is then surjective. 
}\end{equation}
This with the $T^1$ lift theorem implies Theorem \ref{main thm2}.


In \S\ref{sect: remarks} we make remarks on the hypothesis $H^2_{X^\sing}(X,\Om^{n-2}_X)=0$ of Theorem \ref{main thm2}. This is a condition on the singularities and the ordinary double points for instance do not satisfy it. We show that the affine cone on a toric Fano manifold satisfies it.

The Betti number hypothesis of Theorem \ref{main thm1} will in general be easier to verify. For instance, if $(C,\vx)$ is an ordinary double point then $C^\reg$ is diffeomorphic to $T^*S^n\-S^n,$ the complement to the zero-section, and homotopic to an $S^{n-1}$ bundle over $S^n.$ So it satisfies the condition $b_{n-2}(C^\reg)=0.$ 

The readers interested only in the proof of Theorem \ref{main thm1} can skip \S\ref{sect: harm n-1 1} and \S\ref{sect: tens calc}. They can also keep supposing that $\K=\C$ wherever we treat Artin $\K$-algebras in \S\ref{sect: deform}, \S\ref{sect: rel forms}, \S\ref{sect: C^iy deform} and \S\ref{sect: proof1}; or in other words, they will not need the notion of relative K\"ahler forms with $\K=\R.$

\paragraph{Acknowledgements}
Part of the paper was written when I visited Yoshinori Hashimoto at Osaka Metropolitan University in January and February 2025. I would also like to thank Robert Friedman for pointing out an error in the first version of the paper. During the course of the whole work I was supported financially by the grant 21K13788 of the Japan Society for the Promotion of Science.

\section{Riemannian Cones}\l{sect: Riem}
We begin by defining Riemannian cones.
\begin{dfn}
A {\it Riemannian cone} is the data $(C,\vx, C^\reg, C^\lk,r,g^\lk)$ where $C$ is a metric space, $\vx$ a point of $C,$ $C^\reg$ the subset $C\-\{\vx\}$ which is given a manifold structure, $C^\lk$ a compact manifold without boundary such that there is a diffeomorphism $C^\reg\cong (0,\iy)\times C^\lk$ which we will fix, $r$ the composite of the diffeomorphism $C^\reg\cong (0,\iy)\times C^\lk$ and the projection $(0,\iy)\times C^\lk\to(0,\iy),$ and $g^\lk$ a Riemannian metric on $C^\lk$ such that the metric space structure of $C^\reg$ is induced by the Riemannian metric $\d r^2+r^2 g^\lk.$ We call $\vx$ the {\it vertex,} $C^\lk$ the {\it link,} $r:C^\reg\to(0,\iy)$ the {\it radius function} and $\d r^2+r^2 g^\lk$ the {\it cone metric.} 

We call $C$ a Riemannian {\it $l$-cone} if $C^\reg$ is a manifold of dimension $l,$ which is thus the real dimension. 
\end{dfn}

We define homogeneous $p$-forms and harmonic $p$-forms on Riemannian cones.
\begin{dfn}
Let $C$ be a Riemannian cone and $p$ an integer. Denote by $\La^p_{C^\reg}$ the sheaf on $C^\reg$ of $C^\iy$ $p$-forms with complex coefficients. We say that $\ph\in \Ga(\La^p_{C^\reg})$ is {\it homogeneous of order $\al\in\C$} if $\ph=e^{(\al+p)\log r}(\d\log r\w\ph'+\ph'')$ where $r$ is the radius function on $C^\reg,$ $\ph'$ some $p-1$ form on $C^\lk,$ and $\ph''$ some $p$-form on $C^\lk.$

We say that $\ph\in \Ga(\La^p_{C^\reg})$ is {\it harmonic} if $\De\ph=0$ where $\De$ is computed with respect to the cone metric of $C^\reg.$
\end{dfn}

We compute $\d,\d^*$ and $\De$ on Riemannian cones.
\begin{prop}\l{prop: Lap}
Let $C$ be a Riemannian $l$-cone and $r:C^\reg\to(0,\iy)$ its radius function. Denote by $\pi:C^\reg\cong (0,\iy)\times C^\lk\to C^\lk$ the projection onto the second component. Define for $p\in\Z$ a $\C$-vector space isomorphism $\Ga(\La^p_{C^\reg})\cong\Ga(\pi^*\La^{p-1}_{C^\lk})\oplus\Ga(\pi^*\La^p_{C^\lk})$ by writing each $\ph\in \Ga(\La^p_{C^\reg})$ as $\d\log r\wedge\ph'+\ph''$ for some $\ph'\in \Ga(\pi^*\La^{p-1}_{C^\reg})$ and $\ph''\in \Ga(\pi^*\La^p_{C^\reg}).$ Using these isomorphisms write the de Rham differential as $\d: \Ga(\pi^*\La^{p-1}_{C^\lk})\oplus \Ga(\pi^*\La^p_{C^\lk})\to 
\Ga(\pi^*\La^p_{C^\lk})\oplus \Ga(\pi^*\La^{p+1}_{C^\lk}).$ Using the cone metric of $C^\reg$ define $\d^*$ and $\De,$ and write them as $\d^*: \Ga(\pi^*\La^{p-1}_{C^\lk})\oplus \Ga(\pi^*\La^p_{C^\lk})\to 
\Ga(\pi^*\La^{p-2}_{C^\lk})\oplus \Ga(\pi^*\La^{p-1}_{C^\lk})$ and 
$\De: \Ga(\pi^*\La^{p-1}_{C^\lk})\oplus \Ga(\pi^*\La^p_{C^\lk})\to \Ga(\pi^*\La^{p-1}C^\lk)\oplus \Ga(\pi^*\La^p_{C^\lk})$ respectively. These may then be expressed as matrices as follows:
$\d=\begin{pmatrix}-\d& r\frac{\bd}{\bd r} \\ 0& \d \end{pmatrix},$
$r^2\d^*=\begin{pmatrix}-\d^*& 0\\ -r\frac{\bd}{\bd r}+2p-l & \d^* \end{pmatrix}$ and
$r^2\De=\begin{pmatrix}-(r\frac{\bd}{\bd r})^2 &0 \\ 0& -(r\frac{\bd}{\bd r})^2\end{pmatrix}
+(2+2p-l)\begin{pmatrix}r\frac{\bd}{\bd r} &0 \\ 0&r\frac{\bd}{\bd r}\end{pmatrix}+\begin{pmatrix}\De+2l-4p& -2\d^*\\ -2\d& \De \end{pmatrix}$ where $\d,\d^*$ and $\De$ are computed on $C^\reg$ on the left-hand sides and on $C^\lk$ on the right-hand sides.
In particular, if $\ph=r^\be(\d\log r\w\ph'+\ph'')$ is a homogeneous $p$-form of order $\be-p\in\C$ on $C^\reg$ then 
\ea\l{d+d*}
\d\ph=&\be r^{\be-1}\d r \w\ph''+r^\be \d\ph''-r^{\be-1}\d r\w \d\ph',\\
\d^*\ph=&r^{\be-2} \d^*\ph''-(\be+l-2p)r^{\be-2}\ph'-r^{\be-3} \d r\w \d^*\ph',\\
\De\ph=&r^{\be-2}\d\log r\wedge 
[\De\ph'-(\be-2)(\be+l-2p)\ph'-2\d^*\ph'']\\
&+r^{\be-2}[\De\ph''-\be(\be+l-2-2p)\ph''-2\d\ph']
\ea
where again $\d,\d^*$ and $\De$ are computed on $C^\reg$ on the left-hand sides and on $C^\lk$ on the right-hand sides. 
\end{prop}
\begin{proof}
These are the results of straightforward computation. The details about \eq{d+d*} for $l$ even are given for instance by Chan \cite[Proposition 3.3]{Chan2}. His computation applies to every $l$ and implies also the matrix expressions above.
\end{proof}

We study homogeneous harmonic forms on Riemannian cones.
\begin{prop}\l{prop: real eigen}
Let $C$ be a Riemannian $l$-cone and $p$ an integer. Denote by $\cD\sb\C$ the set of $\al$ for which there exists a non-zero $\ph\in\Ga(\La^p_{C^\reg})$ homogeneous of order $\al$ and satisfying $\De\ph=0$ with respect to the cone metric of $C^\reg.$ Then $\cD\sb\R$ and $\cD$ is discrete.
\end{prop}
\begin{proof}
Take $\al\in\cD$ and let $\ph\in\Ga(\La^p_{\C^\reg})$ be non-zero, homogeneous of order $\al$ and with $\De\ph=0.$ Put $\be:=\al+p$ and use the notation of Proposition \ref{prop: Lap}. The equation \eq{d+d*} implies then
\begin{align}
\l{p2}\De\ph'&=(\be-2)(\be+l-2p)\ph'+2\d^*\ph'',\\
\l{p1}\De\ph''&=\be(\be+l-2-2p)\ph''+2\d\ph'.
\end{align}
Applying $\d$ to \eq{p1} we find that $\d\d^*\d\ph''=\be(\be+l-2-2p)\d\ph''.$ So if $\d\ph''\ne0$ then $\be(\be+l-2-2p)$ is an eigenvalue of the Laplacian, which implies that $\al=\be-p$ lies in a discrete subset of $\R$ independent of $\ph.$ Suppose therefore that $\d\ph''=0.$ Put $\ps:=\d\ph'$ so that \eq{p1} becomes
\e\l{p11}
\De\ph''=\be(\be+l-2-2p)\ph''+2\ps.
\e
Applying $\d$ to \eq{p2} and using $\d\ph'=\ps,$ $\d\ph''=0$ and \eq{p11} we find that
\ea\l{p12}
\De\ps&=(\be-2)(\be+l-2p)\ps+2\De\ph''\\
&=(\be-2)(\be+l-2p)\ps+2\be(\be+l-2-2p)\ph''+4\ps\\
&=2\be(\be+l-2-2p)\ph''+[(\be-2)(\be+l-2p)+4]\ps.
\ea
This and \eq{p11} imply $(\De\ph'',\De\ps)=(\ph'',\ps)M$ where
\e\l{dfn M}
M:=\begin{pmatrix} \be(\be+l-2-2p)& 2\be(\be+l-2-2p)\\
2 & (\be-2)(\be+l-2p)+4
\end{pmatrix}.
\e
This matrix is diagonalizable; and in fact, $P^{-1}MP=D$
where
\e\l{dfn P and D}
P=\begin{pmatrix} \be+l-2-2p& \be\\ 1& -1\end{pmatrix}
\text{ and }
D:=\begin{pmatrix} \be(\be+l-2p)& 0\\ 0 & (\be-2)(\be+l-2-2p)
\end{pmatrix}.
\e
So
$(\De\ph'',\De\ps)P=(\ph'',\ps)PD$ and looking at the first component we see that
\e\l{phps}
\De[(\be+l-2-2p)\ph''+\ps]=\be(\be+l-2p)[(\be+l-2-2p)\ph''+\ps].
\e
Thus if $(\be+l-2-2p)\ph''+\ps\ne0$ then $\be(\be+l-2p)$ is an eigenvalue of the Laplacian, which implies that $\al=\be-p$ lies in a discrete subset of $\R$ independent of $\ph.$ Suppose therefore that $(\be+l-2-2p)\ph''+\ps=0.$ Then by \eq{p1} we have 
\e \De\ph''= \be(\be+l-2-2p)\ph''-2(\be+l-2-2p)\ph''=(\be-2)(\be-2+l-2p)\ph''.\e
So if $\ph''\ne0$ then $(\be-2)(\be-2+l-2p)$ is an eigenvalue of the Laplacian, which implies that $\al=\be-p$ lies in a discrete subset of $\R$ independent of $\ph.$ Suppose therefore that $\ph''=0.$ Then by \eq{p2} we have $\De\ph'=(\be-2)(\be+l-2p)\ph'.$ But by hypothesis $\ph\ne0.$ So $(\be-2)(\be-2+l-2p)$ is an eigenvalue of the Laplacian, which implies that $\al=\be-p$ lies in a discrete subset of $\R$ independent of $\ph.$ This completes the proof.
\end{proof}

From the computation above we get the following three corollaries.
\begin{cor}\l{cor: harmonic-1}
Let $C$ be a Riemannian $l$-cone and $p>\frac l2$ an integer. Let $\ph\in\Ga(\La^p_{C^\reg})$ be homogeneous of order $\al\in (-p,p-l)$ and satisfy $\De\ph=0$ with respect to the cone metric of $C^\reg.$ Write $\ph=r^{p+\al}(\d\log r\w\ph'+\ph'')$ as in Proposition \ref{prop: Lap}. Then $\d\ph'=(2+p-l-\al)\ph''.$
\end{cor}
\begin{proof}
Put $\be:=p+\al\in (0,2p-l)$ and follow the proof of Proposition \ref{prop: real eigen}. Applying again $\d$ to \eq{p1} we find $\d\d^*\d\ph''=\be(\be+l-2-2p)\d\ph''.$ But now $\be(\be+l-2-2p)<0$ so $\d\ph''=0.$ Put again $\ps:=\d\ph'.$ Then \eq{phps} holds; that is, 
\e\De[(\be+l-2-2p)\ph''+\ps]=\be(\be+l-2p)[(\be+l-2-2p)\ph''+\ps].
\e
But now $\be(\be+l-2p)<0$ so $(\be+l-2-2p)\ph''+\ps=0$ as we have to prove.
\end{proof}

\begin{cor}\l{cor: harmonic0}
Let $C$ be a Riemannian $l$-cone and $p>\frac l2+1$ an integer. Let $\ph\in\Ga(\La^p_{C^\reg})$ be homogeneous of order $\al\in(2-p,p-l)$ and satisfy $\De\ph=0$ with respect to the cone metric of $C^\reg.$ Then $\ph=0.$
\end{cor}
\begin{proof}
As $p,\al$ satisfy the hypotheses of Corollary \ref{cor: harmonic-1} we can use its result; that is, 
writing again $\ph=r^\be(\ph''+\d\log r\wedge\ph')$ we have $\d\ph'=(2+p-\al-l)\ph''=(2+2p-\be-l)\ph''$ with $\be:=p+\al\in(2,2p-l).$
Equation \eq{p11} holds too with $\ps:=\d\ph'$ and
\[
\De\ph''
=\be(\be+l-2-2p)\ph''-2(\be+l-2-2p)\ph''
=(\be-2)(\be+l-2-2p)\ph''.
\]
But $\be\in(2,2p-l)$ and $(\be-2)(\be+l-2-2p)<0$ so $\ph''=0.$
Equation \eq{p2} implies then $\De \ph'=(\be-2)(\be+l-2p)\ph'.$
But again  $\be\in(2,2p-l)$ so  $\ph'=0.$ Thus $\ph=0.$
\end{proof}
\begin{cor}\l{cor: harmonic1}
Let $C$ be a Riemannian $l$-cone and $p<\frac l2-1$ an integer. Let $\ph\in\Ga(\La^p_{C^\reg})$ be homogeneous of order $\al\in(2+p-l,-p)$ and satisfy $\De\ph=0$ with respect to the cone metric of $C^\reg.$ Then $\ph=0.$
\end{cor}
\begin{proof}
Put $q:=l-p.$ Then $\al\in (2-q,q-l).$ Suppose first that $C^\reg$ is orientable. Then we can define the Hodge dual $*\ph$ as a homogeneous harmonic $q$-form of order $\al,$ to which we can apply Corollary \ref{cor: harmonic0}. So $*\ph=0$ and $\ph=0.$
If $C^\reg$ is unorientable then the result we have just obtained applies to the pull-back of $\ph$ to the double cover of $C^\reg;$ that is, the pull-back vanishes and accordingly so does $\ph.$
\end{proof}

The following may be proved by the separation of variable method; see for instance \cite[Part I, Equation 5.8]{Sim}.
\begin{prop}\l{prop: sep of vars}
Let $C$ be a Riemannian $l$-cone and $\pi:C^\reg\cong (0,\iy)\times C^\lk\to C^\lk$ the projection onto the second component. Let $V$ be a finite-rank $C^\iy$ complex vector bundle over $C^\lk,$ equipped with a Hermitian metric. Let $E:C^\iy(V)\to C^\iy(V)$ be a self-adjoint second-order linear elliptic operator with eigenvalues $\la_0\le \la_1\le \la_2\le\cdots$ which, as is well known, tend to $\iy.$ Let $(e_j)_{j=0}^\iy$ be a complete orthonormal system of $L^2(V)$ where each $e_j$ is an eigenvector of $E$ with eigenvalue $\la_j.$ Fix $m\in\R$ and consider the operator $(r\frac\bd{\bd r})^2-2m r\frac\bd{\bd r}-E:C^\iy(\pi^*V)\to C^\iy(\pi^*V).$  Let this be an elliptic operator. For $j\in\{0,1,2,\dots\}$ with $\la_j\ne -m^2$ denote by $\al_j,\be_j\in\C$ the two distinct roots of the polynomial $\xi^2-2m\xi-\la_j\in\R[\xi].$ Let $u\in C^\iy(\pi^*V)$ satisfy the equation $[(r\frac\bd{\bd r})^2-2m r\frac\bd{\bd r}-E]u=0.$ Then there exist two sequences $(a_j)_{j=0}^\iy,(b_j)_{j=0}^\iy$ of complex numbers such that
\e\l{expansion0} u=\sum_{\la_j\ne -m^2} (a_je^{\al_j\log r}+b_je^{\be_j\log r}) e_j+\sum_{\la_j=-m^2}(a_j+b_j\log r)r^me_j\e
which converges in the compact $C^\iy$ sense. The same result holds also for $u$ defined only on some open set in $C^\reg.$ \qed
\end{prop}

Applying Proposition \ref{prop: sep of vars} to the $p$-form Laplacian, we prove
\begin{cor}\l{cor: sep of vars}
Let $C$ be a Riemannian $l$-cone, fix $p\in\Z$ and put $m:=1+p-\frac l2.$ Define a self-adjoint elliptic operator $E:\Ga(\pi^*\La^{p-1}_{C^\lk})\oplus \Ga(\pi^*\La^p_{C^\lk})\to \Ga(\pi^*\La^{p-1}_{C^\lk})\oplus \Ga(\La^p_{C^\lk})$ by $E:=\begin{pmatrix}\De+2l-4p& -2\d^*\\ -2\d& \De \end{pmatrix}$ where $\d,\d^*$ and $\De$ are computed on $C^\lk.$ Let $(e_j)_{j=0}^\iy$ be a complete orthonormal system of $L^2(\pi^*\La^{p-1}_{C^\lk})\oplus L^2(\pi^*\La^p_{C^\lk})$ which consists of eigenvectors of $E$ with eigenvalues $\la_0\le \la_1\le \la_2\le\cdots$ tending to $\iy.$ Then for every $j=0,1,2,\dots$ we have $\la_j\ge-m^2.$ Moreover the following holds. 

For $j\in\{0,1,2,\dots\}$ with $\la_j>-m^2$ denote by $\al_j>\be_j$ the two distinct real roots of the polynomial $\xi^2-2m\xi-\la_j\in\R[\xi].$ Let $\ph$ be a section of $\La^p_{C^\reg}$ over some open set of $C^\reg,$ satisfying $\De\ph=0$ with respect to the cone metric. Then there exist two sequences $(a_j)_{\la_j\ge {-m^2}},(b_j)_{\la_j\ge {-m^2}}$ of complex numbers such that
\e\l{expansion} \ph=\sum_{\la_j> -m^2} (a_jr^{\al_j}+b_jr^{\be_j}) e_j+\sum_{\la_j=-m^2}(a_j+b_j\log r)r^me_j\e
which converges in the compact $C^\iy$ sense.
\end{cor}
\begin{proof}
Proposition \ref{prop: Lap} implies $-r^2\De=(r\frac{\bd}{\bd r})^2-2mr\frac{\bd}{\bd r}-E,$ to which we can certainly apply Proposition \ref{prop: sep of vars}.
Notice that for any $j$ with $\la_j<-m^2$ the two distinct real roots of the polynomial $\xi^2-2m\xi-\la_j\in\R[\xi]$ are not real numbers. Proposition \ref{prop: real eigen} implies therefore that no such $j$ exists. So $\la_j\ge-m^2$ for every $j.$ The latter part is an immediate consequence of  Proposition \ref{prop: sep of vars}.
\end{proof}

We prove that no logarithm terms appear in \eq{expansion}. 
\begin{thm}\l{thm: no log}
In the circumstances of Corollary \ref{cor: sep of vars} no $j\in\{0,1,2,\dots\}$ is such that $\la_j=-m^2;$ in particular, \eq{expansion} becomes a sum of homogeneous harmonic $p$-forms.
\end{thm}

\begin{proof}
Suppose contrarily that there exists some $j$ with $\la_j=-m^2.$ Putting $\ph:=r^m e_j$ we have then $\De\ph=\De[(\log r)\ph]=0$ where $\De$ is computed on $C^\reg.$ Direct computation shows that
\e\l{ps1}
\ts(\d+\d^*)[(\log r)\ph]=\frac1r\d r\wedge\ph-\frac1r\frac{\bd}{\bd r}\mathbin{\lrcorner}\ph+(\log r)(\d+\d^*)\ph
\e
where $\d,\d^*$ are computed on $C^\reg.$ Applying $\d+\d^*$ to these and looking at the degree-$p$ parts, we get
\e\l{ps2}
0=\ts\De[(\log r)\ph]
=\d^*(\frac1r\d r\wedge\ph)-\d(\frac1r\frac{\bd}{\bd r}\mathbin{\lrcorner}\ph)+(\d+\d^*)[(\log r)(\d+\d^*)\ph].\e
Applying \eq{ps1} to $(\d+\d^*)\ph$ in place of $\ph,$ using the equation $\De\ph=0$ and looking at the degree-$p$ parts, we get
\ea\l{ps3}
(\d+\d^*)[\log r(\d+\d^*)\ph]&=\ts\frac1r\d r\wedge(\d+\d^*)\ph-\frac1r\frac{\bd}{\bd r}\mathbin{\lrcorner}(\d+\d^*)\ph\\
&=\ts\frac1r\d r\wedge\d^*\ph-\frac1r\frac{\bd}{\bd r}\mathbin{\lrcorner}\d\ph.
\ea
By \eq{ps2} and \eq{ps3} we have
\e\l{D0}
\ts
\d^*(\frac1r\d r\wedge\ph)-\d(\frac1r\frac{\bd}{\bd r}\mathbin{\lrcorner}\ph)
+\frac1r\d r\wedge\d^*\ph-\frac1r\frac{\bd}{\bd r}\mathbin{\lrcorner}\d\ph=0.
\e
Write
$\ph:=r^m(\d\log r\wedge \ph'+\ph'')$
where $\ph'$ is a $(p-1)$ form on $C^\lk$ and $\ph''$ a $p$-form on $C^\lk.$ Using \eq{d+d*} with $\al=m$ we see that the first term $\d^*(\frac1r\d r\wedge\ph)=\d^*(r^m\d\log r\wedge\ph'')$ on the left-hand side of \eq{D0} is equal to 
\e\l{Th2}
(m-2)r^{m-2}\ph''-r^{m-3}\d r\wedge\d^*\ph''
\e
where $\d^*$ is computed on $C^\lk.$ The second term $-\d(\frac1r\bd_r\mathbin{\lrcorner}\ph)=-\d(r^{m-2}\ph')$ on the left-hand side of \eq{D0} is equal to
\e\l{Th3}
-(m-2)r^{m-3}\d r\wedge\ph'-r^{m-2}\d\ph'.
\e
Using \eq{d+d*} with $\al=m$ we see that the third term on the left-hand side of \eq{D0} is equal to
\e\l{Th4}
\ts\frac1r\d r\wedge \d^*\ph=\d r\wedge[r^{m-3}\d^*\ph''+(m-2)r^{m-3}\ph'].
\e
where $\d^*\ph''$ is computed on $C^\lk.$ The fouth term on the left-hand side of \eq{D0} is equal to
\e\l{Th5}
\ts
-\frac 1r\frac{\bd}{\bd r}\mathbin{\lrcorner}\d\ph=-m r^{m-2}\ph''+r^{m-2}\d\ph'
\e
All the terms on \eq{Th2}--\eq{Th5} cancel out except the first term on \eq{Th2} and the first term on \eq{Th5}. So \eq{D0} becomes
$-2r^{m-2}\ph''=0;$ that is, $\ph''=0.$ Using \eq{d+d*} with $\al=m$ we see now that $\De\ph'+(m-2)^2\ph'=0$ where $\De$ is computed on $C^\lk.$ But $\ph\ne0$ implies $\ph'\ne0$ so $-(m-2)^2\le0$ is an eigenvalue of $\De,$ which must therefore vanish.
Thus $m=2;$ that is, $p=\frac l2+1.$ 

Notice now that the $(l-p)$ form Laplacian may be written as $(r\frac\bd{\bd r})^2-2\mu r\frac\bd{\bd r}-\cE u$ with $\mu:=1+(l-p)-\frac l2=1-p+\frac l2$ and $\cE$ defined over $C^\lk.$ Computation shows $*\ph=:r^\mu\ep$ with $\ep=(\d\log r\wedge \ep'+\ep'')$ for some $\ep',\ep''$ defined on $C^\lk.$ Since $*\ph$ is harmonic it follows that $\ep$ is an eigenvector of $\cE$ with eigenvalue $-\mu^2.$ So we can apply the result of the paragraph above with $l-p$ in place of $p;$ that is, $\mu=2$ and $l-p=\frac l2+1.$ But this contradicts $p=\frac l2+1,$ completing the proof.
\end{proof}

We prove the following fact about order $-p$ homogeneous harmonic $p$-forms. 
\begin{cor}\l{cor: order -p}
Let $C$ be a Riemannian $l$-cone, $p\le\frac l2-1$ an integer, and $\ph$ an order $-p$ homogeneous harmonic $p$-form on $C^\reg.$ Then $\d\ph=\d^*\ph=0$ where $\d^*$ is computed with respect to the cone metric on $C^\reg.$
\end{cor}
\begin{proof}
We compute $\De,\d$ and $\d^*$ on $C^\lk$ except at the end of the proof. 
By \eq{d+d*} with $\be=0$ we have 
\e\l{Deph}
\De\ph'=-2(l-2p)\ph'+2\d^*\ph''\text{ and }\De\ph''=2\d\ph'
\e
Applying $\d^*$ to the first equation, we find $\d^*\De\ph'=-2(l-2p)\d^*\ph'.$ The left-hand side is equal to $\De\d^*\ph'$ and so $\d^*\ph'$ is an eigenvector with eigenvalue $-2(l-2p)<0$ of $\De.$ Thus $\d^*\ph'=0.$ 

Applying $\d$ to the second equation of \eq{Deph} we find $\d\De\ph''=0.$ So $\d^*\d\d^*\d\ph''=0.$ Doing twice the integration by parts, we get $\d\ph''=0.$ Now \eq{Deph} becomes
\e\l{Deph2}
\d^*\d\ph'=-2(l-2p)\ph'+2\d^*\ph''\text{ and }\d\d^*\ph''=2\d\ph'.
\e
Applying $\d^*$ to the second equation and using then the first equation, we find
\e
\d^*\d\d^*\ph''=2\d^*\d\ph'=-4(l-2p)\ph'+4\d^*\ph''.
\e
Applying $\d$ to both sides and using the second equation of \eq{Deph2} we get
\ea\l{dd*ph''}
\d\d^*\d\d^*\ph''=2\d^*\d\ph'=-4(l-2p)\d\ph'+4\d\d^*\ph''\\
=-2(l-2p)\d\d^*\ph''+4\d\d^*\ph''=-2(l-2p-2)\d\d^*\ph''.
\ea
As the left-hand side is equal to $\De\d\d^*\ph''$ the vector $\d\d^*\ph''$ is an eigenvector with eigenvalue $-2(l-2p-2)\le0$ of $\De.$ If $p<\frac l2-1$ then $\d\d^*\ph''=0.$ If $p=\frac l2-1$ then \eq{dd*ph''} implies $\d\d^*\d\d^*\ph''=0$ and hence it follows by integration by parts that even in this case we have $\d\d^*\ph''=0.$ Doing again integration by parts we find $\d^*\ph''=0.$ The first equation of \eq{Deph2} implies in turn that $\d^*\d\ph'=-2(l-2p)\ph'.$ As $\d^*\ph'=0$ the left-hand side is equal to $\De\ph',$  so $\ph'$ is an eigenvector with negative eigenvalue of $\De.$ Thus $\ph'=0.$ Now \eq{d+d*} with $\be=0$ implies $\d\ph=\d^*\ph=0$ where $\d,\d^*$ are computed on $C^\reg.$
\end{proof}

In the same way we prove
\begin{cor}\l{cor: order 2+p-l}
Let $C$ be a Riemannian $l$-cone and $p\le\frac l2-1$ an integer with $b_p(C^\reg)=0.$ Let $\ph$ be an order $2+p-l$ homogeneous harmonic $p$-form on $C^\reg.$ Then $\ph=0.$
\end{cor}
\begin{proof}
We compute $\De,\d$ and $\d^*$ on $C^\lk.$ By \eq{d+d*} with $\be=2+p-l$ the equation \eq{Deph} holds as it is. Following the subsequent computation we see that $\ph'=0$ and that $\De\ph''=0.$ But by hypothesis $b_p(C^\lk)=b_p(C^\reg)=0$ so $\ph''=0$ too. Thus $\ph=0.$
\end{proof}

\section{K\"ahler Cones}\l{sect: Kahler cones}
We begin by defining K\"ahler cones.
\begin{dfn}
A {\it K\"ahler cone} is a Riemannian cone $C$ whose regular part $C^\reg$ is given a complex structure $J$ with the following properties: the cone metric of $C^\reg$ is a K\"ahler metric on $(C^\reg, J);$ and for each $t\in(0,\iy),$ if we define a diffeomorphism $(0,\iy)\times C^\lk\to (0,\iy)\times C^\lk$ by $(a,b)\mapsto(ta,b)$ for $(a,b)\in(0,\iy)\times C^\lk$ then the corresponding diffeomorphism $C^\reg\to C^\reg$ is holomorphic.

We call $(C,J)$ a K\"ahler $n$-cone if $(C^\reg,J)$ is a complex manifold of complex dimension $n.$
\end{dfn}
It is known that the complex structure of a K\"ahler cone extends automatically to its vertex. We will recall this shortly after making a definition we shall need.
\begin{dfn}\l{dfn: quasi-hom}
We say that the germ $(Y,y)$ of a complex analytic space is {\it quasi-homogeneous} if there exist integers $k;w_1,\dots,w_k\ge1$ and a complex analytic embedding $(Y,y)\sb(\C^k,0)$ such that $(Y,y)$ is invariant under the multiplicative group action $\C^*:=\C\-\{0\}\curvearrowright\C^k$ defined by $t\cdot(z_1,\dots,z_k)=(t^{w_1}z_1,\dots,t^{w_k}z_k).$
\end{dfn}
\begin{thm}[Theorem 3.1 of \cite{vC}]\l{thm: vC}
Every K\"ahler cone $C$ has the structure of a normal complex space which agrees with the complex manifold structure of $C^\reg$ and whose germ $(C,\vx)$ is quasi-homogeneous. \qed
\end{thm}

\section{Compact Conifolds}\l{sect: K conifolds}
We begin by recalling the definition of Sasakian manifolds. 
\begin{dfn}
Let $n\ge1$ be an integer and $M$ a manifold of dimension $2n-1.$ A {\it contact} form on $M$ is a $1$-form $\et\in C^\iy(T^*M)$ such that the $2n-1$ form $\et\wedge (\d\et)^{n-1}$ is nowhere vanishing. Corresponding to this $\et$ there exists a unique $\xi\in C^\iy(TM)$ with $\et(\xi)=1$ and $\xi\mathbin{\lrcorner}\d\et=0,$ called the {\it Reeb} vector field of $(M,\et).$

A {\it Sasakian structure} on $M$ is the pair of a contact form $\et$ and a section $\Ph\in C^\iy(\End TM)$ such that if we denote by $\xi$ the Reeb vector field of $(M,\et)$ then the following hold: $\Ph\xi=0\in C^\iy(TM);$ $\Ph$ maps the sub-bundle $\ker\et\subset TM$ to itself, defining a compatible almost complex structure upon the symplectic vector bundle $(\ker\th,\d\et);$ and for $u,v\in C^\iy(TM)$ we have
\e\l{Ph} [\Ph u,\Ph v]+\Ph^2[u,v]-\Ph[\Ph u,v]-\Ph[u,\Ph v]=-2\d\et(u,v)\xi\in C^\iy(TM).\e
The data $(M;\et,\Ph)$ is called a {\it Sasakian manifold.} Its {\it Sasakian metric} is a Riemannian metric $g$ on $M$ defined by $g(u,v):=\et(u)\et(v)+\d\et(u,\Ph v)$ for $u,v\in C^\iy(TM).$ 
\end{dfn}
\begin{rmk}
If $(M;\et,\Ph)$ is a Sasakian manifold then $(0,\iy)\times M$ has an almost complex structure $J$ with $J(r\frac{\bd}{\bd r})=\xi,$ $J\xi=-r\frac{\bd}{\bd r}$ and $J|_{\ker\et}=\Ph|_{\ker\et}.$ The equation \eq{Ph} implies that $J$ is integrable; for the proof see for instance \cite[Theorem 6.5.9]{BG}. On the other hand, using the projection $r:(0,\iy)\times M\to(0,\iy)$ we can make $(0,\iy)\times M$ into a Riemannian cone. This with $J$ defines a K\"ahler cone.

Conversely, if $C$ is a K\"ahler cone with complex structure $J$ on $C^\reg$ then $C^\lk$ has a contact form $\et:=-(J\d r)|_{\{1\}\times C^\lk}.$ Denote by $\xi$ its Reeb vector field, and define $\Ph\in C^\iy(\End TC^\lk)$ by $\Ph\xi:=0$ and $\Ph|_{\ker\et}=J|_{\ker\et}.$ The pair $(\et,\Ph)$ defines then a Sasakian structure on $C^\lk.$ There are thus two-sided operations between Sasakian manifolds and K\"ahler cones, which are inverses to each other. 
\end{rmk}

We recall the facts we will use about deformations of Sasakian structures.
\begin{dfn}\l{dfn: def of Sas}
Let $(M;\et,\Ph)$ be a Sasakian manifold with Reeb vector field $\xi$ and Sasakian metric $g.$ Let $\xi'\in C^\iy(TM)$ be such that $g(\xi,\xi')>0$ at every point of $M$ and $\et,\Ph$ are invariant under the flow of $\xi'.$ Takahashi \cite{Tak} proves then that there exists on $M$ a Sasakian structure $(\et',\Ph')$ defined by $\et':=(\xi'\cdot\et)^{-1}\et$ and $\Ph':=\Ph\cm(\id-\xi'\otimes\et').$ 

Suppose now that $M$ is compact so that we can define the minimum $\al>0$ of $g(\xi,\xi'):M\to(0,\iy)$ and the maximum $\be>0$ of the same function. Denote by $r:(0,\iy)\times M\to(0,\iy)$ the projection. Extend $\xi'$ to the vector field on $(0,\iy)\times M$ invariant under the flow of $r\frac{\db}{\db r};$ and denote by the same $\xi'$ the extended vector field. Denote by $J,J'$ the complex structures on $(0,\iy)\times M$ corresponding respectively to the Sasakian structures $(\et,\Ph),(\et',\Ph').$ Define then a diffeomorphism $F:(0,\iy)\times M\to (0,\iy)\times M$ to be the identity upon $\{1\}\times M$ and equivariant under the flows of $-J\xi'$ on the domain and of $-J'\xi'=r\frac{\db}{\db r}$ on the co-domain. This is possible because $r^\be\le F^*r\le r^\al$ wherever $r\le1$ and $r^\al\le F^*r\le r^\be$ wherever $r\ge1;$ for the proof see Conlon and Hein \cite[Proposition II.2]{CH}. They prove also that $F:(0,\iy)\times M\to (0,\iy)\times M$ is a bi-holomorphism with respect to $J,J';$ that is, $F_*J=J'.$
\end{dfn}

\begin{ex}\l{ex: def of Sas}
Fix $m\in\{1,2,3,\dots\}$ and denote by $z_1,\dots,z_m:\C^m\to\C$ the co-ordinate functions. The unit sphere $S^{2m-1}\subset\C^m$ is then defined by the equation $|z_1|^2+\dots+|z_m|^2=1.$
The K\"ahler cone metric $\sum_{a=1}^m\d z_a\otimes \d\bar z_a$ on $\C^m$ induces on $S^{2m-1}$ a Sasakian structure $(\et,\Ph)$ with $\et=\frac i2\sum_{a=1}^m(z_a\d \bar z_a-\bar z_a\d z_a).$ Its Reeb vector field may be written as $\xi:=i\sum_{a=1}^m(z_a\frac{\bd}{\bd z_a}-\bar z_a\frac{\bd}{\bd \bar z_a}).$ For $\la_1,\dots,\la_m>0$ define on $S^{2m-1}$ a vector field $\xi':=i\sum_{a=1}^m\la_a(z_a\frac{\bd}{\bd z_a}-\bar z_a\frac{\bd}{\bd \bar z_a}).$ Then $g(\xi,\xi'):=\sum_{a=1}^m \la_a|z_a|^2>0$ at every point of $S^{2m-1},$ with $\al=\min\{\la_1,\dots,\la_m\}$ and $\be=\max\{\la_1,\dots,\la_m\}$ in the notation of Definition \ref{dfn: def of Sas}. The flow of $\xi'$ may be written as $(t;z_1,\dots,z_m)\mapsto (e^{i\la_1 t}z_1,\dots,e^{i\la_m t}z_m)$ for $t\in\R$ and $(z_1,\dots,z_m)\in \C^m,$ which is a holomorphic isometry of $\C^m$ and so leaves invariant the Sasakian structure $(\et,\Ph).$ We can therefore use Definition \ref{dfn: def of Sas} and define on $S^{2m-1}$ the Sasakian structure $(\et',\Ph')$ corresponding to $\xi'.$ Recall also from Definition \ref{dfn: def of Sas} that there is a bi-holomorphism $F:\C^m\-\{0\}\to (0,\iy)\times S^{2m-1}$ where $\C^m\-\{0\}$ is given the ordinary complex structure $J$ and $(0,\iy)\times S^{2m-1}$ the deformed complex structure $J'$ corresponding to $(\et',\Ph').$ Define $r:\C^m\-\{0\}\to(0,\iy)$ by $r^2:=|z_1|^2+\dots+|z_m|^2$ and define $r_\la:\C^m\-\{0\}\to(0,\iy)$ by $r_\la:=F^*r.$ Then $r^\be\le r_\la\le r^\al$ wherever $r\le 1.$ The K\"ahler form $-\frac14\d (J'\d r^2)$ on $(0,\iy)\times S^{2m-1}$ is pulled back by $F$ to the K\"ahler form $-\frac14\d (J\d r_\la^2)=\frac14\d\d^c r_\la^2$ on $\C^m\-\{0\},$ where $\d^c$ is defined with respect to $J.$

Since $\xi'=i\sum_{a=1}^m\la_a(z_a\frac{\bd}{\bd z_a}-\bar z_a\frac{\bd}{\bd \bar z_a})$ it follows that this vector field extends smoothly to $\C^m.$ The two vector fields $-J\xi',\xi'$ generate then the holomorphic $\C$-action $\C\times\C^m\to\C^m$ which maps $(s;z_1,\dots,z_m)$ to $(e^{\la_1 s}z_1,\dots,e^{\la_m s}z_m).$ Suppose now that $X\sb\C^m$ is a closed normal complex subspace with isolated singularity at $0\in\C^m$ and invariant under this $\C$-action. We show then that $X\cap S^{2m-1}\subset X\-\{0\}$ is a compact submanifold. Restricting the $\C$-action to $\R$ and differentiating this at $0\in\R$ we get a vector field $-J\xi_\la=\sum_{a=1}^m\la_a(z_a\frac{\bd}{\bd z_a}+\bar z_a\frac{\bd}{\bd \bar z_a})$ tangent to $X\-\{0\}.$ Define $f:X\-\{0\}\to(0,\iy)$ by restricting to $X\-\{0\}\subset\C^m\-\{0\}$ the $C^\iy$ function $|z_1|^2+\dots+|z_m|^2:\C^m\-\{0\}\to(0,\iy).$ Then $\d f(\xi_\la)=\sum_{a=1}^m\la_a|z_a|^2\ne 0.$ So $f$ is a submersion and $f^{-1}(1)=X\cap S^{2m-1}$ a submanifold. Recall from the definition of $F$ that the image of $X\-\{0\}$ under $F:\C^m\-\{0\}\to (0,\iy)\times S^{2m-1}$ is $(0,\iy)\times (X\cap S^{2m-1}).$ Since $F$ is holomorphic with respect to $J,J'$ it follows moreover that $(0,\iy)\times (X\cap S^{2m-1})$ is a complex submanifold of $((0,\iy)\times S^{2m-1},J').$ The K\"ahler cone structure of $(0,\iy)\times S^{2m-1}$ induces therefore a K\"ahler cone structure of $(0,\iy)\times (X\cap S^{2m-1}).$ Pulling back this by $F$ we get a K\"ahler cone structure of $X\-\{0\}.$ As is clear from definition its radius function $X\-\{0\}\to(0,\iy)$ is induced from $r_\la:\C^m\-\{0\}\to(0,\iy)$ and its K\"ahler form from $\frac14\d\d^c r_\la^2.$
\end{ex}

We now state and prove the key lemma. Recall that for $(Y,J)$ a complex manifold a $C^\iy$ function $f:Y\to\R$ is {\it strictly plurisubharmonic} if for every $v\in C^\iy(TY)$ we have $\d\d^c f(v,Jv)>0$ at every point of $Y.$ The following then holds.
\begin{lem}\l{lem: Kahl pot}
Fix $m\in\{1,2,3,\dots\}$ and $\la_1,\dots,\la_m\in(0,1).$ Define $r_\la:\C^m\-\{0\}\to(0,\iy)$ as in Example \ref{ex: def of Sas}. Let $U\sb\C^m$ be an open neighbourhood of the origin $0\in\C^m,$ and $p:U\to\R$ a strictly plurisubharmonic $C^\iy$ function with $p(0)=0$ and $\nb p(0)=0\in T^*_0\C^m.$ Then there exist $\ep>0$ and a strictly plurisubharmonic $C^\iy$ function $q:U\-\{0\}\to\R$ which outside some punctured neighbourhood of $0\in U$ agrees with $p$ and on some smaller punctured neighbourhood of $0\in U$ agrees with $\ep r_\la^2.$
\end{lem}
\begin{proof}
We set $q:=p+\ep\ph r_\la^2-\ps(\frac{r^2}{\de^2})p$ where $\ep,\de$ are positive constants; $\ph$ a compactly supported $C^\iy$ function $U\to[0,1]$ which is near $0\in U$ identically equal to one; and $\ps$ a $C^\iy$ function $[0,\iy)\to[0,1]$ supported on $[0,1]\subset[0,\iy)$ and which is near $0\in[0,\iy)$ identically equal to one. Let $\ep$ be independent of $\de$ and so small that $p+\ep \ph r_\la^2:U\-\{0\}\to\R$ is strictly plurisubharmonic. This is possible because the derivatives of $\ph$ are supported on a compact set and may therefore by controlled by making $\ep$ small enough. We prove that $q:U\-\{0\}\to\R$ with $\de$ small enough is strictly plurisubharmonic.

We show first that there exists $M>0$ independent of $\de$ and so large that $\d\d^c [\ps(\frac {r^2}{\de^2})p]\le M\d\d^c r^2$ at every point of the support of $\ph.$ Simple computation shows that at every point of $U$ we have
\ea\l{ddc0}
\!\!\!\!\!\!\!\!\d\d^c\Bigl[\ps\Bigl(\frac {r^2}{\de^2}\Bigr)p\Bigr]=\frac{p}{\de^4}\ps''\Bigl(\frac {r^2}{\de^2}\Bigr)\d r^2\wedge \d^c r^2+ \frac{1}{\de^2}\ps'\Bigl(\frac{r^2}{\de^2}\Bigr)(\d p\wedge\d^c r^2+\d^c p\wedge \d r^2)\\
+\frac{p}{\de^2}\ps'\Bigl(\frac{r^2}{\de^2}\Bigr)\d\d^c r^2+\ps\Bigl(\frac{r^2}{\de^2}\Bigr)\d\d^c p.
\ea
We estimate each term on the right-hand side. Since $p(0)=\nb p(0)=0$ it follows that there exists $M_0>0$ independent of $\de$ and so large that at every point of the support of $\ph$ we have
\e\l{ddc1} |p|\le M_0 r^2,\,\,\, |\d p|\le M_0r,\,\,\, \d\d^c p\le M_0\d\d^c r^2\e 
where $|\d p|$ is the pointwise $\ell^2$ norm with respect to the flat metric $\sum_{a=1}^m \d z_a\otimes \d \bar z_a,$ and the last inequality defined as follows: for $A,B$ two real $(1,1)$ forms on a complex manifold $(Y,J)$ we write $A\le B$ if $A(v.Jv)\le B(v,Jv)$ for every $v\in C^\iy(TY).$ Since $S^{2m-1}$ is compact, $\d r^2\wedge\d^c r^2$ a real $(1,1)$ form on $\C^m$ and $\d\d^c r^2$ a positive definite real $(1,1)$ form on $\C^m$ it follows that there exists $M_1>0$ independent of $\de$ and so large that $\d r^2\wedge\d^c r^2\le M_1\d\d^c r^2=M_1r^2\d\d^c r^2$ at every point of $S^{2m-1}.$ Since $\d r^2\wedge\d^c r^2$ and $r^2\d\d^c r^2$ are both homogeneous of order $4$ (with respect to the flow on $\C^m$ generated by $r\frac{\bd}{\bd r}$) it follows that the same estimate holds everywhere; that is, $\d r^2\wedge\d^c r^2\le M_1r^2\d\d^c r^2$ at every point of $\C^m.$ This and the first estimate of \eq{ddc1} imply that at every point of the support of $\ph$ we have
\e\l{ddc2}
\frac{p}{\de^4}\ps''\Bigl(\frac {r^2}{\de^2}\Bigr)\d r^2\wedge \d^c r^2\le M_0M_1\frac{r^4}{\de^4}
\Bigl|\ps''\Bigl(\frac{r^2}{\de^2}\Bigr)\Bigr|
\d\d^c r^2\le M_0M_1\biggl(\sup_{[0,\iy)}|\ps''|\biggr)\d\d^c r^2
\e
where the last inequality follows since $\ps''$ is supported on $[0,1].$ We estimate now the second term on the right-hand side of \eq{ddc0}. For $A$ a constant real $(1,1)$ form on $\C^m$ define $|A|_{\ell^\iy}:=\max_{v\in \C^m}A(v,Jv)$ and denote by $|A|$ the $\ell^2$ norm with respect to the flat metric $\sum_{a=1}^m \d z_a\otimes \d \bar z_a.$ Then $|A|^2$ is the sum of the squared eigenvalues of $A,$ and $|A|_{\ell^\iy}\le |A|.$ So
\e \l{A}A\le \frac i2|A|_{\ell^\iy}\sum_{a=1}^m \d z_a\wedge \d \bar z_a=\frac14 |A|_{\ell^\iy}\d\d^c r^2\le\frac 14 |A|\d\d^c r^2.\e  
We apply this to $\d p\wedge\d^c r^2+\d^c p\wedge \d r^2$ at every point of $U.$ Using the middle estimate of \eq{ddc1} we find indeed that at every point of the support of $\ph$ we have
\e |\d p\wedge\d^c r^2+\d^c p\wedge \d r^2|\le M_0|\d^c r^2|+M_0r |\d r^2|=2M_0|\d r^2|=4M_0 r.\e
Hence it follows by \eq{A} that at every point of the support of $\ph$ we have
\e (\d p\wedge\d^c r+\d^c p\wedge \d r)\le M_0r\d\d^c r^2\e
So at every point of the support of $\ph$ we have
\e\l{ddc3}
\frac1{\de}\ps'\Bigl(\frac{r^2}{\de^2}\Bigr)(\d p\wedge\d^c r+\d^c p\wedge \d r)\le 
M_0\frac{r}{\de}\Bigl|\ps'\Bigl(\frac{r^2}{\de^2}\Bigr)\Bigr|\d\d^c r^2
\le M_0\biggl(\sup_{[0,\iy)}|\ps'|\biggr)\d\d^c r^2.
\e
We estimate now the remaining the two terms on the right-hand side of \eq{ddc0}. The first estimate of \eq{ddc1} implies that at every point of the support of $\ph$ we have
\e\l{ddc4}
\frac{p}{\de^2}\ps'\Bigl(\frac{r^2}{\de^2}\Bigr)\d\d^c r^2
\le M_0\frac{r^2}{\de^2}\Bigl|\ps'\Bigl(\frac{r^2}{\de^2}\Bigr)\Bigr|\d\d^c r^2
\le M_0\biggl(\sup_{[0,\iy)}|\ps'|\biggr)\d\d^c r^2.
\e
Since $\ps$ has values in $[0,1]$ it follows by the last estimate of \eq{ddc1} that 
\e\l{ddc5}
\ps\Bigl(\frac{r^2}{\de^2}\Bigr)\d\d^c p\le M_0\d\d^c r^2.
\e
Define now $M>0$ by $M:=M_0M_1(\sup_{[0,\iy)}|\ps''|)+2M_0(\sup_{[0,\iy)}|\ps'|)+M_0.$ It follows then from \eq{ddc2}, \eq{ddc3}, \eq{ddc4} and \eq{ddc5} that at every point of the support of $\ph$ we have
$\d\d^c[\ps(\frac {r^2}{\de^2})p]\le M\d\d^c r^2.$

We show next that there exists a punctured neighbourhood of $0\in \C^m$ at every point of which we have $\ep r_\la^2-M r^2$ is strictly plurisubharmonic. Denote by $g_\la$ the deformed Sasakian metric on $S^{2m-1}$ corresponding to $\xi_\la$ as in Example \ref{ex: def of Sas}, and by $g$ the ordinary Sakasian metric on $S^{2m-1}$ (that is, the round sphere metric). For $h,h'$ two Riemannian metrics on a manifold $Y$ write $h\ge h'$ if $h(v,v)\ge h'(v,v)$ for every $v\in C^\iy(TY).$ Let $\nu>0$ be so small that $g_\la\ge \nu g_{S^{2m-1}}$ at every point of $S^{2m-1}.$ Since $\log r_\la \ge \be\log r$ with $\be=\max\{\la_1,\dots,\la_m\}\in(0,1)$ it follows then that
\[r_\la^2[(\d\log r_\la)^{\otimes2}+g_\la]\ge r^{2\be}[\be^2(\d\log r)^{\otimes2}+\nu g_{S^{2m-1}}]
\ge \min\{\be^2,\nu\} r^{2(1-\be)}\sum_{a=1}^m\d z_a\otimes \d \bar z_a\]
at every point of $\C^m\-\{0\}.$ The corresponding $(1,1)$ forms satisfy the estimate $\d\d^c r_\la^2\ge  \min\{\be^2,\nu\} r^{2(1-\be)}\d\d^c r^2$ at every point of $\C^m\-\{0\}.$ As $\ep$ and $M$ are independent of $\de$ we can make $\de$ so small that $\min\{\be^2,\nu\} \de^{2(1-\be)}\ge \ep^{-1}M.$ It follows then that at every point of $\C^m\-\{0\}$ at which $r\le\de,$ we have $\ep\d\d^c r_\la^2\ge M\d\d^c r^2.$

Let $\de$ be so small too that $\ph=1$ at those points of $U\-\{0\}$ at which $r\le\de.$ Then $q=\ep r_\la^2$ at the same points; and accordingly, since $\ep\d\d^c r_\la^2\ge M\d\d^c r^2$ at these points it follows that $q$ is strictly plurisubharmonic at them. On the other hand, at the points of $U\-\{0\}$ with $r>\de$ we have $\ps(\frac r\de)=0$ and know already that $q=p+\ep\ph r_\la^2$ is strictly plurisubharmonic at these points. Thus $q:U\-\{0\}\to\R$ is everywhere strictly plurisubharmonic. Choose finally a punctured neighbourhood of $0\in U$ on which $\ps(\frac r\de)=1.$ On this set we have certainly $q=\ep r_\la^2,$ which completes the proof.
\end{proof}
We prove a corollary of Lemma \ref{lem: Kahl pot}.
\begin{cor}\l{cor: Kahl pot}
Fix $m\in\{1,2,3,\dots\}$ and $\la_1,\dots,\la_m\in(0,1).$ Define $r_\la:\C^m\-\{0\}\to(0,\iy)$ as in Example \ref{ex: def of Sas}. Let $U\sb\C^m$ be an open neighbourhood of the origin $0\in\C^m,$ and $\om$ a K\"ahler form on $U.$ Then there exist $\ep>0$ and a K\"ahler form on $U\-\{0\}$ which outside some punctured neighbourhood of $0\in U$ agrees with $\om$ and on some smaller punctured neighbourhood of $0\in U$ agrees with $\ep\d\d^c r_\la^2.$
\end{cor}
\begin{proof}
The local $\bd\db$ lemma implies that there exist an open neighbourhood $V$ of $0\in U$ and a smooth function $f:V\to\R$ such that $\om|_V=\d\d^c f.$ Define $p:V\to\R$ by $p:=f-f(0)-\sum_{a=1}^m(\frac{\bd f}{\bd z_a}(0)z_a+\frac{\bd f}{\bd \bar z_a}(0)\bar z_a).$ Since $\d\d^c z_a=\d\d^c \bar z_a=0$ for $a=1,\dots,m$ it follows then that $\om|_V=\d\d^c p,$ to which we can apply Lemma \ref{lem: Kahl pot} with $V$ in place of $U.$ Let $q:V\-\{0\}\to\R$ be the result of this; then $\d\d^c q$ is a K\"ahler form we want.
\end{proof}

We recall now the definition of K\"ahler complex spaces.
\begin{dfn}\l{dfn: Kahl sp}
For a complex space $X$ we say that a $C^\iy$ function $\ph:X\to\R$ is {\it strictly plurisubharmonic} if every point of $X$ has an open neighbourhood $U$ embedded in some open set $Y\sb\C^n$ for which there exists a strictly plurisubharmonic function $\ps:Y\to\R$ with $\ps|_U=\ph|_U.$ We call $X$ a {\it K\"ahler space} if there exist an open cover $U\cup V\cup\dots=X$ and a corresponding family $(\ph_U:U\to\R)_U$ of $C^\iy$ strictly plurisubharmonic functions such that for each $U$ we have $\om|_U=i\bd\db \ph_U.$ We call $\ph_U,\ph_V,\dots$ {\it K\"ahler potentials} of $X.$
\end{dfn}
\begin{rmk}
When we speak simply of a K\"ahler space $X$ we do not make any particular choice of the family $\ph_U,\ph_V,\dots$ of K\"ahler potentials. This convention is compatible with the statement of our main results, Theorems \ref{main thm1} and \ref{main thm2}, which themselves have nothing to do with the choice of K\"ahler potentials.
\end{rmk}

We make the definition we will use of compact K\"ahler conifolds. 
\begin{dfn}\l{dfn: Kahler conifolds}
A {\it compact K\"ahler conifold} is a compact normal K\"ahler space whose singularities are isolated and quasi-homogeneous. We call $X$ a K\"ahler $n$-conifold if it has (complex) dimension $n.$
\end{dfn}

\begin{lem}\l{lem: conifold metrics}
Let $X$ be a compact K\"ahler conifold. Denote by $X^\reg$ its regular locus and by $X^\sing$ its singular locus. Then there exist a K\"ahler metric $g$ on $X^\reg$ and a finite family $(C_x,g_x)_{x\in X^\sing}$ of K\"ahler cones such that the following holds: for every $x\in X^\sing$ there exists a biholomorphism $(X,x)\cong(C_x,\vx)$ under which $g$ and $g_x$ agree. 
\end{lem}
\begin{proof}
Suppose now conversely that $X$ is a compact normal K\"ahler space whose singularities are isolated and quasi-homogeneous. Choose an open cover $U\cup V\cup\dots =X$ and respective K\"ahler potentials $p_U,p_V,\dots$ on $U,V,\dots$ which define a K\"ahler form on $X.$ For $U$ containing a singular point $x\in X^\sing,$ choose $q$ as in Corollary \ref{cor: Kahl pot} with $x$ in place of $0\in\C^m$ and set $q_U:=q.$ For $U$ not intersecting $X^\sing,$ set $q_U:=p_U.$ The K\"ahler potential $q_U,q_V,\dots$ define then a K\"ahler conifold metric on $X^\reg.$
\end{proof}
\begin{dfn}
In the circumstances of Lemma \ref{lem: conifold metrics} we call $g$ a K\"ahler conifold metric on $X.$ More precisely, this means that there exists $(C_x,g_x)_{x\in X^\sing}$ for which the statement about the biholomorphism $(X,x)\cong(C_x,\vx)$ is true for every $x\in X^\sing.$ 
\end{dfn}
The following lemma is perhaps of interest in itself although we shall not logically need it for the proof of Theorems \ref{main thm1} and \ref{main thm2}.
\begin{lem}\l{lem: C structures of conifolds}
Let $X$ be a compact normal complex space whose singularities are isolated and which has a K\"ahler conifold metric. Then $X$ is a K\"ahler space and its singularities are quasi-homogeneous.
\end{lem}
\begin{proof}
By Theorem \ref{thm: vC}, for every $x\in X^\sing$ the germ $(X,x)$ is quasi-homogeneous. We show that the compact complex space $X$ is a K\"ahler space. As $X$ is normal, if $X$ is one-dimensional then it is non-singular and we have nothing to prove. Suppose therefore that $X$ has dimension $\ge2.$ We use then the following result \cite[Lemma 1]{Fuj2}:
\e\l{Fuj2}\parbox{10cm}{
Let $(Y,y)$ be the germ of a normal complex space of dimension $\ge2,$ and $p:Y\setminus \{y\}\to\R$ a strictly plurisubharmonic $C^\iy$ function. Then there exist a neighborhood $U\sb Y$ of $y$ and a strictly plurisubharmonic $C^\iy$ function $q:Y\to\R$ such that $q|_{Y\setminus U}=p|_{Y\setminus U}.$
}\e
This implies that the K\"ahler potentials which define the cone metrics near $X^\sing$ may be modified so as to define a K\"ahler form on the whole $X.$
\end{proof}

Definition \ref{dfn: CY conifolds} is now equivalent to the following definition.
\begin{dfn}\l{dfn: CY cond}
A {\it compact Calabi--Yau conifold} is a compact K\"ahler conifold $X$ whose canonical sheaf is a rank-one free $\O_X$ module and whose singularities are rational.
\end{dfn}
\begin{rmk}\l{rmk: L^2}
Let $X$ be a compact K\"ahler $n$-conifold whose canonical sheaf is a rank-one free $\O_X$ module. The following three conditions are then equivalent: {\bf(i)} $X^\sing$ is rational; {\bf(ii)} $X^\sing$ is canonical; and {\bf(iii)} $X^\sing$ is log-terminal. This is well known and explained for instance in \cite[Theorem 5.22 and Corollary 5.24]{MK}. The condition (iii) is equivalent also to the following: {\bf(iv)} the nowhere-vanishing $(n,0)$ forms on $X^\reg$ (which are unique up to constant) are $L^2.$ There is in fact also an older result  \cite[Proposition 3.2]{Burns} which proves that (i) and (iv) are equivalent in the present circumstances. Note that the condition (iv) is independent of the choice of a Riemannian metric on $X^\reg$ because an $(n,0)$ form $\Om$ being $L^2$ means $\pm i^n\int_{X^\reg}\Om\wedge\ov\Om<\iy$ ($\pm$ corresponding to the orientation of $X^\reg$).

It is known also that rational singularities are Cohen--Macaulay \cite[Theorem 5.10]{MK}. By \cite[Corollary 3.3(a)]{BS} the germ $(X,x)$ of a Cohen--Macaulay singularity is of depth $\ge n;$ that is, $H^q_x(X,\O_X)=0$ for every integer $q\le n-1.$
\end{rmk}

We prove a fact we will use about compact K\"ahler conifolds. The hypotheses are in fact weaker than being compact K\"ahler conifolds.
\begin{prop}\l{prop: hol 1-forms}
Let $X$ be a compact normal K\"ahler $n$-fold whose canonical sheaf is numerically trivial and whose singularities are non-empty, isolated and log-terminal. Then $H^1(X,\O_X)=0.$ 
\end{prop}
\begin{proof}
We show that no non-zero holomorphic $1$-forms exist on $X^\reg.$ As in \cite{BGL} there exists a quasi-\'etale cover $f:Y\to X$ which is the product of a torus, irreducible Calabi--Yau varieties and irreducible symplectic varieties; moreover, none of the latter two factors has non-zero $1$-forms on the regular locus. As $X$ has isolated singularities, $Y$ is a trivial product; that is, $Y$ is a complex torus, an irreducible Calabi--Yau variety or an irreducible symplectic variety. Consider first the latter two cases and let $X^\reg$ have a holomorphic $1$-form $\ph.$ Then $f^*\ph$ is defined at least outside a co-dimension $2$ subset of $Y,$ where $f$ is \'etale. At every non-singular point $y\in Y$ the sheaf $\Om^1_Y$ is locally free and $f^*\ph$ extends to $y.$ Thus $f^*\ph$ is well defined on $Y^\reg$ and so, as mentioned above, vanishes.
 
Let $Y$ be now a complex torus. By \cite[Lemma 7.4]{GK} there are a complex torus $T$ and a finite group action $G\curvearrowright T$ which is free in co-dimension one and whose quotient space is isomorphic to $X.$ Denote by $\pi:T\to X$ the projection map and let $X^\reg$ have a holomorphic $1$-form $\ph.$ As in the first paragraph the pull-back $\pi^*\ph$ extends to $T$ which is now a $G$-invariant $1$-form. Since $X$ has at least one singular point we get at least one point $t\in T$ at which $G$ has non-trivial stabilizer subgroup $H.$ If $\pi^*\ph$ is non-zero at $t$ then the germ of $\pi(t)\in X$ is the product of $\C,$ corresponding to $\pi^*\ph,$ and some quotient singularity $\C^{n-1}/H.$ But this will not be an isolated singularity. So $\pi^*\ph=0$ at the point $t;$ and as $T$ is a torus, $\pi^*\ph=0$ everywhere. Thus $\ph=0.$

Let $Z\to X$ be a resolution of singularities morphism. Then $H^0(Z,\Om^1_Z)\sb H^0(X^\reg,\Om^1_X)=0.$ As $X$ is K\"ahler so is $Z$ and we have $H^1(Z,\O_Z)\cong H^0(Z,\Om^1_Z)=0.$ Since $X^\sing$ is rational it follows that $H^1(X,\O_X)\cong H^1(Y,\O_Y)=0.$
\end{proof}

\section{Harmonic Forms}\l{sect: harm}
We begin by defining compact Riemannian conifolds.
\begin{dfn}\l{dfn: Riem conifolds}
Let $X$ be a topological space and $x\in X$ any point. Then a {\it punctured neighbourhood} of $x\in X$ is the set $U\-\{x\}$ where $U$ is some (ordinary) neighbourhood of $x\in X.$

A {\it compact Riemannian conifold} consists of a compact metric space $X,$ a Riemannian manifold $(X^\reg,g)$ of dimension $l,$ and a finite family $(C_x,g_x)_{x\in X^\sing}$ of Riemannian cones such that $X=X^\reg\sqcup X^\sing$ as sets; and for every $x\in X^\sing$ there exist a punctured neighbourhood of $\vx\in C_x^\reg,$ a punctured neighbourhood of $x\in X^\reg,$ and a diffeomorphism between these two under which the two Riemannian metrics $g,g_x$ agree with each other.

We call $X$ a Riemannian $l$-conifold if $X$ has real dimension $l.$
\end{dfn}
\begin{rmk}\l{rmk: Riem conifolds}
Although this definition will do for our purpose, the condition that $g$ and $g_x$ should agree locally is stronger than the more standard definition in \cite[Definition 4.6]{Chan}, \cite[Definition 2.2]{HS}, \cite[Definition 2.1]{J1}, \cite[Definition 3.24]{KL} and others. In the latter definition we require only that $g$ should approach with order $\ep>0$ at $x$ the other metric $g_x;$ that is, for $k=0,1,2,\dots$ we have $|\nb^k(g-g_x)|=O(r^\ep)$ where $r$ is the radius function on $C_x^\reg$ and $\nb, |\,\,|$ are computed pointwise with respect to the cone metric $g_x.$
\end{rmk}

We define now weighted Sobolev spaces.
\begin{dfn}\l{dfn: Sob}
Let $X$ be a compact Riemannian $l$-conifold. Choose a smooth function $\rho:X^\reg\to(0,\iy)$ which near every $x\in X^\sing$ agrees with the radius function on $C_x^\reg.$ Define for $k=0,1,2,\dots$ and $\al\in\R$ the weighted Sobolev space $H^k(\La^p_{X^\reg})$ to be the set of $\al\in L^2(\La^p_{X^\reg})$ for which the weak derivatives $\ph,\dots,\nb^k\ph$ exist with
\e \|\ph\|_{H^k_\al}^2:=\int_{X^\reg} \sum_{j=0}^k \rho^{-l}|\rho^{-j-\al}\nb^j\ph|^2 \d\mu<\iy\e
where $|\,\,|,\nb$ and $\d\mu$ are computed with respect to the Riemannian metric of $X^\reg.$ For $k=0$ put $L^2_\al(\La^p_{X^\reg}):=H^0_\al(\La^p_{X^\reg}).$ Put also $L^2(\La^p_{X^\reg}):=L^2_{-l/2}(\La^p_{X^\reg}),$ not $L^2_0(\La^p_{X^\reg}),$ because for $\ph\in L^2(\La^p_{X^\reg})$ we have $\|\ph\|_{L^2}:=\|\ph\|_{L^2_{-l/2}}=\int_{X^\reg}|\ph|^2\d\mu;$ that is, $L^2(\La^p_{X^\reg})$ may be regarded as the unweighted $L^2$ space. We say therefore that a $p$-form $\ph$ on $X^\reg$ is (plainly) $L^2$ if $\ph\in L^2(\La^p_{X^\reg})=L^2_{-l/2}(\La^p_{X^\reg}).$ For $\ph,\ps\in L^2(\La^p_{X^\reg})$ define the inner product $\ph\cdot\ps:=\int_{X^\reg}(\ph,\ps)\d\mu$ where $(\ph,\ps)$ is defined pointwise on $X^\reg,$ using its Riemannian metric. 

Suppose now that $X$ is a K\"ahler conifold. Fix $p,q\in\Z$ and recall that there is a subsheaf $\La^{pq}_{X^\reg}\sb\La^{p+q}_{X^\reg}.$ For $k=0,1,2,\dots$ and for $\al\in\R$ define the weighted Sobolev space $H^k_\al(\La^{pq}_{X^\reg}):=L^2(\La^{pq}_{X^\reg}) \cap H^k_\al(\La^{p+q}_{X^\reg}).$

We can also extend the definitions to a function $\al:X^\sing\to\R,$ allowing $\ph$ to have different orders at different singular points. The modification will be straightforward, which we leave to the reader as an exercise.  
\end{dfn}

We state integration by parts formulae.
\begin{prop}\l{prop: int by parts}
Let $X$ be a compact K\"ahler $n$-conifold. Fix $p,q\in\Z$ and $\al,\be:\R\to X^\sing$ with $\al(x)+\be(x)>1-2n$ for every $x\in X^\sing.$  Then for $\ph\in H^1_\al(\La^{pq}_{X^\reg})$ and $\ps\in H^1_\be(\La^{p\,q+1}_{X^\reg})$ we have $\d\ph\cdot\ps=\ph\cdot\d^*\ps$ and $\db\ph\cdot\ps=\ph\cdot\db^*\ps$ where $\d^*$ and $\db^*$ are computed with respect to the K\"ahler metric of $X^\reg.$
\end{prop}
\begin{proof}
Note that both sides of the formula are well defined by the hypotheses. These will be equal by definition for $\ph,\ps$ with compact support; and such $\ph,\ps$ are dense in the weighted Sobolev spaces. The approximation argument implies therefore that the equality holds for every $\ph,\ps.$
\end{proof}

We define $(p,q)$ form Laplacians and their exceptional values.
\begin{dfn}
Let $X$ be a compact K\"ahler conifold with cone $C_x$ at each $x\in X^\sing.$ Fix $p,q\in\Z.$ Using the K\"ahler metric of $X^\reg$ define the $(p,q)$ form Laplacian $\De:\Ga(\La^{pq}_{X^\reg})\to \Ga(\La^{pq}_{X^\reg})$ by $\De;=\d\d^*+\d^*\d=2\db\db^*+2\db^*\db.$ We call $\al\in\C$ an {\it exceptional value} of $\De$ at $x\in X^\sing$ if there exist non-zero order-$\al$ homogeneous harmonic $(p,q)$ forms on $C_x^\reg.$ 
\end{dfn}
Proposition \ref{prop: real eigen} implies
\begin{prop}
Let $X$ be a compact K\"ahler conifold and $h$ a Hermitian conifold metric on $X.$ Fix $p,q\in\Z.$ Then the set of exceptional values of the $(p,q)$ form Laplacian $\De$ is a discrete subset of $\R.$
\end{prop}
Applying Proposition \ref{prop: Lap} to $(p,q)$ forms we get
\begin{prop}\l{prop: bounded linear}
Let $X$ be a compact K\"ahler conifold, $p,q$ integers and $\De:\Ga(\La^{pq}_{X^\reg})\to \Ga(\La^{pq}_{X^\reg})$ the $(p,q)$ form Laplacian. Then $\De$ defines for $k\in\{2,3,4,\dots\}$ and $\al:X^\sing\to \R$ a bounded linear operator $\De:H^k_\al(\La^{pq}_{X^\reg})\to H^{k-2}_{\al-2}(\La^{pq}_{X^\reg}).$ \qed
\end{prop}

From \cite[Theorem 6.2]{LM} we get also
\begin{prop}\l{prop: non-exceptional Fredholm}
In the circumstances of Proposition \ref{prop: bounded linear} the operator $\De:H^k_\al(\La^{pq}_{X^\reg})\to H^{k-2}_{\al-2}(\La^{pq}_{X^\reg})$ is Fredholm if and only if none of $\{\al(x)\}_{x\in X^\sing}$ is an exceptional value. \qed
\end{prop}

We define the spaces of harmonic $(p,q)$ forms. 
\begin{dfn}\l{dfn: Lap}
In the circumstances of Proposition \ref{prop: bounded linear} denote by $\ker\De^{pq}_\al$ the kernel of the operator $\De:H^k_\al(\La^{pq}_{X^\reg})\to H^{k-2}_{\al-2}(\La^{pq}_{X^\reg}).$ This is independent of $k$ because it consists of harmonic forms, which have the elliptic regularity property. 
\end{dfn}
\begin{rmk}
Note also that even for $k=0,1$ and for $\ph\in H^k_\al(\La^{pq}_{X^\reg})$ we can define the equation $\De\ph=0$ by means of distributions, and that $\ker\De^{pq}_\al$ agrees with the set of its solutions. 
\end{rmk}
From \cite[Lemma 7.3 and \S8]{LM} we get
\begin{prop}\l{prop: improve}
Let $X$ be a compact K\"ahler conifold, $p,q$ integers and $\De:\Ga(\La^{pq}_{X^\reg})\to \Ga(\La^{pq}_{X^\reg})$ the $(p,q)$ form Laplacian. Let a compact interval $[\al,\be]\subset\R$ contain no exceptional values of $\De.$ Then $\ker \De^{pq}_\al=\ker\De^{pq}_\be.$ \qed
\end{prop}

We recall another standard result about elliptic operators between weighted Sobolev spaces.
\begin{prop}\l{prop: Fredholm}
Let $X,p,q,\De$ be as in Proposition \ref{prop: improve}. Let $k\ge2$ be an integer and let $\al:X^\sing\to\R$ have no exceptional value of $\De.$ Then each $\ph\in H^{k-2}_{\al-2}(\La^{pq}_{X^\reg})$ lies in the image of the Fredholm operator $\De:H^k_\al(\La^{pq}_{X^\reg})\to H^{k-2}_{\al-2}(\La^{pq}_{X^\reg})$ if and only if $\ph\cdot \ps=0$ for every $\ps\in \ker\De^{pq}_{2-2n-\al}.$
\end{prop}
\begin{proof}
This is proved in \cite[Theorem 2.14]{J1} for $p=q=0$ and the proof extends immediately to every $p,q.$
\end{proof}

We prove a few fact we shall need about $L^2$ harmonic $n$-forms on $2n$-conifolds. Here $L^2$ is the unweighted $L^2,$ equivalent to the weighted $L^2_{-n}.$ 
\begin{prop}\l{prop: L^2 n-forms}
Let $X$ be a compact Riemannian $2n$-conifold and $\ph$ an $L^2$ harmonic $n$-form on $X^\reg.$  Then $\ph$ is in fact of order $>-n;$ that is, $\ph\in L^2_{\ep-n}(\La^n_{X^\reg})$ for some $\ep>0.$
\end{prop}
\begin{proof}
Fix $x\in X^\sing$ and denote by $\ph_x$ the leading term of $\ph$ expanded as in Theorem \ref{thm: no log}. Write $\ph_x=:r^{-n}(\d\log \r \wedge\ph_x'+\ph_x'')$ where $r$ is the radius function on $C_x^\reg,$ $\ph_x'$ a homogeneous $n-1$ form on $C_x^\lk,$ and $\ph_x''$ a homogeneous $n$-form on $C_x^\lk.$ Since $\ph$ is $L^2$ it follows that $\ph_x$ is $L^2$ over $(0,\de)\times C_x^\lk$ for some $\de>0;$ that is,
\e\l{phph} \ph_x\cdot\ph_x=\int_{-\iy}^{\log\de}\int_{C^\lk}(|\ph_x'|^2+|\ph_x''|^2)\d\mu\,\d\log r<\iy\e 
where $|\,\,|,\d\mu$ are computed on $C_x^\lk.$ So the integral $\int_{C^\lk}(|\ph_x'|^2+|\ph_x''|^2)\d\mu$ is independent of $r,$ which with \eq{phph} implies that $\int_{C^\lk}(|\ph_x'|^2+|\ph_x''|^2)\d\mu=0.$ Thus $\ph_x'=\ph_x''=0$ and $\ph_x=0.$
\end{proof}
\begin{rmk}
We can in fact prove this without using Theorem \ref{thm: no log}. We shall then need to replace $\ph_x$ by $\ph_x+(\log r)\ps_x$ where $\ps_x$ is another homogeneous $n$-form of order $-n.$ But $(\log r)\ps_x$ diverges even faster, which must vanish again by the $L^2$ condition.
\end{rmk}

We prove an integration by part formula for harmonic forms on compact Riemannian conifolds.
\begin{lem}\l{lem: dph1}
Let $X$ be a compact Riemannian $2n$-conifold and $p\le n-1$ an integer. Let $\ep>0$ be so small that $p-n$ is the smallest exceptional value $>-p-\ep$ of $\De:\Ga(\La^p_{X^\reg})\to \Ga(\La^p_{X^\reg}).$ Then for every $\ph\in L^2_{-p-\ep}(\La^p_{X^\reg})$ with $\De\ph=0$ we have $\d\ph=\d^*\ph=0.$
\end{lem}
\begin{proof}
For each $x\in X^\sing$ identify punctured neighbourhoods of $x\in X$ and $\vx\in C_x.$ By Theorem \ref{thm: no log} there exists on $C_x$ an order $-p$ homogeneous harmonic $p$-form $\ps_x$ such that $\ph-\ps_x$ has order $-p+\de$ for some $\de\in(0,\ep).$ By Corollary \ref{cor: order -p} we have $\d\ps_x=\d^*\ps_x=0.$ Accordingly, $\d\ph$ and $\d^*\ph$ are of order $-p-1+\de;$ and in particular, $\De\ph$ is of order $-p-2+\de.$ On the other hand, it follows from the definition of $\ep$ that $\ph$ has order $-p-\frac\de2.$ Recalling $p\le n-1$ and using Proposition \ref{prop: int by parts} we see that the integration by parts formula $\d\ph\cdot\d\ph+\d^*\ph\cdot\d^*\ph=\ph\cdot\De\ph$ is valid. But $\De\ph=0$ and the left-hand side is non-negative. So $\d\ph=\d^*\ph=0.$
\end{proof}

We recall the basic facts we will use about tangent sheaves of normal complex spaces.
\begin{dfn}
If $X$ is a normal complex space then $\Th_X$ denotes its {\it tangent sheaf,} that is, the $\O_X$ module dual to $\Om^1_X.$ 
\end{dfn}
\begin{rmk}
The sheaf $\Th_X$ is as is well known a reflexive sheaf, which has the following properties. Denote by $\io:X^\reg\to X$ the embedding of the regular locus. The natural $\O_X$ module homomorphism $\Th_X\to \io_*\Th_{X^\reg}$ is then an isomorphism. Moreover, $H^1_{X^\sing}(X,\Th_{X^\reg})=0.$
\end{rmk}

We finally state the theorem we will prove in the next section.
\begin{dfn}
Let $X$ be a compact K\"ahler conifold and give it a K\"ahler conifold metric. Take $p,q\in\Z$ and $\al:X^\sing\to\R.$ Then $\ker(\db+\db^*)^{pq}_\al$ consists of $(p,q)$ forms $\ph\in L^2_\al(\La^{pq}_{X^\reg})$ with $\db\ph=\db^*\ph=0.$ 
\end{dfn}
\begin{thm}\l{thm: harm n-1 1}
Let $X$ be a compact Calabi--Yau $n$-conifold with $H^2_{X^\sing}(X,\Om^{n-2}_X)=0$ and give $X$ a K\"ahler conifold metric. Then there exists a $\C$-vector space isomorphism
$H^1(X,\Th_X)\cong \ker(\db+\db^*)^{n-1\,1}_{-n}.$
\end{thm}

\section{Proof of Theorem \ref{thm: harm n-1 1}}\l{sect: harm n-1 1}
We make a definition we will use to prove Theorem \ref{thm: harm n-1 1}.
\begin{dfn}\l{dfn: c}
Let $Y$ be a topological space and $\cF$ a sheaf on it. For $q\in\Z$ denote by ${}_c H^q(Y,\cF)$ the image of the natural map $H_c^q(Y,\cF)\to H^q(Y,\cF)$ from the compact support cohomology group to the plain cohomology group. If $Y$ is embedded in another space $X,$ and $\cF$ induced from a sheaf $\cE$ on $X$ then we write ${}_cH^q(Y,\cE):={}_c H^q(Y,\cF).$   
\end{dfn}

We prove a lemma about Definition \ref{dfn: c}.
\begin{lem}\l{lem: c}
Let $X$ be a topological space, $\cE$ a $\C$-vector space sheaf on $X,$ and $q$ an integer. Let $Y\sb X$ be a finite subset which has a fundamental system $\{U\}$ of neighbourhoods with $H^q(U,\cE)=H^{q+1}(U,\cE)=0$ for each $U.$ The $\C$-vector space ${}_cH^q(X\-Y,\cE)$ is then isomorphic to the image of the natural map $H^q(X,\cE)\to H^q(X\-Y,\cE).$
\end{lem}
\begin{proof}
By hypothesis, for each $U$ the natural map $H^q(U\-Y,\cE)\to H^{q+1}_Y(U,\cE)=H^{q+1}_Y(X,\cE)$ is an isomorphism. On the other hand, there is a commutative diagram
\begin{equation}\begin{tikzcd}[column sep=small]
H^q(X,\cE)\ar[r,swap,"\al"]&H^q(X\-Y,\cE)\ar[r] \ar[d,equal] &H^{q+1}_Y(X,\cE)\\
 H_c^q(X\-Y,\cE)\ar[r,"\be"]&  H^q(X\-Y,\cE)\ar[r]& H^q(U\-Y,\cE)\ar[u,"\cong"]
\end{tikzcd}\end{equation}
with exact rows. The vector space ${}_c H^q(X\-Y,\cE),$ which is by definition the image of $\be$ above, is now equal to the image of $\al$ in the same diagram. 
\end{proof}

We prove a corollary of Lemma \ref{lem: c}. 
\begin{cor}\l{cor: c}
Let $X$ be a compact normal complex space whose singularities are isolated. Then there exists a $\C$-vector space isomorphism ${}_c H^1(X^\reg,\Th_X)\cong H^1(X,\Th_X).$ 
\end{cor}
\begin{proof}
As the Stein neighbourhoods of $X^\sing$ are a fundamental system and $\cE$ a coherent sheaf on $X,$ we can apply Lemma \ref{lem: c} to $\cE=\Th_X,$ $q=1$ and $Y=X^\sing;$ that is, ${}_cH^1(X^\reg,\Th_X)$ agrees with the image of the natural map $H^1(X,\Th_X)\to H^1(X^\reg,\Th_X).$ But $\Th_X$ is a reflexive sheaf and $H^1_{X^\sing}(X,\Th_X)=0.$ The map $H^1(X,\Th_X)\to H^1(X^\reg,\Th_X)$ is therefore injective and hence we get the isomorphism we want.
\end{proof}

We make another definition we will use to prove Theorem \ref{thm: harm n-1 1}.
\begin{dfn}\l{dfn: gr}
Let $X$ be a topological space and $\cE^\bt=(\cE^0\to \cE^1\to\cdots)$ a chain complex of $\C$-vector space sheaves on $X.$ Define for $p=0,1,2,\dots$ a decreasing filtration $F^p\cE^\bt\sb \dots \sb F^0\cE^\bt=\cE^\bt$ as follows: for $q<p$ the degree-$q$ part of $F^p\cE^\bt$ vanishes and for $q\ge p$ its degree part is equal to $\cE^p.$ The inclusion $F^p\cE^\bt\sb \cE^\bt$ induces for $q\in\Z$ a map $H^q(X,F^p\cE^\bt)\to H^q(X,\cE^\bt)$ between the hypercohomology groups, whose image we denote by $F^pH^q(X,\cE^\bt).$ For $p\ge1$ the inclusion $F^p\cE^\bt\sb F^{p-1}\cE^\bt$ induces an inclusion $F^pH^q(X,\cE^\bt)\sb F^{p-1}H^q(X,\cE^\bt)$ which defines thus a decreasing filtration of $H^q(X,\cE^\bt).$ Denote by $\gr^p H^q(X,\cE^\bt)$ the quotient vector space $F^pH^q(X,\cE^\bt)/F^{p+1}H^q(X,\cE^\bt).$ 

Suppose now that $X$ is a complex space and denote by $\io:X^\reg\to X$ the embedding of the regular locus. From the de Rham complex $\Om^\bt_{X^\reg}$ we get for $q\in\Z$ a filtered vector space $H^q(X^\reg,\Om^\bt_{X^\reg}).$ The quasi-isomorphism $\ul\C\to\Om^\bt_{X^\reg}$ induces a $\C$-vector space isomorphism $H^q(X^\reg,\C)\cong H^q(X^\reg,\Om^\bt_{X^\reg}).$ 
We give the subspace ${}_cH^q(X^\reg,\C)\sb H^q(X^\reg,\C)$ the filtration induced from that of $H^q(X^\reg,\Om^\bt_{X^\reg}).$ For $p\in\Z$ the quotient space $\gr^p{}_cH^q(X^\reg,\C)$ is then well defined.
\end{dfn}

We recall a result we will use shortly about harmonic $n$-forms on compact Riemannian conifolds.
\begin{thm}[(0.16) of \cite{Lock}]\l{thm: Lock}
Let $X$ be a compact Riemannian $2n$-conifold and denote by $\ker(\d+\d^*)^n_{-n}$ the $\C$-vector space of $L^2$ closed and co-closed  $n$-forms on $X^\reg.$ The natural projection $\ker(\d+\d^*)^n_{-n}\to H^n(X^\reg,\C)$ which assigns to every $\ph\in \ker(\d+\d^*)^n_{-n}$ its de Rham class $[\ph]\in H^n(X^\reg,\C)$ is then an isomorphism onto ${}_cH^n(X^\reg,\C).$ \qed
\end{thm}

We come now the first step to the proof of Theorem \ref{thm: harm n-1 1}.
\begin{lem}\l{lem: hol improve}
Let $n\ge3$ be an integer and $X$ a compact Calabi--Yau $n$-conifold with $X^\sing$ non-empty. Give $X$ a K\"ahler conifold metric. Let $\ph$ be an $(n,1)$ form of order $>-n-1$ on $X^\reg$ with $\db\ph=\db^*\ph=0.$ Then there exists on $X^\reg$ an $(n,0)$ form $\chi$ of order $>-n$ and with $\db\chi=\ph.$
\end{lem}
\begin{proof}
By Proposition \ref{prop: hol 1-forms} we have $H^1(X,\O_X)=0.$ By hypothesis we have $H^1_{X^\sing}(X,\O_X)=H^2_x(X,\O_X)=0.$ So $H^1(X^\reg,\Om^n_X)\cong H^1(X^\reg,\O_X)\cong H^1(X,\O_X)=0.$ In particular, $\ph$ is $\db$ exact and we can write $\ph=\db\chi$ with $\chi$ an $(n,0)$ form on $X^\reg.$ Then $\db^*\db\chi=0;$ that is, $\chi$ is a harmonic $(n,0)$ form. Write it as a countable sum of homogeneous harmonic forms near each $x\in X^\sing.$ Let $\ph$ have order $\ep-1-n$ with $\ep>0.$ Since $\ph=\db\chi$ it follows then that each homogeneous term $\chi'$ of order $<\ep-n$ of $\chi$ is holomorphic. Recall now from Remark \ref{rmk: L^2} that there exists on $X^\reg$ an $L^2$ holomorphic $(n,0)$ form $\Om.$ We can then write $\chi'=f\Om$ with $f$ some holomorphic function on a punctured neighbourhood of $x\in X.$ As $X$ is normal $f$ extends holomorphically to $x;$ and in particular, it is bounded. But $\Om$ is $L^2;$ and by Proposition \ref{prop: L^2 n-forms}, making $\ep$ smaller if we need, we can suppose that $\Om$ is of order $\ep-n.$ This implies that $f=0$ and that $\chi'=0.$ That is, every homogeneous holomorphic term of order $<\ep-n$ must vanish, which completes the proof.
\end{proof}

We make a more careful study of $\gr^{n-1}{}_c H^n(X^\reg,\Om^\bt_{X^\reg}).$
\begin{thm}\l{thm: inj}
Let $n\ge3$ be an integer and $X$ a compact Calabi--Yau $n$-conifold with $X^\sing$ non-empty.  Give $X$ a K\"ahler conifold metric. Denote by $\ker(\db+\db^*)^{n-1\,1}_{-n}$ the space of $L^2$ harmonic $(n-1,1)$ forms on $X^\reg.$ Then there exists an injective $\C$-linear map $\ker(\db+\db^*)^{n-1\,1}_{-n}\to\gr^{n-1}{}_c H^n(X^\reg,\C).$
\end{thm}
\begin{proof}
We show first that for every $\ph\in \ker(\db+\db^*)^{n-1\,1}_{-n}$ there exists on $X^\reg$ an $(n,0)$ form $\ps$ of order $>-n$ such that $\db\ps=\bd\ph.$ Let $\ep\in(0,1)$ be so small that $\ph$ has order $\ep-n.$ Then $\db^*\bd\ph$ is an $(n,0)$ form of order $\ep-n-2.$ Put $\al:=\ep-n$ and take $\chi\in \ker\De^{n-1\,1}_{2-\al-2n}=\ker\De_{2-\ep-n}.$ Then $\d\chi$ and $\d^*\chi$ have order $>-n$ so integration by parts shows $\d\chi\cdot\d\chi+\d^*\chi\cdot\d^*\chi=\chi\cdot\De\chi=0$ and $\d\chi=\d^*\chi=0.$ Hence it follows that $\db\chi=0$ and that $\chi\cdot\db^*\bd\ph=0.$ So by Proposition \ref{prop: Fredholm} there exists an $(n,0)$ form $\ta$ of order $\al=\ep-n$ and with $\frac12\De\ta=\db^*\bd\ph.$ Then $\bd\ph-\db\ta$ is $\db$ closed and co-closed. By Lemma \ref{lem: hol improve} there exists on $X^\reg$ an $(n,0)$ form $\si$ of order $>n$ and with $\bd\ph-\db\ta=\db\si.$ Put $\ps:=\si+\ta,$ which is of order $\ep-n$ and satisfies $\db\ps=\bd\ph.$

Since $\ph$ is an $(n-1,1)$ form and $\ps$ an $(n,0)$ form it follows that $\d(\ph-\ps)=\db\ph+\bd\ph-\db\ps=0.$ Also $\ph-\ps$ has order $>-n$ and is $L^2.$ By Theorem \ref{thm: Lock} its de Rham cohomology class $[\ph-\ps]$ lies in ${}_cH^n(X^\reg,\C).$ Recalling again that $\ph$ is an $(n-1,1)$ form and $\ps$ an $(n,0)$ form, we can also define $[\ph-\ps]\in \gr^{n-1}{}_cH^n(X^\reg,\C).$ Notice that $[\ph-\ps]\in {}_cH^n(X^\reg,\C)$ may depend on the choice of $\ps$ but that $[\ph-\ps]\in {}_cH^n(X^\reg,\C)$ does not. The map $\ker(\db+\db^*)^{n-1\,1}_{-n}\to\gr^{n-1}{}_c H^n(X^\reg,\C)$ is thus well defined. This is of course $\C$-linear. We show now that it is injective.

Let $\ph\in \ker(\db+\db^*)^{n-1\,1}_{-n}$ be an element of the kernel of the map $\ker(\db+\db^*)^{n-1\,1}_{-n}\to\gr^{n-1}{}_c H^n(X^\reg,\Om^\bt_{X^\reg});$ that is, there exists an $(n,0)$ form $\ps$ of order $>-n,$ with $\bd\ps=\db\ph$ and with $[\ph-\ps]=0\in  \gr^{n-1}{}_c H^n(X^\reg,\C).$ Then there exists another $(n,0)$ form $\ps'$ with $\db\ps'=0$ and $[\ph-\ps-\ps']=0\in H^n(X^\reg,\C).$

As in Remark \ref{rmk: L^2} there exists on $X^\reg$ an $L^2$ holomorphic $(n,0)$ form $\Om.$ Write $\ps'=f\Om$ with $f$ some holomorphic function $X^\reg\to\C.$ As $X$ is normal $f$ extends holomorphically to $X.$ In particular, it is bounded; and accordingly, $f\Om=\ps'$ is $L^2.$ Proposition \ref{prop: L^2 n-forms} implies that $\ps'$ has order $>-n.$  

Thus $\ph-\ps-\ps'$ has order $>-n;$ and accordingly, $\d^*(\ph-\ps-\ps')$ has order $>-n-1.$ Take $\ep>0$ and put $\be:=1+\ep-n.$ Take $\chi\in\ker\De^{n-1}_{2-\be-n}=\ker\De^{n-1}_{1-n-\ep}.$ Letting $\ep>0$ be small enough and using Corollary \ref{cor: order -p} we see that $\d\chi$ and $\d^*\chi$ have order $>-n.$ Integration by parts shows therefore that $\d\chi=\d^*\chi=0$ and that $\chi\cdot\d^*(\ph-\ps-\ps')=0.$ By Proposition \ref{prop: Fredholm} there exists on $X^\reg$ an $n-1$ form $\th$ of order $\be=1+\ep-n$ and with $\De\th=\d^*(\ph-\ps-\ps').$ Then $\d\d^*\th=\d^*(\ph-\ps-\ps'-\d\th).$ Applying $\d^*$ we get $\d^*\d\d^*\th=0;$ that is, $\d^*\th$ is a harmonic $n-2$ form of order $\ep-n.$ By Corollary \ref{cor: harmonic1} however $\d^*\th$ is in fact of order $2-n.$ Now $\d\d^*\th\cdot\d\d^*\th=\d^*\th\cdot \De\d^*\th=0.$ So $\d\d^*\th=0.$ Or equivalently $\d^*(\ph-\ps-\ps'-\d\th)=0.$ Thus $\ph-\ps-\ps'-\d\th\in \ker(\db+\db^*)^n_{-n}.$ Since $[\ph-\ps-\ps']=0\in {}_cH^n(X^\reg,\C)\cong \ker(\db+\db^*)^n_{-n}$ it follows that $\ph-\ps-\ps'-\d\th=0;$ that is, $\ph-\ps-\ps'=\d\th.$ We prove now the following lemma.

\begin{lem}\l{lem: p+q=n-1}
Let $p,q\in\Z$ be such that $p+q=n-1.$ Let $\chi$ be a $(p,q)$ form on $X^\reg$ of order $>1-n$ and with $\db\chi=0.$ Then there exists on $X^\reg$ a $(p,q-1)$ form $\ze$ of order $2-n$ and such that $\chi-\db\ze$ is $\d$ closed and co-closed. 
\end{lem}
\begin{proof}
Notice that $\db^*\chi$ is a $(p,q-1)$ form of order $>-n.$ Take $\ep\in(0,1)$ and put $\al:=2-\ep-n.$ Take $\ze\in\ker\De_{2-\al-2n}^{p\,q-1}=\ker\De^{p\,q-1}_{\ep-n}.$ Integration by parts shows that $\db\ze=\db^*\ze=0.$ So $\ze\cdot\db^*\chi=\db\ze\cdot\chi=0.$ Letting $\ep$ be a non-exceptional value, we see from Proposition \ref{prop: Fredholm} that there exists on $X^\reg$ some $(p,q-1)$ form $\ze$ of order $\al=2-\ep-n$ and with $\frac12\De\ze=\db^*\chi.$ Then $\db\db^*\ze=\db^*(\chi-\db\ze)$ and applying $\db^*$ we get $\db^*\db\db^*\ze=0.$ So $\db^*\th$ is a harmonic $n-3$ form of order $1-\ep-n.$ By Corollary \ref{cor: harmonic1} however $\db^*\ze$ has in fact order $3-n.$ Now $\db\db^*\ze\cdot\db\db^*\ze=\db^*\th\cdot \De\db^*\ze=0.$ So $\db\db^*\ze=0.$ Or equivalently $\db^*(\chi-\db\ze)=0.$ Now $\chi-\db\ze$ is $\db$ closed and co-closed, and in particular harmonic. It is also of order $\al-1=1-\ep-n.$ By Lemma \ref{lem: dph1} therefore $\chi-\db\ze$ is $\d$ closed and co-closed.
\end{proof}
For $(p,q)$ with $p+q=n-1$ denote by $\th^{pq}$ the $(p,q)$ component of the $n-1$ form $\th.$ Since $\d\th=\ph-\ps-\ps'$ has only $(n,0)$ and $(n-1,1)$ components it follows that $\db\th^{n-2\,1}=0.$ Applying Lemma \ref{lem: p+q=n-1} to $\th^{n-2\,1}$ we find some $(n-2,0)$ form $\ze$ of order $2-n$ such that $\th^{n-2\,1}-\db\ze$ is $\d$-closed. In particular, $\partial\th^{n-2\,1}=\partial\db\ze=-\db\partial\ze.$ Put $\et:=\th^{n-1\,0}-\partial\ze,$ which is of order $1-n.$ Then $\d\et=\d\th^{n-1\,0}+\partial\db\ze=\d\th^{n-1\,0}+\partial\th^{n-2\,1}=\ph-\ps-\ps'.$ Looking at the $(n-1,1)$ components of both sides it follows that $\db\et=\ph.$ But $\ph$ has order $>-n$ so $\ph\cdot\ph=\db^*\ph\cdot\et=0$ because $\db^*\ph=0.$ Thus $\ph=0,$ which proves that the map $\ker(\db+\db^*)^{n-1\,1}_{-n}\to\gr^{n-1}{}_c H^n(X^\reg,\Om^\bt_{X^\reg})$ is injective.
\end{proof}

We prove a corollary of Theorem \ref{thm: inj}.
\begin{cor}\l{cor: inj}
Let $X$ be a compact K\"ahler $n$-conifold and give it a K\"ahler conifold metric. The natural projection $\ker(\db+\db^*)^{n-1\,1}_{-n}\to H^1(X^\reg,\Om^{n-1}_{X^\reg})$ is then injective.
\end{cor}
\begin{proof}
Let $\ph\in \ker (\db+\db^*)^{n-1\,1}_{-n}$ have vanishing cohomology class in $H^1(X^\reg,\Om^{n-1}_{X^\reg}).$ So there exists on $X^\reg$ an $(n-1,0)$ form $\chi$ with $\db\chi=\ph.$ As in the proof of Theorem \ref{thm: inj} choose on $X^\reg$ an $(n,0)$ form $\ps$ of order $>-n$ such that $\d(\ph+\ps)=0.$ Now $\ph+\ps-\d\chi$ is an $(n,0)$ form, which means that the map $\ker(\db+\db^*)^{n-1\,1}_{-n}\to\gr^{n-1}{}_c H^n(X^\reg,\C)$ produced in Theorem \ref{thm: inj} maps $\ph$ to $0.$ But this map is injective so $\ph=0$ as we have to prove.
\end{proof}

The image of the natural projection $\ker(\db+\db^*)^{n-1\,1}_{-n}\to H^1(X^\reg,\Om^{n-1}_{X^\reg})$ has also the following property.
\begin{lem}\l{lem: surj}
Let $X$ be a compact K\"ahler $n$-conifold and give it a K\"ahler conifold metric. Then ${}_c H^1(X^\reg,\Om^{n-1}_{X^\reg})$ lies in the image of the natural projection $\ker(\db+\db^*)^{n-1\,1}_{-n}\to H^1(X^\reg,\Om^{n-1}_{X^\reg}).$
\end{lem}
\begin{proof}
Take any element of $ {}_c H^1(X^\reg,\Om^{n-1}_{X^\reg})$ represented on $X^\reg$ by some compactly supported $(n-1,1)$ form $\ph$ with $\db\ph=0.$ Let $\ep\in(0,1)$ be so small as in Lemma \ref{lem: dph1} and put $\al=1+\ep-n.$ Then by Lemma \ref{lem: dph1} every $\chi\in \De^{n-1\,0}_{2-\al-2n}=\De^{n-1\,0}_{1-\ep-n}$ satisfies $\d\chi=0$ and in particular $\db\chi=0.$ So $\db^*\ph\cdot\chi=0.$ By Proposition \ref{prop: Fredholm} therefore there exists $\ps\in H^2_\al(\La^{n-1\,0}_{X^\reg})$ with $\frac12\De\ps=\db^*\ph.$ Or equivalently, $\db^*\db\ps=\db^*(\ph-\db\ps).$ Applying $\db^*$ we get $\De\db^*\ps=0;$ that is, $\db^*\ps$ is a harmonic $n-2$ form of order $\al-1=\ep-n.$ By Corollary \ref{cor: harmonic1} it has in fact order $2-n.$ Now $\db\db^*\ps\cdot\db\db^*\ps=\db^*\ps\cdot \De\db^*\ps=0.$ So $\db\db^*\ps=0.$ Or equivalently, $\db^*(\ph-\db\ps)=0;$ that is, $\ph-\db\ps$ is a $\db$ closed and co-closed $(n-1,1)$ form of order $\ep-n$ as we have to prove. 
\end{proof}

We finally prove
\begin{thm}\l{thm: harm n-1 1 version2}
Let $X$ be a compact Calabi--Yau $n$-conifold with $H^2_{X^\sing}(X,\Om^{n-2}_X)=0$ and give $X$ a K\"ahler conifold metric. Then there exist $\C$-vector space isomorphisms
\e H^1(X,\Th_X)\cong \gr^{n-1}{}_c H^n(X^\reg,\Om^\bt_{X^\reg})\cong \ker(\db+\db^*)^{n-1\,1}_{-n}.\e
\end{thm}
\begin{proof}
The $\O_X$ module sheaf isomorphism $\Om^n_{X^\reg}\cong\O_{X^\reg}$ implies an $\O_X$ module sheaf isomorphism $\Th_{X^\reg}\cong\Om^{n-1}_{X^\reg}.$ Corollary \ref{cor: c} implies therefore a $\C$-vector space isomorphism ${}_cH^1(X^\reg,\Om^{n-1}_{X^\reg})\cong H^1(X,\io_*\Om^{n-1}_X).$ 

By Corollary \ref{cor: inj} and Lemma \ref{lem: surj} there is an injective map ${}_cH^1(X^\reg,\Om^{n-1}_{X^\reg})\to\ker(\db+\db^*)^{n-1\,1}_{-n}.$ Composing this with the injective map produced in Theorem \ref{thm: inj} we get an injective map ${}_cH^1(X^\reg,\Om^{n-1}_{X^\reg})\to  \gr^{n-1}{}_c H^n(X^\reg,\Om^\bt_{X^\reg}).$ We show that this is surjective too. Take any element of $\gr^{n-1}{}_c H^n(X^\reg,\Om^\bt_{X^\reg})$ and represent it on $X^\reg$ by the sum $\ph=\ph^{n0}+\ph^{n-1\,1}$ of an $(n,0)$ form and an $(n-1,1)$ form. As $[\ph]\in {}_cH^n(X^\reg,\C)$ there exists near each $x\in X^\sing,$ on a punctured neighbourhood $U_x$ of it, some $n-1$ form $\ps$ with $\d\ps=\ph.$ More explicitly, using the radius function $r$ we can write $\ph=\d r\wedge\ph'+\ph''$ and simple computation shows that $\d(\int_0^r \ph' \d r)=\ph.$ We can thus take $\ps:=\int_0^r \ph' \d r.$ Since $\ph'=\frac{\partial}{\partial r}\mathbin{\lrcorner} \ph$ is the sum of an $(n-1,0)$ form and an $(n-2,1)$ form it follows that so is $\ps,$ which we write as $\ps^{n-1\,0}+\ps^{n-2\,1}.$ Since $\d\ps$ has vanishing $(n-2,2)$ component it follows that $\db\ps^{n-2\,1}=0.$ Recall from hypothesis that $H^1(U_x,\Om^{n-2}_X)\cong H^2_x(X,\Om^{n-2}_X)=0.$ Then we can find on $U_x$ an $(n-2,0)$ form $\chi$ such that $\db\chi=\ph^{n-2\,1}.$ Now on $U_x$ we have
\e\l{decompose n-1 1} 
\ph^{n-1\,1}=\db \ps^{n-1\,0}+\partial\ps^{n-2\,1}=\db \ps^{n-1\,0}-\db\partial \chi.
\e
So $\ph^{n-1\,1}$ defines an element of ${}_cH^1(X^\reg,\Om^{n-1}_{X^\reg}).$ This element maps to $[\ph]$ under the map ${}_cH^1(X^\reg,\Om^{n-1}_{X^\reg})\to  \gr^{n-1}{}_c H^n(X^\reg,\Om^\bt_{X^\reg}),$ which is thus surjective.

The map ${}_cH^1(X^\reg,\Om^{n-1}_{X^\reg})\to  \gr^{n-1}{}_c H^n(X^\reg,\Om^\bt_{X^\reg})$ is now a $\C$-vector space isomorphism which is the composite of the injective maps ${}_cH^1(X^\reg,\Om^{n-1}_{X^\reg})\to\ker(\db+\db^*)^{n-1\,1}_{-n}$ and $\ker(\db+\db^*)^{n-1\,1}_{-n}\to \gr^{n-1}{}_c H^n(X^\reg,\Om^\bt_{X^\reg}).$ The latter two maps are therefore isomorphisms too, which completes the proof. 
\end{proof}

\section{Harmonic $n-1$ Forms}\l{sect: harm n-1}

The following is an analogue of Lemma \ref{lem: surj}.
\begin{lem}\l{lem: c harm n-1}
Let $X$ be a compact K\"ahler $n$-conifold and give it a K\"ahler conifold metric. Take $p,q\in\Z$ with $p+q=n-1.$ Then ${}_cH^q(X^\reg,\Om^p_{X^\reg})$ lies in the image of the natural projection $\ker(\db+\db^*)^{pq}_{1-n}\to H^q(X^\reg,\Om^p_{X^\reg}).$
\end{lem}
\begin{proof}
Let $\ep\in(0,1)$ be small enough and put $\al:=-\ep+2-n.$ We show first that every $\chi\in\ker\De^{p\, q-1}_{2-\al-2n}=\ker\De^{p\, q-1}_{-n+\ep}$ satisfies $\db\chi=0.$ Corollary \ref{cor: harmonic1} and Proposition \ref{prop: improve} imply that $\chi\in \ker\De^{p\, q-1}_{2-n-\de}$ for any $\de>0.$ Expanding $\chi$ near $X^\sing$ as in Theorem \ref{thm: no log} we see that $\chi$ has order $2-n.$ Integration by parts therefore makes sense so that $\chi\cdot\d^*\d\chi=\d\chi\cdot \d\chi+\d^*\chi\cdot\d^*\chi.$ So $\d\chi=0.$ As $\chi$ is a pure $(p,q)$ form, we have $\db\chi=0.$

Take an element of $H^q(X^\reg,\Om^p_X)$ represented on $X^\reg$ by a compactly supported $(p,q)$ form $\ph$ with $\db\ph=0.$ Letting $\chi$ be as above we have $0=\ph\cdot\db\chi=\db^*\ph\cdot\chi.$ But $\chi$ is an arbitrary element of $\ker\De^{p\, q-1}_{2-\al-2n},$ so $\db^*\ph$ is orthogonal to $\ker\De^{p\, q-1}_{2-\al-2n}.$ Proposition \ref{prop: Fredholm} implies therefore that $\db^*\ph$ lies in the image of the Fredholm operator $\De: H^2_\al(\La^{p\, q-1}_{X^\reg})\to L^2_{\al-2}(\La^{p\, q-1}_{X^\reg}).$ Write $\db^*\ph=\frac12\De\ps$ with $\ps\in H^2_\al(\La^{p\, q-1}_{X^\reg}).$ Then $\db^*(\ph-\db\ps)=\db\db^*\ps$ and so $\db^*\ps$ is harmonic. As $\al-1=-\ep+1-n>-n$ it follows from Corollary \ref{cor: harmonic1} that $\db^*\ps$ has order $2-n.$ Integration by parts therefore makes sense and we have $\db\db^*\ps=0.$ So $\db^*(\ph-\db\ps)=\db\db^*\ps=0$ and $\ph-\db\ps$ is harmonic of order $1-n.$ By Corollary \ref{cor: order -p} then $\d(\ph-\db\ps)$ and $\d^*(\ph-\db\ps)$ have order $-n.$ Integration by parts implies therefore $\d(\ph-\db\ps)=\d^*(\ph-\db\ps)=0.$ Thus $\ph-\db\ps$ is a representative we want of the cohomology class of $\ph.$
\end{proof}

Under the hypothesis of Theorem \ref{main thm1} we can improve the result above. 
\begin{cor}\l{cor: c harm n-1}
Let $X,n,p$ and $q$ be as in Lemma \ref{lem: c harm n-1}. Suppose that for each $x\in X^\sing$ we have $b_{n-2}(C_x^\reg)=0$ or $b_{n-1}(C_x^\reg)=0.$ Then there exists $\ep>0$ such that ${}_cH^q(X^\reg,\Om^p_{X^\reg})$ lies in the image of the natural projection $\ker(\db+\db^*)^{pq}_{1+\ep-n}\to H^q(X^\reg,\Om^p_{X^\reg}).$
\end{cor}
\begin{proof}
We modify the proof of Lemma \ref{lem: c harm n-1}. Let $\ep>0$ be small enough and define $\al:X^\sing\to\R$ by saying that for $x\in X^\sing$ with $b_{n-2}(C_x^\reg)=0$ we should have $\al(x)=\ep+2-n$ and for every other $x\in X^\sing$ (with $b_{n-1}(C_x^\reg)=0$) we should have $\al(x)=-\ep+2-n.$ 
We show that every $\chi\in\ker\De^{p\, q-1}_{2-\al-2n}$ satisfies $\db\chi=0.$ For $x\in X^\sing$ with $b_{n-2}(C_x^\reg)=0$ the order of $\chi$ is $-n-\ep,$ and Corollaries \ref{cor: harmonic1} and \ref{cor: order 2+p-l} imply that it is in fact $2-n.$ For $x\in X^\sing$ with $b_{n-1}(C_x^\reg)=0$ the order of $\chi$ is $-n+\ep$ and Corollary \ref{cor: harmonic1} implies that it is in fact $2-n.$ In both cases $\chi$ has order $2-n$ and as in the proof of  Lemma \ref{lem: c harm n-1} we can verify $\db\chi=0.$

Let $\ph,\ps$ be as in the proof of Lemma \ref{lem: c harm n-1}. Then $\ps$ has order $\al,$ and $\ph-\db\ps$ is a harmonic form of order $\al-1.$ For $x\in X^\sing$ with $b_{n-2}(C_x^\reg)=0$ we have $\al(x)=\ep+2-n>2-n.$ So $\ph-\db\ps$ has order $>1-n$ at $x.$ On the other hand, for $x\in X^\sing$ with $b_{n-1}(C_x^\reg)=0$ we expand $\ph-\db\ps$ at $x$ into the sum of homogeneous harmonic forms. As $\al(x)-1=-\ep+1-n$ with $\ep$ small enough the leading term has order $1-n$ which, by Corollary \ref{cor: order -p}, vanishes. So $\ph-\db\ps$ has order $>1-n,$ which completes the proof.  
 \end{proof}

We make a more careful study of $(1,n-2)$ forms. We prove a lemma we will use for the next theorem.
\begin{lem}\l{lem: p exact}
Let $X$ be a compact Riemannian $l$-conifold, $p$ an integer and $\ph$ a $C^\iy$ $p$-form on $X^\reg$ of order $>-p.$ Then every $x\in X^\sing$ has a punctured neighbourhood $U^\reg$ on which $\ph$ is $\d$-exact.
\end{lem}
\begin{proof}
Denote by $C_x$ the model cone at $x$ of $X.$ Write $\ph=\d\log r\wedge \ph'+\ph''$ where $\ph'$ is the pull-back of some $p-1$ form on $C^\lk_x,$ and $\ph''$ that of some $p$-form on $C^\lk_x.$ Making $U^\reg$ small enough we can suppose that it is diffeomorphic to $(0,\de)\times C^\lk_x$ for some $\de>0.$ Take any $p$-cycle $A$ on $C^\lk_x.$ Then $\int_{\{r\}\times A}\ph=\int_{\{r\}\times A}\ph''$ is independent of $r\in(0,\de).$ But since $\ph$ is of order $>-p$ it follows that so is $\ph''$ and hence that $\int_{\{r\}\times A}\ph''$ converges to $0$ as $r$ tends to $0.$ Thus $\int_{\{r\}\times A}\ph=0,$ which implies that $\ph$ is $\d$-exact on $U^\reg.$
\end{proof}

For $(1,n-2)$ forms we improve Corollary \ref{cor: c harm n-1} as follows.
\begin{lem}\l{lem: 1 n-2}
Let $X$ be a compact K\"ahler $n$-conifold whose singularities are of depth $\ge n.$ Suppose that for each $x\in X^\sing$ we have $b_{n-2}(C_x^\reg)=0$ or $b_{n-1}(C_x^\reg)=0.$ Give $X$ a K\"ahler conifold metric. Then there exists $\ep>0$ such that ${}_cH^{n-2}(X^\reg,\Om^1_{X^\reg})$ agrees with the image of the natural projection $\ker(\db+\db^*)^{1\, n-2}_{1+\ep-n}\to H^{n-2}(X^\reg,\Om^1_{X^\reg}).$
\end{lem}
\begin{proof}
Let $\ph$ be a harmonic $(1,n-2)$ form on $X^\reg$ of order $>1-n.$ Integration by parts shows that $\d\ph=\d^*\ph=0.$ Lemma \ref{lem: p exact} implies that every $x\in X^\sing$ has a punctured neighbourhood on which we can write $\ph=\d\ps$ with $\ps$ some $n-2$ form. As $\ph$ is a $(1,n-2)$ form we can write also $\ph=\bd\ps'+\db\ps''$ where $\ps'$ is some $(0,n-2)$ form with $\db\ps'=0,$ and $\ps''$ some $(1,n-3)$ form; this may be verified in the same way as in the paragraph containing \eq{decompose n-1 1}. Let $U$ be a Stein neighbourhood of $x\in X^\sing$ (which is an ordinary neighbourhood and contains therefore $x\in X^\sing$). Since $(X,x)$ has depth $\ge n$ it follows then that $H^{n-2}(U\-\{x\},\O_X)\cong H^{n-1}_x(U,\O_X)=0.$ So $\ps'=\db\chi$ where $\chi$ is some $(0,n-3)$ form on $U\-\{x\}.$ Thus $\ph=\bd\db\chi+\db\ps''=\db(-\bd\chi+\ps'')$ is $\db$ exact. Using cut-off functions, we see that the $\db$ cohomology class of $\ph$ lies in ${}_c H^{n-2}(X^\reg,\Om^1_X).$

For $\ep>0$ the image of the natural projection $\ker(\db+\db^*)^{1\, n-2}_{1+\ep-n}\to H^{n-2}(X^\reg,\Om^1_{X^\reg})$ thus lies in ${}_cH^{n-2}(X^\reg,\Om^1_{X^\reg}).$ This with Corollary \ref{cor: c harm n-1} completes the proof.
\end{proof}

\section{Deformation Functors}\l{sect: deform}
We make a basic definition we will use in what follows.
\begin{dfn}\l{dfn: Art}
For $\K=\R$ or $\C$ a {\it local $\K$-algebra} is a $\K$-algebra $A$ with unique maximal ideal $\fm A$ such that the natural maps $\K\to A\to  A/\fm A$ induce an isomorphism $\K\cong A/\fm A.$ An {\it Artin} local $\K$-algebra is a local $\K$-algebra $A$ which is an {\it Artin} ring; that is, every descending chain of ideals in it should be finite. 
\end{dfn}
\begin{rmk}
It is well-known that Artin rings are Noetherian rings. Moreover, for a local $\K$-algebra $A$ the following three conditions are equivalent: {\bf(i)} $A$ is an Artin ring; {\bf(ii)} $A$ is a Noetherian ring, and there exists an integer $n\ge1$ such that $(\fm A)^n=0;$ and {\bf(iii)} $A$ is a finite-dimensional $\K$-vector space. The proof is as follows. If (i) holds then by the descending chain condition there exists an integer $n\ge1$ such that $(\fm A)^n=(\fm A)^{n+1}.$ Nakayama's lemma implies therefore $(\fm A)^n=0.$ Using the $A$-module exact sequence $0\to (\fm A)^k\to (\fm A)^{k-1}\to (\fm A)^{k-1}/(\fm A)^k\to0$ for $k=1,\dots,n$ we see also that (iii) holds. Conversely, it is clear that (iii) implies the descending chain condition which is equivalent to (i). The three conditions are thus equivalent.
\end{rmk}

We prove a lemma we will use in Definition \ref{dfn: real para}.
\begin{lem}\l{lem: complexify}
If $A$ is an Artin local $\R$-algebra then the tensor product $A\otimes_\R\C,$ which is naturally a $\C$-algebra, is an Artin local $\C$-algebra. 
\end{lem}
\begin{proof}
The composite of the natural maps $A\otimes_\R\C\to (A/\fm A)\otimes_\R\C\cong \C$ has kernel $\fm A\otimes_\R\C,$ which is therefore a maximal ideal of $A\otimes_\R\C.$ We show that its complement consists of invertible elements. Take $a\in( A\otimes_\R\C)\-(\fm A\otimes_\R\C)$ and write $a=:a'\otimes1+a''\otimes i$ with $a',a''\in A.$ Put $a'=:b'+c'$ and $a''=:b''+c''$ where $b',b''\in \C$ and $c',c''\in\fm A.$ Since $a\notin  (\fm A\otimes_\R\C)$ it follows that $b:=b'\otimes 1+b''\otimes i\in \C\-\{0\}.$ Since $c',c''\in\fm A$ are nilpotent it follows that so is $c:=c'\otimes1+c''\otimes i\in A\otimes_\R\C.$ As $b\ne0$ we can define $\frac cb\in A\otimes_\R\C,$ and $c$ nilpotent implies $\frac cb$ nilpotent. Consequently $1+\frac cb$ is invertible and accordingly so is $b+c=a.$

The ring $A\otimes_\R\C$ is thus a local ring with unique maximal ideal $\fm A\otimes_\R\C.$ Since $A$ is a finite-dimensional $\R$-vector space it follows that $A\otimes_\R\C$ is a finite-dimensional $\C$-vector space, which must therefore be an Artin ring as we have to prove.
\end{proof}

\begin{dfn}
Let $\K=\R$ or $\C.$ We denote by $(\Art)_\K$ the category whose objects are Artin local $\K$-algebras and whose morphisms are $\K$-algebra homomorphisms. A {\it small extension} homomorphism in $(\Art)_\K$ is a surjective $\K$-algebra homomorphism $A\to B$ whose kernel is a non-zero principal ideal $(\ep)\sb A$ such that the product ideal $(\ep)\fm A\sb A$ vanishes.
\end{dfn}

\begin{dfn}\l{dfn: H}
Denote by $(\sets)$ the category with objects sets and morphisms maps. Call a functor $D:(\Art)_\C\to(\sets)$ a {\it deformation} functor if $D(\C)$ consists of a single element. For a deformation functor $D:(\Art)_\C\to(\sets)$ we consider the following conditions:
\iz
\item[\bf (H1)] Let $A,B,C$ be Artin local $\C$-algebra, $A\to C$ a $\C$-algebra homomorphism and $B\to C$ a small extension homomorphism in $(\Art)_\C.$ The induced map $D(A\times_CB)\to D(A)\times_{D(C)} D(B)$ is then surjective.
\item[\bf (H2)] Let $A$ be an Artin local $\C$-algebra and take $B:=\C[t]/t^2.$ The induced map $D(A\times_\C B)\to D(A)\times_{D(\C)}D(B)$ is then bijective.
\iz
It is known that (H1) and (H2) imply the following condition: 
\begin{equation}\l{H12}\parbox{10cm}{
If $A\to B$ is a small extension homomorphism in $(\Art)_\C$ then the additive group $D(\C[t]/t^2)$ acts transitively upon the non-empty fibres of $D(A)\to D(B).$
}\end{equation}
This is proved as follows. Denote by $\pi:A\to A/\fm A \cong\C$ the natural projection and by $(\ep):=\ker (A\to B)$ the non-zero principal ideal of $A.$ There is then a $\C$-algebra isomorphism $A\times_\C (\C[t]/t^2)\cong A\times_BA$ defined by $(a,\pi a+\la t)\mapsto (a,a+\la\ep)$ for $a\in A$ and $\la\in\C.$ Using this and the condition (H1) we get a bijection $D(A)\times_{D(\C)}D(\C[t]/t^2)\cong D(A\times_BA).$ On the other hand, (H2) implies that the induced map $D(A\times_BA)\to D(A)\times_{D(B)}D(A)$ is surjective. Combining these two maps we get a surjection $D(A)\times_{D(\C)}D(\C[t]/t^2)\to D(A)\times_{D(B)}D(A)$ which defines the transitive action we want.

It is also easy to show that if (H2) holds then $D(\C[t]/t^2)$ has a natural $\C$-vector space structure. In this case consider the following condition:
\iz
\item[\bf (H3)] $D(\C[t]/t^2)$ is a finite-dimensional $\C$-vector space.
\iz
Schlessinger \cite[Theorem 2.11(1)]{Schless} proves that (H1)--(H3) hold if and only if $D$ has a hull \cite[Definition 2.7]{Schless}. 

Define for $k=0,1,2,\dots$ two $\C$-algebras $A_k:=\C[t]/t^{k+1}$ and $A_k[\ep]:=\C[t,\ep]/(t^{k+1},\ep^2).$ Consider the natural projection $A_k[\ep]\to A_k$ and the induced map $D(A_k[\ep])\to D(A_k).$ For $\xi\in D(A_k)$ denote by $T^1(\xi)$ the set of $\et\in D(A_k[\ep])$ which maps to $\xi$ under $D(A_k[\ep])\to D(A_k).$ Following \cite[\S1.5]{Gross} consider the condition called (H5). If (H1)--(H3) and (H5) hold then for $k=0,1,2,\dots$ and $\xi\in D(A_k)$ there is on $T^1(\xi)$ a natural $A_k$ module structure.

An {\it obstruction space of} $D$ a $\C$-vector space $T^2$ with the following two properties: {\bf(i)} for every small extension $0\to(\ep)\to A\stackrel{f}{\to} B\to 0$ in $(\Art)_\C$ there exists a sequence $D(A)\xrightarrow{D(f)} D(B)\to T^2\otimes_\C(\ep)$ which is exact in the sense that the image of the former map $D(f)$ agrees with the fibre over $0\in T^2\otimes_\C(\ep)$ of the latter map; and {\bf(ii)} if there is in $(\Art)_\C$ a commutative diagram
\begin{equation} \begin{tikzcd}
0\ar[r]& (\ep)\ar[r]\ar[d,"\al"]& A\ar[r,"f"] \ar[d,"\al"]& B \ar[d,"\be"]\ar[r]&0\\
 0\ar[r]& (\ep')\ar[r]              & A'\ar[r,"g "]               &  B'\ar[r]                   &0
 \end{tikzcd}\end{equation}
whose rows are small extension homomorphisms then there is a commutative diagram
$\begin{tikzcd}
 D(A)\ar[r,"D(f)"]\ar[d,"D(\al)"]& D(B)\ar[r]\ar[d,"D(\be)"]& T^2\otimes_\C(\ep)\ar[d,"\id\otimes \al"]\\
D(A')\ar[r,"D(g)"]& D(B')\ar[r]& T^2\otimes_\C(\ep')
\end{tikzcd}$
whose rows are the exact sequences in (i) just mentioned. We call the map $D(A)\to T^2(X)\otimes_\C(\ep)$ the {\it obstruction map of $(T^2(X),f).$}
\end{dfn}

We recall a version we will use of $T^1$ lift theorems; for the original versions see \cite{Kaw,Ran}, and for the more complex version we will use to prove Theorem \ref{main thm1} see Lemma \ref{lem: concentrate}. 
\begin{thm}[Theorem 1.8 of \cite{Gross}]\l{thm: T1}
Let $(\Art)_\C\to(\sets)$ be a deformation functor satisfying {\rm(H1)--(H3)} and {\rm(H5)} and having an obstruction space. For $k=1,2,3,\dots$ denote by $\pi_k:A_k\to A_{k-1}$ the natural projection. Suppose that for $k=1,2,3,\dots$ and $\xi\in D(A_k),$ if we put $\et:=D(\pi_k)(\xi)\in D(A_{k-1})$ then the natural map $T^1(\xi)\to T^1(\et)$ is surjective. The maps $D(\pi_1),D(\pi_2),D(\pi_3),\dots$ are then all surjective. \qed
\end{thm}
\begin{rmk}\l{rmk: T1}
More precisely, the following holds. Take $k=1,2,3,\dots$ and $\xi\in D(A_k).$ Define a $\C$-algebra homomorphism $\th_k:A_k\to A_{k-1}[\ep]$ by $t\mapsto t+\ep$ module ideals (so that if we define $\et$ as in Theorem \ref{thm: T1} above then $D(\th_k)(\xi)\in T^1(\et)$). Denote by $\varpi_k: A_k[\ep]\to A_{k-1}[\ep]$ the natural projection. Then $\xi$ lies in the image of $D(\pi_{k+1}):D(A_{k+1})\to D(A_k)$ if and only if $D(\th_k)(\xi)$ lies in the image of $D(\varpi_k):D(A_k[\ep])\to D(A_{k-1}[\ep]).$
\end{rmk}

We turn now to the examples of deformation functors. We begin by recalling the definition of $A$-ringed spaces.
\begin{dfn}\l{dfn: complex spaces}
Let $A$ be a commutative ring with unit. Then an {\it $A$-ringed space} is the pair $(X,\O_X)$ where $X$ is a topological space and $\O_X$ a sheaf on $X$ of $A$-algebras. A {\it morphism} from an $A$-ringed space $(X,\O_X)$ to another $A$-ringed space $(Y,\O_Y)$ is the pair of a continuous map $X\to Y$ and an $A$-algebra sheaf homomorphism $\O_Y\to f_*\O_X.$
\end{dfn}
\begin{rmk}
Complex spaces are thus $\C$-ringed spaces. Morphisms of complex spaces are by definition the morphisms of $\C$-ringed spaces. 
\end{rmk}
We recall the standard definitions about deformations of complex spaces.
\begin{dfn}\l{dfn: def}
Let $X$ be a compact complex space. A {\it deformation of $X$} is the data $(\cX,S,o,f,\ph)$ where $\cX,S$ are complex analytic spaces, $o\in S$ a point of the underlying topological space, $f:\cX\to S$ a proper flat morphism of complex spaces, and $\ph:\cX\times_S\{o\}\cong X$ a complex space isomorphism. For $X$ compact, we require $f$ to be proper. We omit $o,f,\ph$ when they are clear from the context. 

Fix a complex space $S$ and a point $o\in S.$ Let $f:\cX\to S$ and $g:\cY\to S$ be deformations of $X.$ Then an {\it isomorphism} from $f:\cX\to S$ to $g:\cY\to S$ is a complex space isomorphism $\cX\to\cY$ which induces over $o\in S$ the identity map $X\cong \cX\times_S\{o\}\to \cY\times_S\{o\}\cong X.$

There is always a deformation of $X$ defined by $X\times S$ with the projection $X\times S\to S,$ which we call the {\it trivial} deformation of $X.$

For an Artin local $\C$-algebra $B$ we denote by $\spec B$ the complex space whose underlying topological space consists of one point and whose stalk over it is exactly $B.$ A deformation {\it over} $B$ of $X$ is a deformation $(\cX,S,o,f,\ph)$ with $S=\spec B.$ We denote this by $\cX/B$ for short.
\end{dfn}
\begin{rmk}
It follows from definition that $\cX$ has the same underlying space as $X$ and that $\O_\cX$ is a sheaf of $B$-algebras. Thus $\cX$ is a $B$-ringed space $(X,\O_\cX).$ Isomorphisms of two deformations over $B$ are the $B$-ringed space isomorphisms.
\end{rmk}

We define also deformations over real parameter spaces. This will be crucial to defining real differential forms including K\"ahler forms; for more details see Definitions \ref{dfn: glue} and \ref{dfn: Kahl}.
\begin{dfn}\l{dfn: real para}
Let $X$ be a complex space and $A$ an object of $(\Art)_\R.$ Recall from Lemma \ref{lem: complexify} that $B:=A\otimes_\R\C$ is an Artin local $\C$-algebra. A {\it deformation} over $A$ of $X$ is a deformation over $B$ of $X.$ We denote this by $\cX/A.$ So $\cX/A=\cX/B$ in notation, and we choose the more convenient one according to the context. 
\end{dfn}

We give now the first key example of deformation functors. We give only a short account of the relevant facts; for more details see for instance \cite{Nam}. 
\begin{ex}\l{ex: D}
Let $X$ be a compact reduced complex space. Denote by $D:(\Art)_\C\to (\sets)$ the deformation functor which assigns to every Artin local $\C$-algebra $A$ the set of isomorphism classes of deformations over $A$ of $X.$ It is known that $D$ satisfies (H1)--(H3) and (H5). It has also an obstruction space $T^2(X):=\Ext^2_{\O_X}(L_X,\O_X)$ where $L_X\in D^-(\Mod\O_X)$ is the cotangent complex of $X.$ We can therefore apply Theorem \ref{thm: T1} to the deformation functor $D.$ 

It is known that for $k=0,1,2,\dots,$ if $X_k/A_k$ is a deformation of $X$ then its $T^1$ module $T^1(X_k/A_k)$ is isomorphic to $\Ext^1_{\O_{X_k}}(\Om^1_{X_k},\O_{X_k}).$ Thus, if for $k=1,2,3,\dots$ the natural map  $T^1(X_k/A_k)\to T^1(X_{k-1}/A_{k-1})$ is always surjective then so are $D(\pi_1),D(\pi_2),D(\pi_3),\dots$ as in Theorem \ref{thm: T1}. By \cite{Douad,Grauert1, Kur} there exists a Kuranishi space $\Def(X),$ the base space of semi-universal deformations of $X.$ If $D(\pi_1),D(\pi_2),D(\pi_3),\dots$ are surjective then every tangent vector to $\Def(X)$ may be lifted to a formal path; that is, $\Def(X)$ is non-singular. 
\end{ex}

Deformations of complex space germs are defined in the same way as in Definition \ref{dfn: def}. One difference is that a deformation $(\cX,x)\to (S,o)$ of the germ $(X,x)$ is no longer a proper map. But otherwise the modification is straightforward. The corresponding deformation functors have properties similar to those of Example \ref{ex: D}, as we recall briefly now; for more details see for instance \cite{GLS}.
\begin{ex}
Let $(X,x)$ be the germ of a reduced complex space. Denote by $D:(\Art)_\C\to (\sets)$ the deformation functor which assigns to every Artin local $\C$-algebra $A$ the set of isomorphism classes of deformations over $A$ of $X.$ It is known that $D$ satisfies (H1)--(H3) and (H5). It has also an obstruction space $T^2:=\Ext^2_{\O_{X,x}}(L_{X,x},\O_{X,x})$ where $L_{X,x}\in D^-(\Mod\O_{X,x})$ is the cotangent complex of $(X,x).$ We can therefore apply Theorem \ref{thm: T1} to the deformation functor $D.$ 

It is known that for $k=0,1,2,\dots,$ if $X_k/A_k$ is a deformation of $(X,x)$ then its $T^1$ module is isomorphic to 
$\Ext^1_{\O_{X_k,x}}((\Om^1_{X_k/A_k})_x,\O_{X_k,x}).$ Thus, if for $k=1,2,3,\dots$ the natural map $T^1(X_k/A_k)\to T^1(X_{k-1}/A_{k-1})$ is always surjective then so are $D(\pi_1),D(\pi_2),D(\pi_3),\dots$ as in Theorem \ref{thm: T1}. This will imply that the Kuranishi space, which exists by \cite{Don,Tyur}, is non-singular.
\end{ex}

We recall also the basic facts about locally trivial deformations; for more details see for instance \cite[Corollary 2.6 and Remark 2.7]{BGL}.
\begin{ex}\l{ex: loc triv def}
Let $X$ be a compact complex space. Denote by $D:(\Art)_\C\to (\sets)$ the deformation functor which assigns to every Artin local $\C$-algebra $A$ the set of isomorphism classes of locally trivial deformations over $A$ of $X.$ It is known that $D$ satisfies (H1)--(H3) and (H5). It has also an obstruction space $T^2(X):=H^2(X,\Th_X).$ We can therefore apply Theorem \ref{thm: T1} to $D.$

Suppose now that $\cX/A$ is a locally trivial deformation of $X$ so its isomorphism class defines an element of $D(A).$ There exists then a surjective $A$-module homomorphism $H^1(X,\Th_{\cX/A})\to T^1(\cX/A).$ This will be an isomorphism if $\cX/A$ has no non-trivial automorphisms that, restricted to $X,$ become the identity. 

Let $k\ge1$ be an integer, $X_k/A_k$ a deformation of $X,$ and $X_{k-1}/A_{k-1}$ the deformation of $X$ defined by $X_{k-1}:=\spec A_{k-1}\times_{\spec A_k}X_k.$ There is then a commutative diagram
\begin{equation}\l{ga3}
\begin{tikzcd}
H^1(X,\Th_{X_k/A_k})\ar[r]\ar[d,"\al"]& T^1(X_k/A_k)\ar[d,"\ga"]\\
H^1(X,\Th_{X_{k-1}/A_{k-1}})\ar[r,"\be"]& T^1(X_{k-1}/A_{k-1})
\end{tikzcd}
\end{equation}
where the horizontal maps are those introduced above and the vertical maps those induced by $\pi_k:A_k\to A_{k-1}.$ Suppose now that the left vertical map $\al$ is surjective. Since the bottom horizontal map $\be$ is surjective as mentioned above it follows then that $\be\cm\al$ is surjective. Accordingly, so is $\ga.$ The functor $D$ thus satisfies the hypothesis of Theorem \ref{thm: T1}; and consequently, $D(\pi_1),D(\pi_2),D(\pi_3),\dots$ are surjective. This will imply that the Kuranishi space, which exists by \cite[Corollary 0.3]{FK}, is non-singular.
\end{ex}

As in \cite{Gross2} we cannot expect that our deformation functors are always unobstructed. Following \cite[Theorem 2.2]{Nam0} therefore we make
\begin{dfn}\l{dfn: concentrate}
Let $X$ be a compact reduced complex space whose singularities are isolated. For $x\in X^\sing$ denote by $D_x:(\Art)_\C\to (\sets)$ the deformation functor which assigns to every Artin local $\C$-algebra $A$ the set of isomorphism classes of deformations over $A$ of the germ $(X,x).$ It is known that $D_x$ has an obstruction space $T^2(X):=\Ext^2_{\O_{X,x}}(L_{X,x},\O_{X,x})$ where $L_{X,x}\in D^-(\Mod\O_{X,x})$ is the cotangent complex of $(X,x).$ Define $D_\loc:(\Art)_\C\to(\sets)$ by $D_\loc(A):=\prod_{x\in X^\sing}D_x(A)$ for $A$ an object of $(\Art)_\C.$ Put $T^2_\loc(X):=\bop_{x\in X^\sing}T^2(X,x).$ Notice that for each $k=0,1,2,\dots$ the map $\pi_{k+1}:A_{k+1}\to A_k$ is a small extension homomorphism in $(\Art)_\C,$ with kernel the principal ideal $(t^{k+1})\sb A_{k+1}.$ We define then a commutative diagram
\begin{equation}\l{concentrate}\begin{tikzcd}[column sep=large]
D(A_{k+1})\ar[r,"D(\pi_{k+1})"]\ar[d] &D(A_k)\ar[r,"\al"]\ar[d] &T^2(X)\otimes_\C(t^{k+1})\ar[d,"\be"]\\
D_\loc(A_{k+1})\ar[r,"D_\loc(\pi_{k+1})"]& D_\loc(A_k)\ar[r,"\al_\loc"]& T^2_\loc(X)\otimes_\C(t^{k+1}).
\end{tikzcd}\end{equation}
where $\al$ is the obstruction map of $(T^2(X),\pi_{k+1})$ and $\al_\loc$ that of $(T^2_\loc(X),\pi_{k+1}).$ Define the leftmost vertical map $D(A_{k+1})\to D_\loc(A_k)$ by taking an element of $D(A_{k+1}),$ representing it by a deformation $X_k/A_k$ of $X,$ taking the germ at $X^\sing$ of $X_k/A_k,$ and taking its isomorphism class (which is independent of the choice of the representative $X_k/A_k$). Define in the same way the middle vertical map $D(A_k)\to D_\loc(A_k).$ We define now the rightmost vertical map $\be:T^2(X)\otimes_\C(\ep)\to T^2_\loc(X)\otimes_\C\ep.$ Denote by $f:X^\sing\to X$ the inclusion map and notice that there is a natural isomorphism $T^2_\loc(X)\cong \Ext^2_{\O_X}(L_X,f_*f^*\O_X).$ The natural map $\id\to f_*f^*$ induces therefore a map $\Ext^2_{\O_X}(L_X,\O_X)\to \Ext^2_{\O_X}(L_X,f_*f^*\O_X).$ But the domain $\Ext^2_{\O_X}(L_X,\O_X)$ of the latter map is exactly $T^2(X)$ and hence we get a map $T^2(X)\to T^2_\loc(X).$ Tensoring this with $(\ep)$ we get a map $T^2(X)\otimes_\C(\ep)\to T^2_\loc(X)\otimes_\C\ep$ which we call $\be.$

We say that the obstruction to deforming $X$ {\it concentrates upon} its singularities if for each $k=0,1,2,\dots$ the map $\be|_{\im\al}:\im\al\to T^2_\loc(X)\otimes_\C(\ep)$ is injective.
\end{dfn}
\begin{rmk}
Suppose that the obstruction to deforming $X$ concentrates upon its singularities and that every $(X,x)$ has unobstructed deformations. We show then that the whole $X$ has unobstructed deformations. Let $k\ge0$ be an integer and take any element $\xi\in D(A_k).$ As the deformations of each $(X,x)$ are unobstructed, in \eq{concentrate} the left bottom horizontal map $D_\loc(\pi_k):D_\loc(A_{k+1})\to D_\loc(A_k)$ is surjective. In particular, the composite map $D(A_k)\to D_\loc(A_k)\to T^2_\loc(X)\otimes_\C(t^{k+1})$ maps $\xi$ to zero. But by Definition \ref{dfn: concentrate} the map $\be|_{\im\al}$ is injective, so in \eq{concentrate} the right top horizontal map $D(A_k)\to T^2(X)\otimes_\C(t^{k+1})$ maps $\xi$ to zero. Thus $\xi$ may be lifted to $D(A_{k+1})$ in \eq{concentrate}. The map $D(\pi_{k+1}):D(A_{k+1})\to D(A_k)$ is therefore surjective. As this holds for every $k=0,1,2,\dots$ the deformations of $X$ are unobstructed.   
\end{rmk}

The condition in Definition \ref{dfn: concentrate} is rather hard to verify as it is. Following \cite[Theorem 2.2]{Gross} therefore we make
\begin{dfn}\l{dfn: Gross}
Let $X$ be a compact reduced complex space whose singularities are Cohen--Macaulay. Let $k\ge1$ be an integer, $X_k/A_k$ a deformation of $X,$ and $X_{k-1}/A_{k-1}$ the deformation of $X$ defined by $X_{k-1}:=\spec A_{k-1}\times_{\spec A_k}X_k.$ Suppose that if we denote by $\io:X^\reg\to X$ the inclusion of the regular locus then
\e\l{CY}
\text{$\io_*\Om^n_{X_k/A_k}$ is a rank-one free $\O_{X_k}$ module.}
\e
Consider the deformation functor of Example \ref{ex: D} and its $T^1$ modules. We define then an $A_k$ module exact sequence
\begin{equation}\l{T1T2}
T^1(X_k/A_k)\to T^1(X_{k-1}/A_{k-1})\to \Ext^2_{\O_X}(\Om^1_X,\O_X).
\end{equation}
Consider the $A_k$ module short exact sequence $0\to A_{k-1}\to A_k\to\C\to0$ where the first arrow is the multiplication by $t$ modulo ideals and the second arrow the natural projection. Tensoring these with the sheaf $\Om^1_{X_k/A_k}$ we get an exact sequence $\Om^1_{X_{k-1}/A_{k-1}}\to \Om^1_{X_k/A_k}\to \Om^1_X\to0.$ As $\Om^1_{X_k/A_k}$ is flat over $X^\reg$ the kernel of the first arrow, which we call $\ta,$ is supported on $X^\sing.$ Since $X^\sing$ has dimension $\le n-2$ it follows that $H^{n-1}(X,\ker\ta)=H^n(X,\ker\ta)=0$ and hence that the natural map $ H^{n-1}(X,\Om^1_{X_{k-1}/A_{k-1}})\to H^{n-1}(X,\im\ta)$ is an isomorphism. Using this we get an $A_k$ module exact sequence 
\e\l{T1T21}  H^{n-2}(X,\Om^1_X)\to H^{n-1}(X,\Om^1_{X_{k-1}/A_{k-1}})\to H^{n-1}(X,\Om^1_{X_k/A_k}).\e
Since $A_k$ is a Gorenstein ring and an injective $A_k$ module it follows that the functor $\hom_{A_k}(\bt,A_k)$ is exact; and in particular, taking the dual of \eq{T1T21} we get an exact sequence
\ea\l{T1T22} \hom_{A_k}(H^{n-1}(X,\Om^1_{X_k/A_k}),A_k)\to \hom_{A_k}(H^{n-1}(X,\Om^1_{X_{k-1}/A_{k-1}}),A_k)\\
\to\hom_{A_k}(H^{n-2}(X,\Om^1_X),A_k).\ea
Note now that for $M$ an $A_j$ module with $j<k$ there is a natural isomorphism $\hom_{A_k}(M,A_k)\cong\hom_{A_j}(M,A_j).$ The sequence \eq{T1T22} may then be re-written as
\ea\l{T1T23}
\!\!\!\!\!
\hom_{A_k}(H^{n-1}(X,\Om^1_{X_k/A_k}),A_k)\to \hom_{A_{k-1}}(H^{n-1}(X,\Om^1_{X_{k-1}/A_{k-1}}),A_{k-1})\\
\to\hom_\C(H^{n-2}(X,\Om^1_X),\C).\ea
The condition \eq{CY} implies now that the relative canonical sheaves of $X_0,\dots,X_k$ are all free of rank one. The three $A_k$ modules of \eq{T1T23} are then isomorphic by Serre duality to those three of \eq{T1T2}. Using this we define the arrows of \eq{T1T2} to be those of \eq{T1T23}.

Suppose now that $X^\sing$ is isolated. For $x\in X^\sing,$ if we denote by $f:\{x\}\to X$ the inclusion map then using the natural map $\id\to f_*f^*$ we get for $k=0,1,2,\dots$ a map $\Ext^1_{\O_{X_k,x}}(\Om^1_{X_k/A_k},\O_{X_k})\to \Ext^1_{\O_{X_k,x}}((\Om^1_{X_k/A_k})_x,\O_{X_k,x})$ or equivalently a map $T^1(X_k/S_k)\to T^1_\loc(X_k/A_k)$ where the latter denotes the $T^1$ module for the deformation functor $D_\loc.$ There is also a map from $\Ext^2_{\O_{X_k,x}}(\Om^1_{X_k/A_k},\O_{X_k})$ to $\Ext^2_{\O_{X_k,x}}((\Om^1_{X_k/A_k})_x,\O_{X_k,x}).$ There is now a commutative diagram
\begin{equation}\l{Gross}\begin{tikzcd}
\!\!\!\!\!\!\!\!T^1(X_k/A_k)\ar[r]\ar[d] &T^1(X_{k-1}/A_{k-1})\ar[r,"\ga"]\ar[d] &\Ext^2_{\O_X}(\Om^1_X,\O_X)\ar[d,"\de"]\\
\!\!\!\!\!\!\!\!T_\loc^1(X_k/A_k)\ar[r]& T_\loc^1(X_{k-1}/A_{k-1})\ar[r]&\!\!\! \displaystyle\bop_{x\in X^\sing}\!\!\!\Ext^2_{\O_{X,x}}(\Om^1_{X,x},\O_{X,x}).
\end{tikzcd}\end{equation}
\end{dfn}

The following is a more complex version of $T^1$ lift theorems. Although this is known to experts, we give it a proof for the sake of clarity; in \cite[Theorem 2.2]{Nam0}, for instance, the result is stated without proof.
\begin{lem}\l{lem: concentrate}
Let $X$ be a compact reduced complex space whose singularities are Cohen--Macaulay and isolated; the latter implies that Definition \ref{dfn: concentrate} makes sense. Let \eq{CY} hold so that Definition \ref{dfn: Gross} makes sense, and suppose that in \eq{Gross} the map $\de|_{\im\ga}:\im\ga\to \bop_{x\in X^\sing}\Ext^2_{\O_{X,x}}(\Om^1_{X,x},\O_{X,x})$ is injective. The obstruction to deforming $X$ then concentrates upon its singularities.
\end{lem}
\begin{proof}
Recall that there is a $\C$-algebra homomorphism $\th_k:A_k\to A_{k-1}[\ep]$ defined by $t\mapsto t+\ep$ module ideals. Consider also the $\C$-algebra homomorphism $A_k[\ep]\to A_{k-1}[\ep]\times_{A_{k-1}}A_k$ made of the projections $A_k[\ep]\to A_{k-1}[\ep]$ and $A_k[\ep]\to A_k.$ There is then an $A_{k+1}$ module commutative diagram
\begin{equation}\begin{tikzcd}
0\ar[r]&(t^{k+1})\ar[r]\ar[d,"\cong"] &A_{k+1}\ar[r,"\pi_{k+1}"]\ar[d,"\th_{k+1}"]& A_k\ar[r]\ar[d,"\th_k\times\id"]&0\\ 
0\ar[r]&(t^k\ep)\ar[r] &A_k[\ep]\ar[r]& A_{k-1}[\ep]\times_{A_{k-1}}A_k\ar[r]&0
\end{tikzcd}\end{equation}
whose rows are small extensions in $(\Art)_\C.$ The leftmost vertical map $(t^{k+1})\to (t^k\ep)$ is a $\C$-vector space isomorphism which map $t^{k+1}$ to $(k+1)t^k\ep.$ By the defining property of the obstruction spaces there is a commutative diagram
\begin{equation}\l{KR0}\begin{tikzcd}
D(A_{k+1})\ar[r]\ar[d] & D(A_k)\ar[r]\ar[d] &T^2(X)\otimes (t^{k+1})\ar[d,"\cong"]\\
D(A_k[\ep])\ar[r]\ar[d] & D(A_k\times_{A_{k-1}}A_{k-1}[\ep])\ar[r]\ar[d] &T^2(X)\otimes (t^k\ep)\ar[d]\\
D_\loc(A_k[\ep])\ar[r] & D_\loc(A_k\times_{A_{k-1}}A_{k-1}[\ep])\ar[r] &T^2_\loc(X)\otimes (t^k\ep).
\end{tikzcd}\end{equation}
Let $\xi\in D(A_k)$ be any element with $\be\cm\al(\xi)=0.$ Denote by $\et\in D(A_k\times_{A_{k-1}}A_{k-1}[\ep]))$ its image under $D(A_k)\to D(A_k\times_{A_{k-1}}A_{k-1}[\ep]).$ Let $X_k/A_k$ represent the image of $\et$ under $D(A_k\times_{A_{k-1}}A_{k-1}[\ep])\to D(A_k)$ and let $Y_{k-1}/A_{k-1}[\ep]$ represent the image of $\et$ under $D(A_k\times_{A_{k-1}}A_{k-1}[\ep])\to D(A_{k-1}[\ep]).$ Then $X_{k-1}:=\spec A_{k-1}\times_{\spec A_k}X_k$ is isomorphic as an $A_{k-1}$ ringed space to $\spec A_{k-1}\times_{\spec A_{k-1}[\ep]} Y_{k-1}.$ We have thus an element $[Y_{k-1}]\in T^1(X_{k-1}/A_{k-1})$ represented by $Y_{k-1}.$ There is on the other hsnd a commutative diagram
\begin{equation}\l{KR1}\begin{tikzcd}[row sep=tiny, column sep=tiny]
D(A_k)\ar[rr,"\al"]\ar[rd]\ar[dd]&&T^2(X)\otimes_\C(t^{k+1})\ar[dd,"\be" near start]\ar[rd]&\\
&D(A_k\times_{A_{k-1}}A_k[\ep])\ar[rr, crossing over]&&T^2(X)\otimes_\C(t^k\ep)\ar[dd]\\
D_\loc(A_k)\ar[rr]\ar[rd]&&T^2_\loc(X)\otimes_\C(t^{k+1})\ar[rd]&\\
&D_\loc(A_k\times_{A_{k-1}}A_k[\ep])\ar[rr]\ar[from=uu, crossing over]&&T^2_\loc(X)\otimes_\C(t^k\ep).
\end{tikzcd}\end{equation}
Since $\be\cm\al(\xi)=0$ it follows that in this commutative diagram the composite map $D(A_k)\to T^2_\loc(X)\otimes_\C(t^k\ep)$ maps $\xi\in D(A_k)$ to $0\in T^2_\loc(X)\otimes_\C(t^k\ep).$ Denote by $\ze\in D_\loc(A_k\times_{A_{k-1}}A_{k-1}[\ep])$ the image in \eq{KR1} of $\xi\in D(A_k).$ This appears also in \eq{KR0}. Since $\ze$ maps in \eq{KR1} to $0\in T^2_\loc(X)\otimes_\C(t^k\ep)$ it follows that so does $\ze$ in \eq{KR0}. In \eq{KR0} therefore $\ze$ lifts to some element $\om\in D_\loc(A_k[\ep]).$ 

We look now at the commutative diagram \eq{Gross}. The image of $[Y_{k-1}]\in T^1(X_{k-1}/A_{k-1})$ maps to an element of $T^1_\loc(X_{k-1}/A_{k-1})$ which is the image of $\om\in D_\loc(A_k[\ep]).$ Thus $\de\cm\ga[Y_{k-1}]=0.$ The current hypothesis (that of Lemma \ref{lem: concentrate}) implies therefore $\ga[Y_{k-1}]=0.$ So $[Y_{k-1}]$ lifts to some element of $T^1(X_k/A_k).$ In \eq{KR0} accordingly $\et\in D(A_k\times_{A_{k-1}}A_{k-1}[\ep])$ lifts to some element of $D(A_k[\ep]).$ So $\et$ maps to zero under $D(A_k\times_{A_{k-1}}A_{k-1}[\ep])\to T^2(X)\otimes_\C(t^k\ep).$ But the vertical map $T^2(X)\otimes_\C(t^{k+1})\to T^2(X)\otimes_\C(t^k\ep)$ is an isomorphism, and $\xi\in D(A_k)$ therefore maps to $0\in  T^2(X)\otimes_\C(t^{k+1})$ as we have to prove.
\end{proof}
\begin{rmk}
Obstruction maps are in general hard to compute as they are. On the other hand, $T^1$ and Ext modules are less functorial but easier to compute. The effect of Lemma \ref{lem: concentrate} is that computing Ext modules is sufficient for our current purpose. Something similar is done for instance in \cite[Proposition 2.6]{San}.
\end{rmk}

\section{Relative Differential Forms}\l{sect: rel forms}
The next four definitions, Definitions \ref{dfn: X times A}--\ref{dfn: glue}, are devoted to defining sheaves of holomorphic forms, $C^\iy$ forms and real analytic forms.
\begin{dfn}\l{dfn: X times A}
Suppose first that $X$ is a complex manifold with structure sheaf $\O_X.$ Denote by $C^\iy_X$ the sheaf on $X$ of $\C$-valued $C^\iy$ functions, which is therefore a $\C$-algebra sheaf. Denote by $C^\om_X\sb C^\iy_X$ the $\C$-algebra subsheaf on $X$ made from $\C$-valued real analytic functions. Denote by $\Om^\bt_X$ the $\Z$-graded $\O_X$ module sheaf on $X$ of holomorphic forms, and by $\La^\bt_X$ the $\Z$-graded $C^\iy_X$ module sheaf on $X$ of $C^\iy$ forms. There is also a real analytic version of $\La^\bt_X$ for which however we do not introduce any particular symbol (because we shall not have to use it directly). 

Both $\Om^\bt_X$ and $\La^\bt_X$ are sheaves of differential graded algebras over $\C,$ equipped with the de Rham differentials $\d_X:\Om^\bt_X\to\Om^{\bt+1}_X$ and $\d_X:\La^\bt_X\to\La^{\bt+1}_X,$ together with the wedge product maps $\wedge:\Om^\bt_X\otimes_{\O_X}\Om^\bt_X\to \Om^\bt_X$ and $\wedge:\La^\bt_X\otimes_{C^\iy_X}\La^\bt_X\to \La^\bt_X.$ For $p,q\in \Z$ denote by $\La^{pq}_X$ the sheaf on $X$ of $C^\iy$ $(p,q)$ forms so that for $r\in\Z$ we have $\La^r_X=\bop_{p+q=r}\La^{pq}_X.$ For $p,q\in\Z$ the differential $\d_X:\La^{p+q}_X\to\La^{p+q+1}_X$ induces two $\C$-vector space sheaf homomorphisms $\La^{pq}_X\to \La^{p+1\, q}_X$ and $\La^{pq}_X\to \La^{p\, q+1}_X$ which we denote by $\bd_X$ and $\db_X.$ Since $\d_X^2=0$ it follows that $\bd_X^2=\db_X^2=\bd_X\db_X+\db_X\bd_X=0.$ There are also real analytic versions of $(\La^\bt_X,\d_X,\wedge)$ and $\La^{\bt\bt}_X.$

Suppose now that $X$ is embedded as an open set in some $\C^n.$ Let $A$ be an Artin local $\C$-algebra and define an $A$-ringed space $\cX:=(X,\O_\cX)$ by $\O_\cX:=\O_X\otimes_\C A.$ Put $C^\iy_\cX:=C^\iy_X\otimes_\C A$ and $C^\om_\cX:=C^\om_X\otimes_\C A,$ which are also $A$-algebra sheaves on $X.$ There is on $X$ a $\Z$-graded $\O_\cX$ module sheaf $\Om^\bt_{\cX/A}$ defined by $\Om^p_{\cX/A}:=\Om^p_X\otimes_\C A$ for $p\in\Z.$ There is on $X$ a $\Z$-graded $C^\iy_\cX$ module sheaf $\La^\bt_{\cX/A}$ defined by $\La^p_{\cX/A}:=\La^p_X\otimes_\C A$ for $p\in\Z.$ Define a degree-one $A$-module sheaf homomorphism $\d_{\cX/A}:\Om^\bt_{\cX/A}\to \Om^{\bt+1}_{\cX/A}$ by $\d_{\cX/A}:=\d_X\otimes\id_A.$ Define by the same formula a degree-one $A$-module sheaf homomorphism $\d_{\cX/A}:\La^\bt_{\cX/A}\to \La^{\bt+1}_{\cX/A}.$ In either case $(\d_{\cX/A})^2=0;$ that is, $\d_{\cX/A}$ is a differential. There are also for $p,q\in\Z$ an $\O_\cX$ module homomorphism $\wedge:\Om^p_{\cX/A}\otimes_{\O_\cX}\Om^q_{\cX/A}\to \Om^{p+q}_{\cX/A}$ and  a $C^\iy_\cX$ module homomorphism $\wedge:\Om^p_{\cX/A}\otimes_{\O_\cX}\Om^q_{\cX/A}\to \Om^{p+q}_{\cX/A}.$ The triples $(\Om^\bt_{\cX/A},\d_{\cX/A},\wedge)$ and $(\La^\bt_{\cX/A},\d_{\cX/A},\wedge)$ are both differential graded $A$-algebra sheaves. For $p,q\in \Z$ put $\La^{pq}_{\cX/A}:=\La^{pq}_X\otimes_\C A$ so that for $r\in\Z$ we have $\La^r_{\cX/A}=\bop_{p+q=r}\La^{pq}_{\cX/A}.$ For $p,q\in\Z$ the differential $\d_{\cX/A}:\La^{p+q}_{\cX/A}\to\La^{p+q+1}_{\cX/A}$ induces two $A$-module sheaf homomorphisms $\bd_{\cX/A}:\La^{pq}_{\cX/A}\to \La^{p+1\, q}_{\cX/A}$ and $\db_{\cX/A}:\La^{pq}_{\cX/A}\to \La^{p\, q+1}_{\cX/A}$ which we denote by $\bd_X$ and $\db_X.$ Since $\d_{\cX/A}^2=0$ it follows that $\bd_{\cX/A}^2=\db_{\cX/A}^2=\bd_{\cX/A}\db_{\cX/A}+\db_{\cX/A}\bd_{\cX/A}=0.$ There are also real analytic versions of $(\La^\bt_{\cX/A},\d_{\cX/A},\wedge)$ and $\La^{\bt\bt}_{\cX/A}.$
\end{dfn}

We define next the model sheaf of holomorphic forms.
\begin{dfn}\l{dfn: hol forms}
Let $X$ be a complex space embedded in an open set $Y\sb\C^n,$ and $\cX$ an $A$-ringed space embedded in $\cY:=Y\times\spec A$ by an ideal sheaf $\cI\sb \O_\cY;$ that is, if we denote by $\cQ$ the quotient sheaf of $\cI\sb \O_\cY$ then $\O_\cX:=\cQ|_X.$ We define on $X$ a $\Z$-graded $\O_\cX$ module sheaf $\Om^\bt_{\cX/A}$ equipped with a degree-one $A$-module sheaf homomorphism $\d_{\cX/A}:\Om^\bt_{\cX/A}\to \Om^{\bt+1}_{\cX/A}$ such that $(\d_{\cX/A})^2=0.$ We do this by an induction on $p.$ For $p<0$ set $\Om^p_{\cX/A}=0$ and the differential $\d_{\cX/A}:\Om^{p-1}_{\cX/A}\to\Om^p_{\cX/A}$ must vanish. For $p\ge0$ consider the $\O_X$ submodule sheaf $\d_{\cY/A}\cI\wedge\Om^{p-1}_{\cY/A}+\cI\Om^p_{\cY/A}\sb \Om^p_{\cY/A}$ whose quotient sheaf we denote by $\cQ^p.$ Set $\Om^p_{\cX/A}:=\cQ^p|_X.$ The differential $\d_{\cY/A}:\Om^{p-1}_{\cY/A}\to \Om^p_{\cY/A}$ induces then an $A$-module sheaf homomorphism $\d_{\cX/A}:\Om^{p-1}_{\cX/A}\to \Om^p_{\cX/A}$ with $(\d_{\cX/A})^2=0.$ Now for $p,q\in\Z$ the wedge product map $\wedge:\Om^p_{\cY/A}\otimes_{\O_\cY}\Om^q_{\cY/A}\to \Om^{p+q}_{\cY/A}$ induces an $\O_\cX$ module homomorphism $\wedge:\Om^p_{\cX/A}\otimes_{\O_\cX}\Om^q_{\cX/A}\to \Om^{p+q}_{\cX/A}$ which satisfies the Leibniz rule with respect to $\d_{\cX/A}.$ The triple $(\Om^\bt_{\cX/A},\d_{\cX/A},\wedge)$ is thus a sheaf of differential graded $A$-algebras.
\end{dfn}

We define also the model sheaf of $C^\iy$ forms.
\begin{dfn}\l{dfn: C^iy forms}
Let $A$ be an Artin local $\R$-algebra and recall from Lemma \ref{lem: complexify} that $B:=A\otimes_\R\C$ is an Artin local $\C$-algebra. Let $X$ be a complex space embedded in an open set $Y\sb\C^n,$ and $\cX$ a $B$-ringed space embedded in $\cY:=Y\times\spec B$ by an ideal sheaf $\cI\sb \O_\cY.$ Put $\La^\bt_{\cY/A}:=\La^\bt_{\cY/B}=\La^\bt_Y\otimes_\C B,$ which we identify naturally with $\La^\bt_X\otimes_\R A.$ The complex conjugate map $\La^\bt_Y\to \La^\bt_Y$ and the identity map $A\to A$ induce then an $\R$-algebra sheaf homomorphism $\La^\bt_{\cY/A}\to \La^\bt_{\cY/A}$ which we call the {\it complex conjugate map.} Denote by $\ov\cI$ the image under this of $\cI\sb \O_\cY\sb C^\iy_\cY=\La^0_{\cY/A}.$ There is then an ideal sheaf $\cJ:=\cI+\ov\cI\sb C^\iy_\cY$ whose quotient we denote by $\cQ.$ The restriction $C^\iy_\cX:=\cQ|_X$ defines on $X$ a $\C$-algebra sheaf. 

We define on $X$ a $\Z$-graded $C^\iy_\cX$ module sheaf $\La^\bt_{\cX/A}$ equipped with a degree-one $B$-module sheaf homomorphism $\d_{\cX/A}:\La^\bt_{\cX/B}\to \La^{\bt+1}_{\cX/B}$ such that $(\d_{\cX/A})^2=0.$ We do this in the same way as in Definition \ref{dfn: hol forms} with $\La^\bt_{\cY/B}$ in place of $\Om^\bt_{\cY/A}$ and with $\cJ$ in place of $\cI.$ This produces at the same time for $p,q\in\Z$ the wedge product map $\wedge:\La^p_{\cX/A}\otimes_{C^\iy_\cX}\La^q_{\cX/A}\to \La^{p+q}_{\cX/A}$ is defined in the same way. The triple $(\La^\bt_{\cX/A},\d_{\cX/A},\wedge)$ is thus a sheaf of differential graded $B$-algebras.

For $p,q\in \Z$ denote by $\La^{pq}_{\cX/A}$ the image of $\La^{pq}_{\cY/A}|_X$ under the projection $\La^{p+q}_{\cY/A}|_X\to \La^{p+q}_{\cX/A}.$ Each $\La^{pq}_{\cX/A}$ is then a $C^\iy_\cX$ submodule of $\La^{p+q}_{\cX/A}$ so that for $r\in\Z$ we have $\La^r_{\cX/A}=\bop_{p+q=r}\La^{pq}_{\cX/A}.$ For $p,q\in\Z$ the differential $\d_{\cX/A}:\La^{p+q}_{\cX/A}\to\La^{p+q+1}_{\cX/A}$ induces two $B$-module sheaf homomorphisms $\bd_{\cX/A}:\La^{pq}_{\cX/A}\to \La^{p+1\, q}_{\cX/A}$ and $\db_{\cX/A}:\La^{pq}_{\cX/A}\to \La^{p\, q+1}_{\cX/A}$ which we denote by $\bd_X$ and $\db_X.$ Since $\d_{\cX/A}^2=0$ it follows that $\bd_{\cX/A}^2=\db_{\cX/A}^2=\bd_{\cX/A}\db_{\cX/A}+\db_{\cX/A}\bd_{\cX/A}=0.$ Also for $p,q\in\Z$ the complex conjugate map $\La^{p+q}_{\cY/A}\to\La^{p+q}_{\cY/A}$ induces an $\R$-algebra sheaf homomorphism $\La^{pq}_{\cX/A}\to\La^{qp}_{\cX/A}$ which we call the {\it complex conjugate map.} 

There are also real analytic versions of $(\La^\bt_{\cX/A},\d_{\cX/A},\wedge),$ $\La^{\bt\bt}_{\cX/A}$ and their complex conjugate maps.
\end{dfn}

We finally glue together the local models above.
\begin{dfn}\l{dfn: glue}
Let $X$ be a complex space, $A$ an Artin local $\R$-algebra and $\cX/A$ a deformation of $X.$ Put $B:=A\otimes_\R\C$ and recall from Definition \ref{dfn: real para} that $\cX/A$ is a deformation $\cX/B$ of $X.$ Choose an open cover $X=U\cup V\cup\dots$ such that each $\cU:=(U,\O_\cX|_U)$ is embedded as a $B$-ringed space into $Y\times \spec B$ for some open set $Y\sb\C^n.$ These $U,V,\dots$ exist by \cite[Chapter 2, Proposition 1.5]{GLS}. Applying Definition \ref{dfn: hol forms} to $\cU,\cV,\dots$ we get on $U,V,\dots$ the sheaves $\Om^\bt_{\cU/A},\Om^\bt_{\cV/A},\dots,$ which we can glue together. The result is an $\O_\cX$ module sheaf on $X$ which we denote by $\Om^\bt_{\cX/A}.$ The gluing process defines also a differential and a wedge product map, which we denote by $\d_{\cX/A}$ and $\wedge$ respectively. The triple $(\Om^\bt_{\cX/A},\d_{\cX/A},\wedge)$ is thus a sheaf on $X$ of differential graded $A$-algebras. We define on $X$ another differential graded $A$-algebra sheaf $(\La^\bt_{\cX/A},\d_{\cX/A},\wedge)$ in the same way with Definition \ref{dfn: C^iy forms} in place of Definition \ref{dfn: hol forms}. 

For $p,q\in \Z$ define a $C^\iy_\cX$ module sheaf $\La^{pq}_{\cX/A}$ by gluing together the local models $\La^{pq}_{\cU/A},\La^{pq}_{\cV/A},\dots$ corresponding to $U,V,\dots$ respectively. Each $\La^{pq}_{\cX/A}$ is then a $C^\iy_\cX$ submodule of $\La^{p+q}_{\cX/A}$ so that for $r\in\Z$ we have $\La^r_{\cX/A}=\bop_{p+q=r}\La^{pq}_{\cX/A}.$ For $p,q\in\Z$ the differential $\d_{\cX/A}:\La^{p+q}_{\cX/A}\to\La^{p+q+1}_{\cX/A}$ induces two $B$-module sheaf homomorphisms $\bd_{\cX/A}:\La^{pq}_{\cX/A}\to \La^{p+1\, q}_{\cX/A}$ and $\db_{\cX/A}:\La^{pq}_{\cX/A}\to \La^{p\, q+1}_{\cX/A}$ which we denote by $\bd_X$ and $\db_X.$ Since $\d_{\cX/A}^2=0$ it follows that $\bd_{\cX/A}^2=\db_{\cX/A}^2=\bd_{\cX/A}\db_{\cX/A}+\db_{\cX/A}\bd_{\cX/A}=0.$ 

Also for $p,q\in\Z$ the complex conjugate maps for the local models are glued up into an $\R$-algebra sheaf homomorphism $\La^{pq}_{\cX/A}\to\La^{qp}_{\cX/A}$ which we call the {\it complex conjugate map.} For $p\in\Z$ denote by $\Re\La^{pp}_{\cX/A}\sb\La^p_{\cX/A}$ the subsheaf invariant under the complex conjugate map $\La^p_{\cX/A}\to\La^p_{\cX/A}.$

There are also real analytic versions of $(\La^\bt_{\cX/A},\d_{\cX/A},\wedge)$ and $\La^{\bt\bt}_{\cX/A}.$ 

For $A=\R$ we write $X=\cX/\R$ to define $(\Om^\bt_X,\d_X,\wedge),$ $(\La^\bt_X,\d_X,\wedge)$ and $\La^{\bt\bt}_X.$ We also write $\d=\d_X,$ $\bd=\bd_X$ and $\db=\db_X,$ omitting the index $X.$ A $C^\iy$ function $X\to\R$ means a section of $\Re\La^{00}_X=\Re C^\iy_X.$
\end{dfn}
\begin{rmk}
For $A=\R$ the definitions above, Definitions \ref{dfn: X times A}--\ref{dfn: glue}, are equivalent to those of \cite[\S1.1]{Fuj1}. There is another way of making the same definitions, which is to use the diagonal map $X\to X\times_{\spec A}X$ as in \cite[\S1]{Bing}. 
\end{rmk}

We write more explicitly the sheaves $\Om^\bt_{\cX/A}$ and $\La^\bt_{\cX/A}$ for $X$ a complex manifold.
\begin{rmk}
Let $X$ be a complex manifold, $A$ an Artin local $\R$-algebra and $\cX/A$ a deformation of $X.$ Put again $B:=A\otimes_\R\C$ and recall now from \cite[Theorem 3.21]{Fisch} that there exists an open cover $U\cup V\cup\dots=X$ such that each $(U,\O_\cX|_U)$ is isomorphic as a deformation of $U$ to the trivial deformation $U\times \spec B.$ The sheaf $\Om^p_{\cX/B}$ is then defined by gluing together the local models $\Om^p_U\otimes_\R A,\Om^p_V\otimes_\R A,\dots$ for $U,V,\dots\sb X.$ The sheaf $\La^p_{\cX/A}$ is defined by gluing together the local models $\La^p_U\otimes_\R A,\La^p_U\otimes_\R A,\dots$ for $U,V\dots\sb X.$

These expressions imply that we can use the ordinary Dolbealt lemma for the complex manifold $X;$ that is, if $X$ is of complex dimension $n$ then for $p=0,1,2,\dots$ the sequence
\e\l{Dol} 0\to \Om^p_{\cX/A}\xrightarrow{\db_{\cX/A}} \La^{p0}_{\cX/A}\xrightarrow{\db_{\cX/A}} \dots\xrightarrow{\db_{\cX/A}} \La^{pn}_{\cX/A}\to0\e
is exact. Getting rid of the first non-zero term $\Om^p_{\cX/A}$ and applying the global section functor $\Ga$ we get a complex
\e\l{Dol2} 0\to \Ga( \La^{p0}_{\cX/A})\xrightarrow{\db_{\cX/A}} \dots\xrightarrow{\db_{\cX/A}} \Ga(\La^{pn}_{\cX/A})\to0.\e
Since each $\La^{pq}_{\cX/A}$ is a fine sheaf (admitting partitions of unity) it follows that for $q=0,1,2,\dots$ the sheaf cohomology group $H^q(X,\Om^p_{\cX/A})$ is isomorphic to the $q^{\rm th}$ cohomology group of \eq{Dol2}. For $A=\R$ this reduces to the ordinary Dolbeault isomorphism.
\end{rmk}

We make now the definition of K\"ahler forms on infinitesimal deformations.
\begin{dfn}\l{dfn: Kahl}
Recall that a K\"ahler form on a complex space $X$ is an element $\om\in\Ga(\Re\La^{11}_X)$ for which there exist an open cover $U\cup V\cup\dots=X$ and a corresponding family $(\ph_U:U\to\R)_U$ of $C^\iy$ strictly plurisubharmonic functions such that for each $U$ we have $\om|_U=i\bd\db \ph_U.$

Let $A$ be an Artin local $\R$-algebra and $\cX/A$ a deformation of $X.$ Then a {\it K\"ahler form} on $\cX/A$ is a section $\om_{\cX/A}\in\Ga(\Re\La^{11}_{\cX/A})$ for which there exist an open cover $U\cup V\cup\dots=X$ and a corresponding family $(\ph_U\in \Re C^\iy_\cX(U))_U$ such that for each $U$ we have $\om_{\cX/A}|_U=i\bd_{\cX/A}\db_{\cX/A}\ph_U$ and the restriction map $\Re C^\iy_\cX(U)\to \Re C^\iy_X(U)=C^\iy(U,\R)$ maps $\ph_U$ to some strictly plurisubharmonic function.
\end{dfn}
\begin{rmk}
Denote by $\cK_\cX$ the cokernel of the map $\O_\cX\to \Re C^\iy_\cX$ which maps a local section $f$ to $\frac12(f+\bar f).$ The family $(\ph_U\in\Re C^\iy_\cX(U))$ corresponding to a K\"ahler form on $\cX/A$ is then a section of $\cK_\cX.$ Conversely, every section of $\cK_X$ is obtained from such a family except that the restrictions to $X$ need not be strictly plurisubharmonic. Put $\cK_X:=\cK_\cX$ when $A=\R.$
\end{rmk}

We state now the key result we shall need about K\"ahler forms. Notice that if $X$ is a complex space then the inclusion of its constant sheaf $\ul\R$ into the structure sheaf $\O_X$ induces an $\R$-linear map $H^2(X,\R)\to H^2(X,\O_X).$
\begin{thm}[Theorem 6.3 of \cite{Bing}]\l{thm: Bing}
Let $X$ be a K\"ahler space for which the map $H^2(X,\R)\to H^2(X,\O_X)$ is surjective. Then for every Artin local $\R$-algebra $A$ and every deformation $\cX/A$ of $X$ there exist K\"ahler forms on $\cX/A.$
\end{thm}
\begin{rmk}
Bingener \cite{Bing} deals not only with the infinitesimal deformations as above but also with the deformations over a complex space germ $(S,o)$ of positive dimension. But we shall not have to do so for our purpose, for which the weaker statement above will do.

We give now a direct proof of Theorem \ref{thm: Bing} because this is simpler than the original one. Using the hypothesis and the $\R$-vector space sheaf isomorphism $i:\O_X\to\O_X$ which multiplies by $i=\sqrt{-1}$ we see that the inclusion $i\ul\R\to\O_X$ induces also a surjective map $H^2(X,i\R)\to H^2(X,\O_X).$ That is, the corresponding map $H^2(X,(\O_X/i\ul\R)) \to H^3(X,i\ul\R)$ is injective. Suppose now that $0\to(\ep)\to A\to B\to0$ is a small extension in $(\Art)_\R.$ Since the map $\R\cong(\ep)\to A$ is injective it follows that so is the $\R$-linear map $H^3(X,i\R)\to H^3(X,iA).$ Composing this with the injection $H^2(X,(\O_X/i\ul\R)) \to H^3(X,i\ul\R)$ we see that the map $H^2(X,(\O_X/i\ul\R)) \to H^3(X,i\ul A)$ is injective. Using the commutative diagrams,
\begin{equation}\begin{tikzcd}[column sep=small]
\!\!\!\!\!\!0\ar[r] &i\ul\R\ar[r]\ar[d]& \O_X\ar[r]\ar[d] & \O_X/i\ul\R\ar[r]\ar[d]&0 \!\!\!& H^2(X,(\O_X/i\ul\R))\ar[r]\ar[d]& H^3(X,i\R)\ar[d]\\
\!\!\!\!\!\!0\ar[r] &i\ul A\ar[r]& \O_\cX\ar[r] & \O_\cX/i\ul A\ar[r]&0,                     \!\!\!& H^2(X,(\O_\cX/i\ul A)\ar[r]& H^3(X,iA)
\end{tikzcd}\end{equation}
we see that the map $H^2(X,(\O_X/i\ul\R)) \to H^2(X,(\O_\cX/i\ul A))$ is also injective. Introducing now the $B$-ringed space $\cY:=(X,\O_\cX\otimes_A B)$ we get a commuative diagram
\begin{equation}\begin{tikzcd}
&0\ar[d]&0\ar[d]&0\ar[d]&\\
0\ar[r]&\O_X/i\ul\R\ar[r]\ar[d]& \Re C^\iy_X\ar[r]\ar[d]& \cK_X\ar[r]\ar[d]&0\\
0\ar[r]&\O_\cX/i\ul{A}\ar[r]\ar[d]& \Re C^\iy_\cX\ar[r]\ar[d]& \cK_\cX\ar[r]\ar[d]&0\\
0\ar[r]&\O_\cY/i\ul{B}\ar[r]\ar[d]& \Re C^\iy_\cY\ar[r]\ar[d]& \cK_\cY\ar[r]\ar[d]&0\\
&0&0&0.&
\end{tikzcd}\end{equation}
Since $\Re C^\iy_X$ and $\Re C^\iy_\cX$ have vanishing higher cohomology groups, we get isomorphisms $H^1(X,\cK_X)\cong H^2(X,\O_X/if^{-1}\ul\R)$ and $H^1(X,\cK_X)\cong H^2(X,\O_\cX/if^{-1}\ul A).$ The map $H^2(X,(\O_X/i\ul\R)) \to H^2(X,(\O_\cX/i\ul A))$ being injective implies now the map $H^1(X,\cK_X) \to H^1(X,\cK_\cX)$ being injective. The map $H^0(X,\cK_\cX)\to H^0(X,\cK_\cY)$ is accordingly surjective. This means that every K\"ahler form on $\cY/B$ extends to $\cX/A.$ The induction therefore completes the proof. \qed
\end{rmk}

There is a useful criterion for the hypothesis of Theorem \ref{thm: Bing}.
\begin{thm}[Proposition 5 of \cite{Nam2}]\l{thm: Nam}
Let $X$ be a compact normal K\"ahler space whose singularities are rational. The map $H^2(X,\R)\to H^2(X,\O_X)$ is then surjective, so the conclusion of Theorem \ref{thm: Bing} holds. \qed
\end{thm}

We make a definition we will use often in what follows.
\begin{dfn}\l{dfn: relative conifold metrics}
Let $X$ be a compact K\"ahler conifold, $A$ an Artin local $\R$-algebra and $\cX/A$ a locally trivial deformation of $X.$ Then a {\it K\"ahler conifold metric on $\cX/A$} is a K\"ahler form $\om$ on $(X^\reg,\O_\cX|_{X^\reg})$ such that the following holds: for every $x\in X^\sing$ there exist a punctured neighbourhood $U$ of $x\in X$ and some K\"ahler cone metric $g_x$ on $(X,x)$ such that if we denote by $\om_x$ the K\"ahler form of $g_x$ then $\om|_U\in\Ga(\La^{11}_{\cX/A}|_U)$ corresponds to $\om_x\otimes1\in \Ga(\La^{11}_U)\otimes_\R A$ under an isomorphism $\La^{11}_{\cX/A}|_U\cong \La^{11}_U\otimes_\R A$ (which does exist for $U$ small enough, as $\cX/A$ is locally trivial).
\end{dfn}

Using Theorem \ref{thm: Nam} we generalize Corollary \ref{cor: Kahl pot} as follows.
\begin{cor}\l{cor: extending conifold metrics}
Let $X$ be a compact K\"ahler conifold whose singularities are rational. Let $A$ be an Artin local $\R$-algebra and $\cX/A$ a locally trivial deformation of $X.$ Then there exists on $\cX/A$ a K\"ahler conifold metric.
\end{cor}
\begin{proof}
Using Theorem \ref{thm: Nam} choose an open cover $U\cup V\cup\dots =X$ and a corresponding family $(p_U,p_V,\dots)$ which define a K\"ahler form on $\cX/A.$ Put $B:=A\otimes_\R\C.$ For $U$ containing a singular point $x\in X^\sing$ embed the $B$-ringed space $(U,\O_\cX|_U)$ into $\C^m\times \spec B.$ Extend $p_U$ to some open set in $\C^m$ as a $C^\iy$ function with values in $B.$ Put $p':=\displaystyle{p_U-p_U(0)-\sum_{a=1}^m(\frac{\bd p_U}{\bd z_a}(0)z_a+\frac{\bd p_U}{\bd \bar z_a}(0)\bar z_a)}.$ On the other hand, let $\ep r_\la:U\to\R$ be as in Lemma \ref{lem: Kahl pot}. Regard this as a $B$-valued function and as a smooth function on $U\times\spec B.$ Choose also a cut-off function $\ps$ as in the proof of Lemma \ref{lem: Kahl pot}. Define a $C^\iy$ function $q_U:U\times\spec B\to\R$ by 
\e
q_U:=p'+\ep\ph r_\la^2-\ps\Bigl(\frac{r^2}{\de^2}\Bigr)p'
\e
Then $q_U=p_U$ at the points far enough from $x.$ Since $\bd_{\cX/A}\db_{\cX/A} z_1=\dots=\bd_{\cX/A}\db_{\cX/A} z_m=0$ and $\bd_{\cX/A}\db_{\cX/A}\bar z_1=\dots=\bd_{\cX/A}\db_{\cX/A}\bar z_m=0$ it follows that $i\bd_{\cX/A}\db_{\cX/A}q_U=i\bd_{\cX/A}\db_{\cX/A}p'=i\bd_{\cX/A}\db_{\cX/A}p_U.$ We can therefore glue together $q_U$ and the other K\"ahler potentials. That is, for $U$ not intersecting $X^\sing,$ set $q_U:=p_U.$ The family $(q_U,q_V,\dots)$ defines then a section over $X^\reg$ of the sheaf $\cK_{\cX/A}.$ Its image under the restriction map $\cK_{\cX/A}\to \cK_X$ defines a K\"ahler conifold metric on $X^\reg,$ as in the proof of Lemma \ref{lem: Kahl pot}. The family $(q_U,q_V,\dots)$ defines thus a K\"ahler conifold metric on $\cX/A.$
\end{proof}

\section{Tensor Calculus}\l{sect: tens calc}
We generalize several standard notions from K\"ahler geometry. Let $X$ be a complex manifold, $A$ an Artin local $\R$-algebra and $\cX/A$ a deformation of $X.$ A {\it local coordinate system} on $\cX/A$ is the data $(U;z^1,\dots,z^n)$ where $U\sb X$ is an open set isomorphic to an open set in $\C^n$ and such that there exists an $A$-algebra sheaf isomorphism $\O_\cX|_U\cong \O_U\times_\R A;$ and $\ze^1,\dots,\ze^n$ are the coordinates on $U$ embedded in $\C^n.$ For $a=1,\dots,n$ we write $z^a:=\ze^a\otimes 1$ which is a section of $\O_\cX|_U.$  So if $\ph$ is a section of $\La^{pq}_{\cX/A}$ with $p,q\in\Z$ then we can write
\e \ph=\frac1{p!q!}\sum_{\substack{a_1,\dots,a_p=1,\dots,n\\ b_1,\dots,b_q=1,\dots,n}} \ph_{a_1\dots a_p\bar b_1\dots\bar b_q} \d z^{a_1}\wedge\dots\wedge \d z^{a_p}\wedge \d \ov{ z^{b_1}}\wedge \dots\wedge \d \ov{ z^{b_q}}\e
with $\ph_{a_1\dots a_p\bar b_1\dots\bar b_q}\in C^\iy(U,\C)\otimes_\R A.$

Suppose now that $\cX/A$ is given a K\"ahler form $\om.$ In each local coordinate system $(U;z^1,\dots,z^n)$ write $\om=\frac i2\sum_{a,b=1}^n g_{a\bar b}\d z^a\wedge \d z^b$ with $g_{a\bar b}=g_{\bar b a}\in C^\iy(U,\R)\otimes_\R A.$ Denote by $g^{a\bar b}=g^{\bar b a}$ the inverse matrix of $g_{a\bar b},$ both $n{\times}n$ with entries in $A.$ Define for $p,q\in\Z$ an $C^\iy_\cX$ bi-linear sheaf homomorphism $g_{\cX/A}:\La^{pq}_{\cX/A}\times \La^{pq}_{\cX/A}\to C^\iy_\cX$ by saying that if $\ph,\ps\in\Ga(\La^{pq}_{\cX/A})$ then
\e\l{phps}
g_{\cX/A}(\ph,\ps):=\sum g^{a_1\bar c_1}\dots g^{a_p\bar c_p} g^{\bar b_1 d_1}\dots g^{\bar b_q d_q}\ph_{a_1\dots a_p\bar b_1\dots\bar b_q}\ov\ps_{\bar c_1\dots \bar c_p d_1\dots d_q}
\e
where $\sum$ is over $a_1,\dots,a_p;b_1,\dots,b_q; c_1,\dots,c_p; d_1,\dots,d_q=1,\dots,n.$ Since $\La^{pq}_{\cX/A}$ is a locally free $C^\iy_{\cX}$ module and admits partitions of unity it follows that \eq{phps} for $\ph,\ps\in\Ga(\La^{pq}_{\cX/A})$ determines the sheaf homomorphism $g_{\cX/A}.$

The same computation as for ordinary K\"ahler manifolds shows that there exists a unique $A$-module sheaf homomorphism $\nb_{\cX/A}:\La^{10}_{\cX/A}\to \La^1_{\cX/A}\otimes_{C^\iy_\cX}\La^{10}_{\cX/A}$ with the following properties.
\iz
\item[\bf(i)] If $\ph\in \Ga(\La^{10}_{\cX/A})$ and $f\in \Ga(C^\iy_\cX)$ then $\nb_{\cX/A}(f\ph)=\d_{\cX/A}f\otimes \ph+ f\nb_{\cX/A}\ph.$ 
\item[\bf(ii)] If $\ph,\ps\in \Ga(\La^{10}_{\cX/A})$ then $\d [g_{\cX/A}(\ph,\ps)]=g_{\cX/A}(\nb_{\cX/A}\ph,\ps)+g_{\cX/A}(\ph,\nb_{\cX/A}\ps).$
\item[\bf(iii)] Using $\La^1_{\cX/A}=\La^{10}_{\cX/A}\oplus \La^{01}_{\cX/A}$ define the projections $\La^1_{\cX/A}\to \La^{01}_{\cX/A}$ and $\La^1_{\cX/A}\otimes_{C^\iy_\cX} \La^{10}_{\cX/A}\to \La^{01}_{\cX/A}\otimes_{C^\iy_\cX}\La^{10}_{\cX/A}.$ The composite of the latter with $\nb_{\cX/A}$ is then equal to $\db_{\cX/A}.$
\iz
The properties (i)--(iii) imply also that we can write $\nb_{\cX/A}$ more explicitly in each local coordinate system $(U;z^1,\dots,z^n).$ For $a,b,c=1,\dots,n$ put
\e \Ga^c_{ab}:=\displaystyle\sum_{k=1}^n g^{c\bar k}\frac{\partial g_{a\bar k}}{\partial z^b}\Bigl(=-\displaystyle\sum_{k=1}^n \frac{\partial g^{c\bar k}}{\partial z^b}g_{a\bar k}\Bigr).\e 
For $\ph\in\La^{10}_{\cX/A}(U)$ put $\displaystyle\nb_a\ph_b=\frac{\partial\ph_b}{\partial z^a}-\sum_{c=1}^n\Ga^c_{ab}\ph_c.$ Then
\e
\nb_{\cX/A}\ph=:\sum_{a,b=1}^n\nb_a \ph_b \d z^a\otimes\d z^b+\sum_{a,b=1}^n\db \ph_b \otimes\d z^b.
\e
This is the Levi-Civita connection in the following sense. Making $U$ smaller if we need, we can suppose that there exists $f\in C^\iy_{\cX}(U)$ with $\displaystyle\frac{\partial^2 f}{\partial z^a\partial z^{\bar b}}=g_{a\bar b}.$ This implies $\Ga^c_{ab}=\Ga^c_{ba}.$ 

There exists also an $A$-module sheaf homomorphism $\nb_{\cX/A}:\La^{01}_{\cX/A}\to \La^1_{\cX/A}\otimes_{C^\iy_\cX}\La^{01}_{\cX/A}$ characterized by the same conditions (i)---(iii) with $\partial_{\cX/A}$ in place of $\db_{\cX/A}$ at the end of (iii). The generalized Christoffel symbols are defined by $\Ga^{\bar c}_{\bar a\bar b}:=\displaystyle\sum_{k=1}^n g^{\bar c k}\frac{\partial g_{k\bar a}}{\partial z^{\bar b}}$ which is also equal to the complex conjugate $\ov{\Ga^c_{ab}}.$ 

For $p,q\in\Z$ extend $\nb_{\cX/A}$ to an operator $\La^{pq}_{\cX/A}\to \La^1_{\cX/A}\otimes_{C^\iy_\cX}\La^{pq}_{\cX/A}$ by the Leibniz rule. In the local coordinate expression, for $c=1,\dots,n$ define $\nb_c,\nb_{\bar c}:\La^{pq}_{\cX/A}(U)\to \La^{pq}_{\cX/A}(U)$ by
\ea
\!\!\!\nb_c\ph_{a_1\dots a_p\bar b_1\dots\bar b_q}=\frac{\partial}{\partial z^c}\ph_{a_1\dots a_p\bar b_1\dots\bar b_q}-\sum_{j=1}^p\sum_{k=1}^n\Ga^k_{ca_j}\ph_{a_1\dots a_{j-1}ka_{j+1}\dots a_p \bar b_1\dots\bar b_q},\\
\!\!\!\nb_{\bar c}\ph_{a_1\dots a_p\bar b_1\dots\bar b_q}=\frac{\partial}{\partial z^{\bar c}}\ph_{a_1\dots a_p\bar b_1\dots\bar b_q}-\sum_{j=1}^q\sum_{k=1}^n\Ga^{\bar k}_{\bar c\bar b_j}\ph_{a_1\dots a_p \bar b_1\dots \bar b_{j-1}\bar k\bar b_{j+1}\dots\bar b_q}.
\ea
These are then the components of $\nb_{\cX/A}\ph$ for $\d z^c\otimes \d z^{a_1}\wedge\dots\wedge \d z^{a_p}\wedge \d z^{\bar b_1}\wedge \dots \wedge \d z^{\bar b_q}$ and $\d z^{\bar c}\otimes \d z^{a_1}\wedge\dots\wedge \d z^{a_p}\wedge \d z^{\bar b_1}\wedge \dots \wedge \d z^{\bar b_q}$ respectively.

Define now an $A$-module sheaf homomorphism $\bar\partial^*_{\cX/A}:\La^{pq}_{\cX/A}\to\La^{p-1\,q}_{\cX/A}$ by 
\e
(\bar\partial_{\cX/A}^*\ph)_{a_1\dots a_p b_1\dots b_{q-1}}:=
(-1)^{p-1} \sum_{\al,\be=1}^ng^{\bar \be\al}\nb_\al\ph_{a_1\dots a_p\bar\be\bar b_1\dots\bar b_{q-1}}
\e
in the local coordinate expression. Define an $A$-module sheaf homomorphism $\De_{\cX/A}:\La^{pq}_{\cX/A}\to \La^{pq}_{\cX/A}$ by
\e\De_{\cX/A}:=2(\bar\partial^*_{\cX/A}\bar\partial_{\cX/A}+\bar\partial_{\cX/A}\bar\partial^*_{\cX/A}).\e
This is the obvious generalization of the Laplacian. The key properties we shall need are the following. 

Notice that the K\"ahler form on $\cX/A$ induces a K\"ahler form on $X.$ Denote by $\De_X:\La^{pq}_X\to\La^{pq}_X$ the Laplacian with respect to the induced K\"ahler form on $X.$ On the other hand, there is a restriction map $\La^{pq}_{\cX/A}\to\La^{pq}_X.$ The diagram
\begin{equation}\begin{tikzcd}
\La^{pq}_{\cX/A}\ar[r]\ar[d,"\De_{\cX/A}"]&\La^{pq}_X\ar[d,"\De_X"]\\
\La^{pq}_{\cX/A}\ar[r]&\La^{pq}_X
\end{tikzcd}\end{equation}
then commutes.

As $\nb_{\cX/A}$ is the Levi-Civita connection in the sense above we can compute $\partial_{\cX/A}$ and $\bar\partial_{\cX/A}$ in terms of $\nb_{\cX/A},$ as we do for ordinary K\"ahler manifolds; that is, if $\ph\in\Ga(\La^{pq}_{\cX/A})$ then in the local coordinate expression we have
\ea
(\partial_{\cX/A}\ph)_{a_1\dots a_{p+1}\bar b_1\dots\bar b_q}&=\sum_{j=1}^{p+1}(-1)^{j-1} \nb_{a_j}\ph_{a_1\dots \hat a_j\dots a_{p+1}\bar b_1\dots\bar b_q},\\
(\bar\partial_{\cX/A}\ph)_{a_1\dots a_p\bar b_1\dots\bar b_{q+1}}&=\sum_{j=1}^{q+1}(-1)^{p+j-1} \nb_{\bar b_j}\ph_{a_1\dots a_p\bar b_1\dots \hat b_j\dots\bar b_{q+1}}.
\ea
The latter implies readily that $\bar\partial^*_{\cX/A}:\La^{pq}_{\cX/A}\to\La^{p\, q-1}_{\cX/A}$ is the formal adjoint of $\bar\partial_{\cX/A}:\La^{p-1\,q}_{\cX/A}\to\La^{pq}_{\cX/A}$ with respect to the measure $\om^n_{\cX/A};$ that is, for every section $\ph\in\Ga(\La^{p-1\, q}_{\cX/A})$ and every compactly supported section $\ps\in\Ga_c(\La^{pq}_{\cX/A})$ we have
\e\l{dbar}
\int_Xg_{\cX/A}(\db_{\cX/A}^*\ph,\ps)\om^n_{\cX/A}=\int_Xg_{\cX/A}(\ph,\bar\partial_{\cX/A}\ps)\om^n_{\cX/A}.
\e 
In the same way, define an $A$-module sheaf homomorphism $\partial^*_{\cX/A}:\La^{pq}_{\cX/A}\to\La^{p-1\,q}_{\cX/A}$ by
\e
(\partial^*_{\cX/A}\ph)_{a_1\dots a_{p-1}\bar b_1\dots\bar b_q}:=-\sum_{\al,\be=1}^ng^{\al\bar\be} \nb_{\bar\be}\ph_{\al a_1\dots \dots a_{p-1}\bar b_1\dots\bar b_q}.
\e
This is then the formal adjoint of $\partial_{\cX/A}.$ 

Put $\ov{\De_{\cX/A}}:=\partial^*_{\cX/A}\partial_{\cX/A}+\partial_{\cX/A}\partial^*_{\cX/A}$ and $\d_{\cX/A}^*=\partial^*_{\cX/A}+\bar\partial^*_{\cX/A}.$ Generalizing the standard computation for K\"ahler manifolds, we prove that
\e\l{De0} \square:=\d_{\cX/A}^*\d_{\cX/A}+\d_{\cX/A}\d^*_{\cX/A}=2\De_{\cX/A}=2\ov{\De_{\cX/A}}.\e
\begin{proof}[Proof of \eq{De0}]
Define a $C^\iy_\cX$ module homomorphism $\om_{\cX/A}\wedge: \La^{pq}_{\cX/A}\to\La^{p+1\, q+1}_{\cX/A}$ to be the left multiplication by $\om_{\cX/A}.$ Define a $C^\iy_\cX$ module homomorphism $\La:\La^{pq}_{\cX/A}\to \La^{p-1\, q-1}_{\cX/A}$ to be the pointwise adjoint of $\om_{\cX/A}\wedge;$ or equivalently, in the local coordinate expression, if $\ph$ is a section of $\La^{pq}_{\cX/A}$ then set
\e
(\La\ph)_{a_1\dots a_{p-1} \bar b_1\dots\bar b_{q-1}}:=i(-1)^p\sum_{a,b=1}^ng^{\bar b a}\ph_{aa_1\dots a_{p-1}\bar b\bar b_1\dots\bar b_{q-1} }.
\e
We show that $[\partial_{\cX/A},\La]=-i\bar\partial^*_{\cX/A}$ as $A$-module sheaf homomorphisms from $\La^{pq}_{\cX/A}$ to $\La^{p\,q-1}_{\cX/A}.$ If $\ph$ is a local section of $\La^{pq}_{\cX/A}$ then
\ea\l{La1}
(\partial_{\cX/A}\La\ph)_{a_1\dots a_p\bar b_1\dots\bar b_{q-1}}
=\sum_{j=1}^p(-1)^{j-1} \nb_{a_j} (\La\ph)_{a_1\dots\hat a_j\dots a_p\bar b_1\dots\bar b_{q-1}}\\
=i (-1)^p\sum_{a,b=1}^ng^{\bar b a} \sum_{j=1}^p(-1)^{j-1} \nb_{a_j} \ph_{aa_1\dots  \hat a_j\dots a_p\bar b\bar b_1\dots\bar b_{q-1}}.
\ea
On the other hand,
\begin{align*}
&(\La\partial_{\cX/A}\ph)_{a_1\dots a_p\bar b_1\dots\bar b_{q-1}}
=i(-1)^{p-1} \sum_{a,b=1}^ng^{\bar b a}(\partial\ph)_{aa_1\dots a_p\bar b\bar b_1\dots\bar b_{q-1}}\\
&=i(-1)^{p-1} \sum_{a,b=1}^ng^{\bar b a}(\nb_a\ph_{a_1\dots a_p\bar b\bar b_1\dots\bar b_{q-1}}+(-1)^j \nb_{a_j} \ph_{aa_1\dots  \hat a_j\dots a_p\bar b\bar b_1\dots\bar b_{q-1}}). 
\end{align*}
This with \eq{La1} implies
\[
(\partial_{\cX/A}\La\ph-\La\partial_{\cX/A}\ph)_{a_1\dots a_p\bar b_1\dots\bar b_{q-1}}=-
i(-1)^{p-1} \sum_{a,b=1}^ng^{\bar b a}\nb_a\ph_{a_1\dots a_p\bar b\bar b_1\dots\bar b_{q-1}}
\]
which is equal to $-i(\bar\partial^*_{\cX/A}\ph)_{a_1\dots a_p\bar b_1\dots\bar b_{q-1}}$ as claimed. We compute now $\De_{\cX/A}:=\bar\partial^*_{\cX/A}\bar\partial_{\cX/A}+\bar\partial_{\cX/A}\bar\partial^*_{\cX/A}.$ The identity $[\partial_{\cX/A},\La]=-i\bar\partial^*_{\cX/A}$ implies
\ea\l{De1}
-i\De_{\cX/A}&=\bar\partial_{\cX/A}[\partial_{\cX/A},\La]+[\partial_{\cX/A},\La]\bar\partial_{\cX/A}\\
&=\bar\partial_{\cX/A}\partial_{\cX/A}\La-\bar\partial_{\cX/A}\La\partial_{\cX/A}+\partial_{\cX/A}\La\bar\partial_{\cX/A}-\La\partial_{\cX/A}\bar\partial_{\cX/A}.
\ea
Since $\La$ is a real operator it follows also that $-i\partial^*_{\cX/A}=-[\bar\partial_{\cX/A},\La]$ and hence that
\ea\l{De2}
-i\ov{\De_{\cX/A}}&=-\bar\partial_{\cX/A}[\partial_{\cX/A},\La]-[\bar\partial_{\cX/A},\La]\partial_{\cX/A}\\
&=\bar\partial_{\cX/A}\partial_{\cX/A}\La+\partial_{\cX/A}\La\bar\partial_{\cX/A}-\bar\partial_{\cX/A}\La\partial_{\cX/A}-\La\partial_{\cX/A}\bar\partial_{\cX/A}
\ea
whose right-hand side is equal to that of \eq{De1}. Thus $\De_{\cX/A}=\ov{\De_{\cX/A}}.$ On the other hand,
\ea\l{De3}
\square=(\partial_{\cX/A}+\bar\partial_{\cX/A})(\partial^*_{\cX/A}+\bar\partial^*_{\cX/A})+(\partial^*_{\cX/A}+\bar\partial^*_{\cX/A})(\partial_{\cX/A}+\bar\partial_{\cX/A})\\
\!\!\!\!\!\!=\De_{\cX/A}+\ov{\De_{\cX/A}}+(\partial_{\cX/A}\bar\partial^*_{\cX/A}+\bar\partial^*_{\cX/A}\partial_{\cX/A})+(\bar\partial_{\cX/A}\partial^*_{\cX/A}+\partial^*_{\cX/A}\bar\partial_{\cX/A}).
\ea
Using again the identity $[\partial_{\cX/A},\La]=-i\bar\partial^*_{\cX/A}$ we find
\[
-i(\partial_{\cX/A}\bar\partial^*_{\cX/A}+\bar\partial^*_{\cX/A}\partial_{\cX/A})=
\partial_{\cX/A}(\partial_{\cX/A}\La-\La\partial_{\cX/A})+(\partial_{\cX/A}\La-\La\partial_{\cX/A})\partial_{\cX/A}=0;
\]
that is, the second last term of \eq{De3} vanishes. Taking the complex conjugates we see also that the last term of \eq{De3} vanishes. The equation \eq{De3} implies therefore $\square=\De_{\cX/A}+\ov{\De_{\cX/A}}=2\De_{\cX/A},$ proving \eq{De0}.
\end{proof}
Also $\De_{\cX/A}:\La^{pq}_{\cX/A}\to \La^{pq}_{\cX/A}$ is an elliptic operator with
\begin{align*}
(\De_{\cX/A}\ph)_{a_1\dots a_p \bar b_1\dots\bar b_q}&=-g^{\ov\be \al}\nb_\al\nb_{\ov\be}\ph_{a_1\dots a_p\bar b_1\dots\bar b_q}\\
&+\sum_{j=1}^q(-1)^{j-1}g^{\ov\be \al}[\nb_\al,\nb_{\bar b_j}]\ph_{a_1\dots a_p\ov\be\bar b_1\dots\bar b_{j-1}\bar b_{j+1}\dots\bar b_q}.
\end{align*}
The proof is similar to that for ordinary K\"ahler manifolds.

\section{$C^\iy$ Deformations}\l{sect: C^iy deform}

We introduce now a notion of deforming $C^\iy$ manifolds.
\begin{dfn}\l{dfn: C^iy deform}
Let $X$ be a $C^\iy$ manifold. If $A$ is an Artin local $\C$-algebra then a {\it deformation} over $A$ of $X$ is an $A$-algebra sheaf $\cF$ on $X$ equipped with a $\C$-algebra sheaf isomorphism $\cF\otimes_A(A/\fm A)\cong C^\iy_X$ and such that 
\begin{equation}\l{A-analytic}\parbox{10cm}{
every point of $X$ has an open neighbourhood $U$ on which there exists an $A$-algebra sheaf isomorphism $\cF|_U\cong C^\iy_U\otimes_\C A.$ 
}\end{equation}
The last $A$-algebra sheaf $C^\iy_U\otimes_\C A$ may be regarded as the sheaf on $U$ of $A$-valued $C^\iy$ functions.

Let $\cG$ be another deformation over $A$ of $X.$ Then an {\it isomorphism} from $\cF$ to $\cG$ is an $A$-algebra sheaf isomorphism $\ph:\cF\to\cG$ such that if we denote by $\pi_\cF:\cF\to C^\iy_X$ and $\pi_\cG:\cG\to C^\iy_X$ the natural projections then $\pi_\cG\cm\ph=\pi_\cF.$ We say that $\cF$ and $\cG$ are {\it isomorphic} if there exists an isomorphism from one to the other, which is clearly an equivalence relation. The functor $\De:(\Art)_\C\to (\sets)$ assigns to each Artin local $\C$-algebra $A$ the set $\De(A)$ of isomorphism classes of deformations over $A$ of $X.$

The $A$-algebra sheaf $C^\iy_X\otimes_\C A$ is certainly a deformation over $A$ of $X,$ which we call the {\it trivial} deformation over $A$ of $X.$ It is clear that $\De(\C)$ consists of a single element represented by the trivial deformation of $X.$
\end{dfn}
We show that $\De$ satisfies a condition stronger than (H1) and (H2) in Definition \ref{dfn: H}.
\begin{prop}\l{prop: C^iy}
let $A,B,C$ be Artin local $\C$-algebras and $A\to C,B\to C$ any $\C$-algebra homomorphisms; the induced map $\De(A\times_C B)\to \De(A)\times_{\De(C)}\De(B)$ is then bijective. 
\end{prop}
\begin{proof}
We define explicitly the inverse map $\De(A)\times_{\De(C)}\De(B)\to \De(A\times_C B).$ Take therefore an element of $\De(A)\times_{\De(C)}\De(B)$ and represent it by $(\cF,\cG)$ where $\cF$ is a deformation over $A$ of $X,$ and $\cG$ a deformation over $B$ of $X$ such that $\cF\otimes_A C\cong \cG\otimes_B C.$ Denote by $\cE$ the fibre product of $\cF,\cG$ over $\cF\otimes_A C\cong \cG\otimes_B C.$ We show that $\cE$ is a deformation over $D:=A\times_C B.$ It is clear that there is a $\C$-algebra isomorphism isomorphism $\cE\otimes_A \C\cong\O_X.$ Since $\cF,\cG$ satisfy the condition \eq{A-analytic} it follows that every point of $X$ has an open neighbourhood $U$ on which there exists an $A$-algebra sheaf isomorphism $\cF|_U\cong C^\iy_U\otimes_\C D.$ So $\cE$ represents an element of $\De(A\times_C B)$ and defines the inverse map we want.
\end{proof}

We show that $\De$ satisfies a condition stronger than (H3) in Definition \ref{dfn: H}.
\begin{thm}\l{thm: C^iy}
$\De(\C[t]/t^2)$ consists of a single element represented by the trivial deformation of $X.$
\end{thm}
\begin{proof}
Choose an open cover $U\cup V\cup\dots=X$ such that $\cF$ is made from $C^\iy_U\otimes_\C A,C^\iy_V\otimes_\C A,\dots$ by gluing them together. For each $U$ denote by $\si_U:\cF|_U\to C^\iy_U\otimes_\C A$ the $A$-algebra sheaf isomorphism on $U;$ and for each $U,V$ define $\si_{UV}:C^\iy_{U\cap V}\otimes_\C A\to C^\iy_{U\cap V}\otimes_\C A$ by $\si_{UV}:=(\si_U|_{U\cap V})\cm (\si_V^{-1}|_{U\cap V}),$ which we call the {\it transition function} for $U,V$ of $\cF.$ Define $\ta_{UV}: C^\iy_{U\cap V}\otimes_\C A \to C^\iy_{U\cap V}\otimes_\C A$ by $\ta_{UV}:=\si_{UV}-\id.$ Simple computation shows that this is an $A$-linear derivation. The exact sequence $0\to C^\iy_{U\cap V}\otimes_\C(t)\to C^\iy_{U\cap V}\otimes_\C A\to C^\iy_{U\cap V}\to 0$ shows that $\ta_{UV}$ may be regarded as a $\C$-linear derivation $C^\iy_{U\cap V}\to C^\iy_{U\cap V}\otimes_\C (t).$ Thus $\ta_{UV}$ defines over $U\cap V$ a section of the sheaf $\cD_X:=\der_\C(C^\iy_X,C^\iy_X).$ Varying $U,V$ we get a \v{C}ech $1$-cochain $\ta:=(\ta_{UV})_{U,V}$ of $\cD_X.$ Since $(\si_{UV})$ is the family of transition functions it follows that $\ta$ is a cocycle. On the other hand, as $\cD_X$ is a $C^\iy_X$ module sheaf admitting partitions of unity, there exists a $0$-cocycle $(\th_U)$ whose coboundary is equal to $\ta.$ Define for each $U$ an $A$-algebra sheaf isomorphism $\ze_U:C^\iy_{U\cap V}\otimes_\C A\to C^\iy_{U\cap V}\otimes_\C A$ by $\ze_U:=\id+\th_U.$ Then for each $U,V$ we have $\ze_U|_{U\cap V}\cm \si_{UV}=\ze_V|_{U\cap V}.$ Since $\si_{UV}$ is the transition function over $U\cap V$ of $\cF$ we get, varying $(U,V),$ an $A$-algebra sheaf isomorphism $\cF\cong C^\iy_X\otimes_\C A.$
\end{proof}
\begin{rmk}
The sheaf $\cD_X$ is slightly different from the sheaf of $C^\iy$ vector fields; see also Remark \ref{rmk: real an} below.
\end{rmk}

\begin{cor}\l{cor: triv}
For every Artin local $\C$-algebra $A$ the set $\De(A)$ consists of a single element represented by the trivial deformation of $X.$  
\end{cor}
\begin{proof}
We prove by an induction on the length of $A$ that $\De(A)$ consists of a single element. We know that this holds for $A=\C$ and suppose therefore that $A\to B$ a small extension homomorphism in $(\Art)_\C$ with $\De(B)$ consisting of a single element. Recall then from \eq{H12} that the zero vector space $\De(\C[t]/t^2)$ acts transitively upon the unique fibre of $\De(A)\to\De(B).$ Thus $\De(A)$ consists also of a single element, which completes the induction. It is clear that the single element of $\De(A)$ is represented by the trivial deformation of $X.$
\end{proof}

We return now to the study of a complex manifold $X$ and its deformations.
\begin{cor}\l{cor: triv2}
Let $X$ be a complex manifold and let $\K$ be $\R$ or $\C.$ Let $A$ be an Artin local $\K$-algebra and $\cX/A$ a deformation of $X.$ Then there exists a differential graded $A$-algebra sheaf isomorphism $(\La^\bt_{\cX/A},\d_{\cX/A},\wedge)\cong (\La^\bt_X\otimes_\R A,\d_X\otimes\id_A,\wedge).$
\end{cor}
\begin{proof}
Suppoes $\K=\R$ and put $B:=A\otimes_\R\C.$ Recall from Definition \ref{dfn: glue} that $C^\iy_\cX$ is a deformation over $B$ of the $C^\iy$ manifold which underlies $X.$ Applying Corollary \ref{cor: triv} to this $C^\iy_\cX$ we get an $B$-algebra sheaf isomorphism $C^\iy_\cX\cong C^\iy_X\otimes_\C B\cong  C^\iy_X\otimes_\R A.$ More explicitly, after choosing an open cover $U\cup V\cup\dots=X$ we can reproduce the sheaf $C^\iy_\cX$ from the local models $C^\iy_U\otimes_\R A,C^\iy_V\otimes_\R A,\dots,$ gluing them together under the identity functions. We can then reproduce the sheaf $\La^\bt_{\cX/A}$ from the local models $\La^\bt_U\otimes_\R A,\La^\bt_V\otimes_\R A,\dots,$ gluing them together under the identity functions. There is thus an isomorphism $\La^\bt_{\cX/A}\cong \La^\bt_X\otimes_\R A$ which is compatible with the differentials and wedge products. The case $\K=\C$ may be treated in the same way. 
\end{proof}
\begin{rmk}\l{rmk: real an}
We can prove Corollary \ref{cor: triv2} also by using real analytic functions, which we explain briefly now. It is easy to modify Definition \ref{dfn: C^iy deform} and Proposition \ref{prop: C^iy}. The real analytic version of Theorem \ref{thm: C^iy} will be slightly different. The stalks of real analytic functions will be Noetherian rings and the sheaf $\Th_X:=\der_\C(C^\om_X,C^\om_X)$ of derivations will agree with the locally free sheaf of real analytic vector fields. Recall now from \cite[p461]{Grauert2} that the real analytic manifold $X$ is embeddable into a Stein complex space $Y$ so that $X\sb Y$ has a fundamental system of Stein open neighbourhoods. Cartan \cite[Th\'eor\`eme 1]{Cart} proves then that if $\cF$ is a coherent $C^\om_X$ module sheaf on $X$ then for $p=1,2,3,\dots$ we have $H^p(X,\cF)=0.$ After we apply this to $\cF_X=\Th_X$ and $p=1$ we can follow the proof of Theorem \ref{thm: C^iy}. It is also easy to modify Corollary \ref{cor: triv}. Hence we get the real analytic version of the isomorphism $\La^\bt_{\cX/A}\cong \La^\bt_X\otimes_\R A.$ Tensoring the local models with the sheaves of $C^\iy$ functions, we come back to the same conclusion as in Corollary \ref{cor: triv2}.
\end{rmk}

\section{Proof of Theorem \ref{main thm1}}\l{sect: proof of (i)}\l{sect: proof1}

We make now the definitions we will use about linear differential operators over manifolds.
\begin{dfn}\l{dfn: ell}
Let $X$ be a $C^\iy$ manifold (which need not be compact) and $E$ a $C^\iy$ complex vector bundle over $X.$ Denote by $\Ga(E)$ the set of $C^\iy$ sections of $E,$ by $\Ga_c(E)$ the set of compactly supported $C^\iy$ sections of $E,$ by $L^1_{\rm loc}(E)\sb\cD'(E)$ the set of locally $L^1$ sections of $E,$ and by $\cD'(E)$ the set of distribution sections of $E.$ The last means that each element of $\cD'(E)$ is a continuous $\C$-linear map from $\Ga_c(E)$ to $\C$ where $\Ga_c(E)$ is given the compact $C^\iy$ topology.

Suppose now that $X$ is given a Riemannian metric and $E$ a Hermitian metric. Denote by $L^2(E)\sb L^1_{\rm loc}(E)$ the set of $\xi\in L^1_{\rm loc}(E)$ with $\int_X |\xi|^2 \d\mu<\iy$ where $|\xi|$ is the pointwise norm with respect to the Hermitian metric on $E,$ and $\d\mu$ the volume measure of the Riemannian metric on $X.$ The Cauchy--Schwarz inequality implies that for $\xi,\et\in L^2(E)$ the pointwise pairing $\xi\cdot\et$ relative to the Hermitian metric on $E$ defines a globally $L^1$ function $X\to\C,$ whose integral defines the inner product $(\xi,\et)_{L^2}:=\int_X \xi\cdot\et\,\d\mu.$ It is well known that $L^2(E)$ is a Hilbert space and $\Ga_c(E)$ a dense subspace of $L^2(E).$  

Let $F$ be another complex vector bundle over $X,$ and $P:\Ga(E)\to\Ga(F)$ a linear differential operator (with $C^\iy$ coefficients). Note that $P$ extends to a linear operator $\cD'(E)\to \cD'(F)$ which we denote by the same $P.$ Denote by $\ker P\sb L^2(E)$ the kernel of the operator $P$ restricted to $L^2(E);$ that is, $\ker P:=\{\xi\in L^2(E):P\xi=0\}.$ This is a closed subspace of the Hilbert space $L^2(E).$

Suppose now that $F$ is given a Hermitian metric and define then the formal adjoint operator $P^*:\Ga_c(F)\to \Ga_c(E)$ by saying that for every $\xi\in \Ga(E)$ and $\et\in \Ga_c(F)$ we have $(P^*\et,\xi)_{L^2}=(\et,P\xi)_{L^2}.$ This $P^*$ is also a linear differential operator (with $C^\iy$ coefficients and of the same order as $P$). \end{dfn}

We make another definition we will use to prove Theorem \ref{main thm1}.
\begin{dfn}\l{dfn: La^p}
Let $X$ be a complex manifold which is given a K\"ahler metric. Let $\K$ be $\R$ or $\C,$ $A$ an Artin local $\K$-algebra and $\cX/A$ a deformation of $X.$ Using Corollary \ref{cor: triv2} choose a differential graded $A$-algebra sheaf isomorphism $(\La^\bt_{\cX/A},\d_{\cX/A})\cong (\La^\bt_X\otimes_\K A,\d_X\otimes \id_A).$ For $p,q\in\Z$ the restriction map $\La^{p+q}_{\cX/A}\to \La^{p+q}_X$ is surjective because it is induced from the projection $A\to A/\fm A=\K.$ On the other hand, the $\K$-algebra homomorphism $\K\to A$ induces a $C^\iy_X$ module sheaf homomorphism
\e\l{La} \La^{p+q}_X=\La^{p+q}_X\otimes_\K\K\to \La^{p+q}_X\otimes_\K A\cong \La^{p+q}_{\cX/A}.\e
Since the map $\K\to A$ splits the projection $A\to\K$ it follows that \eq{La} splits the restriction map $\La^{p+q}_{\cX/A}\to \La^{p+q}_X.$

Regard the $C^\iy_X$ module $\La^{p+q}_X$ as a complex vector bundle and give it the Hermitian metric induced from the K\"ahler metric of $X.$ 
Using the isomorphism $C^\iy_\cX\cong C^\iy_X\otimes_\K A$ regard the $C^\iy_\cX$ module $\La^{p+q}_{\cX/A}$ as a $C^\iy_X$ module with $C^\iy_X$ acting trivially upon the $A$ factor. The isomorphism $\La^{p+q}_{\cX/A}\cong\La^{p+q}_X\otimes_\K A$ defines then a $C^\iy_X$ module isomorphism and accordingly a vector bundle isomorphism. Here $A$ is regarded as a finite-dimensional $\K$-vector space. We give this a metric (that is, a positive definite symmetric $\K$-bilinear form) which is compatible with the splitting $A=\K\oplus \fm A.$ The latter condition means that the restriction to the $\K$ factor agrees with the product structure of $\K$ (which makes sense because the products in $\K$ may be regarded as inner products and accordingly as a metric). As $\La^{p+q}_X$ is given already a Hermitian metric, using the metric on $A$ we get a Hermitian metric on $\La^{p+q}_X\otimes_\K A.$

Using these Hermitian metrics, define $L^2(\La^{p+q}_{\cX/A})$ and $L^2(\La^{p+q}_X)$ as in Definition \ref{dfn: ell}. The restriction map $\La^{p+q}_{\cX/A}\to\La^{p+q}_X$ induces then a map $L^2(\La^{p+q}_{\cX/A})\to L^2(\La^{p+q}_X)$ which we denote by $R.$ Since $\La^{pq}_{\cX/A}\sb\La^{p+q}_{\cX/A}$ we get also the map $R:L^2(\La^{pq}_{\cX/A})\to L^2(\La^{pq}_X).$ 
\end{dfn}
\begin{rmk}
No K\"ahler forms on $\cX/A$ are relevant to Definition \ref{dfn: La^p}.
\end{rmk}

We prove a lemma we will use shortly for the proof of Theorem \ref{thm: cH^n-2} below. 
\begin{lem}\l{lem: db_X/A exact}
Let $(X,x)$ be the germ of a complex space of dimension $n$ and of depth $\ge n.$ Let $A$ be an Artin local $\C$-algebra and $\cX/A$ a deformation of $X.$ Let $\ph\in\Ga(\La^{1\, n-2}_{\cX/A}|_{X^\reg})$ be $\d_{\cX/A}$ exact on a punctured neighbourhood of $x\in X^\sing.$ Then $\ph$ is $\db_{\cX/A}$ exact on a punctured neighbourhood of $x\in X^\sing.$
\end{lem}
\begin{proof}
We prove by an induction on the length of $A$ that $H^{n-1}_x(X,\O_\cX)=0.$ By the depth condition this is true for $A=\C.$ Let $0\to (\ep)\to A\to B\to0$ be a small extension in $(\Art)_\C,$ and $\cY/B$ the deformation of $X$ defined by $\O_\cY:=\O_\cX\otimes_AB.$ The $A$-module sheaf short exact sequence $0\to\O_X\to \O_\cX\to \O_\cY\to0$ induces the long exact sequence containing $0=H^{n-1}_x(X,\O_X)\to H^{n-1}_x(X,\O_\cX)\to H^{n-1}_x(X,\O_\cY).$ By the induction hypothesis, however, $H^{n-1}_x(X,\O_\cY)$ vanishes; and accordingly, so does $H^{n-1}_x(X,\O_\cX).$

Write $\ph=\d_{\cX/A}\ps$ with $\ps\in\Ga(\La^{n-2}_{\cX/A}|_{X^\reg}).$ Using the argument in the proof of Theorem \ref{thm: harm n-1 1 version2} and noting that $\ph$ is a $(1,n-2)$ form, we can in fact write $\ph=\bd_{\cX/A}\ps'+\db_{\cX/A}\ps''$ where $\ps'$ is some $(0,n-2)$ form with $\db_{\cX/A}\ps'=0,$ and $\ps''$ some $(1,n-3)$ form. Let $U$ be a Stein neighbourhood of $x\in X^\sing.$ Then $H^{n-2}(U\-\{x\},\O_\cX)\cong H^{n-1}_x(U,\O_\cX)=0.$ So $\ps'=\db_{\cX/A}\chi$ where $\chi$ is some $(0,n-3)$ form on $U\-\{x\}.$ Thus $\ph=\bd_{\cX/A}\db_{\cX/A}\chi+\db_{\cX/A}\ps''=\db_{\cX/A}(-\bd_{\cX/A}\chi+\ps'')$ is $\db_{\cX/A}$ exact.
\end{proof}

We define now the topologies we will use of sheaf cohomology groups.
\begin{dfn}\l{dfn: top}
Let $X$ be a complex manifold (which we suppose $\si$-compact). For $p,q\in\Z$ we give $\Ga(\La^{pq}_X),$ the space of $(p,q)$ forms on $X,$ the Fr\'echet space topology whose semi-norms are the $C^l$ norms on compact subsets of $X.$ We define also a topology of $\Ga_c(\La^{pq}_X),$ the $(p,q)$ forms with compact support. For every compact $K\sb X$ denote by $\Ga_K(\La^{pq}_X)\sb \Ga_c(\La^{pq}_X)$ the vector subspace consisting of those $(p,q)$ forms supported in $K.$ We give this the Fr\'echet space topology whose semi-norms are the $C^l$ norms with $l\ge0$ an integer. Define a fundamental neighbourhood of the origin $0\in\Ga_c(\La^{pq}_X)$ to be a convex subset $U\sb \Ga(\La^{pq}_X)$ such that for every compact $K\sb X$ the intersection $U\cap \Ga_K(\La^{pq}_X)$ is an open subset of $\Ga_K(\La^{0q}_{X^\reg}).$ These fundamental neighbourhoods define a topology of $\Ga_c(\La^{pq}_X)$ with which it is a locally convex topological vector space. Notice that for every compact $K\sb X$ the inclusion map $\Ga_K(\La^{pq}_X)\to \Ga(\La^{pq}_X)$ is continuous. This implies that the inclusion map $\Ga_c(\La^{pq}_X)\to \Ga(\La^{pq}_X)$ is continuous. 

Write the cohomology groups $H^q(X,\Om^p_X),H^q_c(X,\Om^p_X)$ as those of the $\db$ resolutions and give them the quotient topologies (called the strong topologies in \cite[Theorem 2.1]{Lauf} in which it is proved that these agree with the other topologies corresponding to other expressions of the cohomology groups). The induced map $H^q_c(X,\Om^p_X)\to H^q(X,\Om^p_X)$ is then continuous.
\end{dfn}
\begin{rmk}
For $X$ compact the statements above are obvious. For $X$ non-compact, however, the cohomology groups are in general of infinite dimension and we need therefore more careful treatment.
\end{rmk}

We come now to the heart of the proof of Theorem \ref{main thm1}.
\begin{thm}\l{thm: cH^n-2}
Let $n\ge3$ be an integer and $X$ a compact K\"ahler $n$-conifold with $X^\sing$ non-empty, rational and of depth $\ge n.$ Suppose that for each $x\in X^\sing$ we have $b_{n-2}(C_x^\reg)=0$ or $b_{n-1}(C_x^\reg)=0.$ Let $A$ be an Artin local $\C$-algebra and $\cX/A$ a deformation of $X.$ The natural map ${}_c H^{n-2}(X^\reg,\Om^1_{\cX/A})\to{}_c H^{n-2}(X^\reg,\Om^1_X)$ is then surjective.
\end{thm}
\begin{proof}
Take any element of ${}_c H^{n-2}(X^\reg,\Om^1_X)$ and represent it on $X^\reg,$ using Lemma \ref{lem: 1 n-2}, by a harmonic $(1,n-2)$ form of order $1-n,$ which we call $\ph_0.$ Using the $\C$-algebra homomorphism $\C\to A$ lift $\ph_0$ to a section of $\La^{n-1}_{X^\reg}\otimes_\C A\cong\La^{n-1}_{\cX/A}|_{X^\reg}$ of the same order, which we call $\ph\in \Ga(\La^{n-1}_{\cX/A}|_{X^\reg}).$ For $p,q\in\Z$ with $p+q=n-1$ denote by $\ph^{pq}\in \Ga(\La^{pq}_{\cX/A}|_{X^\reg})$ the $(p,q)$ part of $\ph.$ Since $\d\ph_0=0$ it follows that $\d_{\cX/A}\ph=0$ and hence that $\db_{\cX/A}\ph^{0\,n-1}=0.$ Since $\ph$ is of order $>1-n$ it follows that $\ph$ is $\d_{\cX/A}$ exact on a punctured neighbourhood of $X^\sing\sb X,$ on which therefore there exists a relative $n-2$ form $\th$ of order $>2-n$ and such that $\d_{\cX/A}\th=\ph.$ Let $\ze$ be a relative $(0,n-1)$ form on $X^\reg$ which near $X^\sing$ agrees with the $(0,n-1)$ part of $\th;$ such forms may be made using partitions of unity, and any choice will do. Now $\ph^{0\,n-1}-\db_{\cX/A}\ze$ is a compactly supported $\db_{\cX/A}$ closed relative $(0,n-1)$ form on $X^\reg.$ 

We show that $H^{n-1}_c(X^\reg,\O_X)=0.$ Since $n\ge3$ it follows that for $U\sb X$ a Stein neighbourhood of $X^\sing,$ if we write $U^\reg:=U\-X^\sing$ then $H^{n-2}(U^\reg,\O_X)\cong H^{n-1}_{X^\sing}(U,\O_X);$ and this vanishes because $X^\sing$ is of depth $\ge n.$ Hence we get an exact sequence $0\to H^{n-1}_c(X^\reg,\O_X)\xrightarrow{\al} H^{n-1}(X^\reg,\O_X).$ Recall from Definition \ref{dfn: top} that the latter $\al$ is continuous, which implies that the closure of the origin $0\in H^{n-1}_c(X^\reg,\O_X)$ lies in the kernel of $\al.$ But $\al$ is injective, so the closure of the origin is a zero vector space. By \cite[Theorem 5.11]{KN} therefore $H^{n-1}_c(X^\reg,\O_X)$ is Hausdorff.

By \cite[Theorem 2.1]{Lauf} (or \cite[Chapter IIV, Theorem 4.2]{BS}) there exists a Serre duality isomorphism $H^{n-1}_c(X^\reg,\O_X)\cong H^1(X^\reg,\O_X).$ Since $X^\sing$ is of depth $\ge3$ we get another isomorphism $H^1(X^\reg,\O_X)\cong H^1(X,\O_X).$ By Proposition \ref{prop: hol 1-forms} however $H^1(X,\O_X)=0.$ So $H^{n-1}_c(X^\reg,\O_X)=0.$  

By an induction on the length of $A$ we can verify that $H^{n-1}_c(X^\reg,\O_{\cX})=0.$ Since $\ph^{0\,n-1}-\db_{\cX/A}\ze$ defines an element of this cohomology group it follows that $\ph^{0\,n-1}-\db_{\cX/A}\ze=\db_{\cX/A}\et$ where $\et$ is some compactly supported relative $(0,n-2)$ form on $X^\reg.$ Thus $\ph=\db_{\cX/A}(\et+\ze).$ Since $\ze$ near $X^\sing$ agrees with $\th$ it follows that $\ze$ is of order $>2-n$ and hence that so is $\et+\ze.$ In particular, $\et+\ze$ is $L^2.$ We apply to this the orthogonal decomposition
\e\l{relative orthogonal}
L^2(\La^{0\,n-2}_{\cX/A}|_{X^\reg})=\ker(\db_{\cX/A}+\db_{\cX/A}^*)\oplus \ov{\db_{\cX/A}(\Ga_c(\La^{0\,n-3}_{\cX/A}|_{X^\reg}))}\oplus \ov{\db_{\cX/A}^*(\Ga_c(\La^{0\,n-1}_{\cX/A}|_{X^\reg}))}
\e
where the formal adjoint $\db_{\cX/A}^*$ is defined with respect to the Hermitian metric in Definition \ref{dfn: La^p} (and so different from the $\db_{\cX/A}^*$ in \S\ref{sect: tens calc}) and where the closures are defined in the $L^2$ sense. Denote by  $\xi\in \ov{\db_{\cX/A}^*(\Ga_c(\La^{0\,n-1}_{\cX/A}|_{X^\reg}))}$ the last of the three components of $\et+\ze.$ Since the first two components on the right-hand side of \eq{relative orthogonal} lie in $\ker\db_{\cX/A}$ it follows that $\db_{\cX/A}\xi=\db_{\cX/A}(\et+\ze)$ and hence that $\ph^{0\,n-1}=\db_{\cX/A}\xi.$ On the other hand, since the last component on the right-hand side of \eq{relative orthogonal} lies in $\ker\db^*_{\cX/A}$ it follows that $\db_{\cX/A}^*\xi=0.$ 

We show next that if we denote by $R:L^2(\La^{\bullet}_{\cX/A}|_{X^\reg})\to L^2(\La^{\bullet}_{X^\reg})$ the restriction map and if we put $\chi:=\ph^{1\,n-2}-\partial_{\cX/A}\xi$ then $R\chi=\ph_0.$ We show indeed that $R\partial_{\cX/A}\xi=0.$ Since $R\ph=\ph_0$ is a $(1,n-2)$ form it follows that $R(\ph^{0\,n-1})=0$ and hence that $\db_X R\xi=R\db_{\cX/A}\xi=R\ph^{0\,n-1}=0.$ Also $\db_{\cX/A}^*\xi=0$ implies $\db_X^*R\xi=R(\db_{\cX/A}^*\xi)=0.$ Thus $R\xi$ is an $L^2$ harmonic $(0,n-2)$ form on $X.$ By Lemma \ref{lem: dph1} therefore we have $\partial_XR\xi=0$ and accordingly $R\partial_{\cX/A}\xi=\partial_XR\xi=0.$ So $R\chi=R(\ph^{1\,n-2}-\partial_{\cX/A}\xi)=R\ph^{1\,n-2}=\ph_0.$

As $-\db_{\cX/A}\partial_{\cX/A}\xi=\partial_{\cX/A}\ph^{0\,n-1}$ we have $\db_{\cX/A}\chi=\db_{\cX/A}(\ph^{1\,n-2}-\partial_{\cX/A}\xi)=\db_{\cX/A}\ph^{1\,n-2}+\partial_{\cX/A}\ph^{0\,n-1}=0.$ Thus $\chi$ defines an element of the $\db_{\cX/A}$ cohomology group $H^{n-2}(X^\reg,\Om^1_{\cX/A}).$ We show that this lies in the subspace ${}_cH^{n-2}(X^\reg,\Om^1_{\cX/A}).$ We look at $\ph-\d_{\cX/A}\xi,$ which is $\d_{\cX/A}$ cohomologous to $\ph.$ Since $\ph$ is locally (near $X^\sing$) $\d_{\cX/A}$ exact it follows that so is $\ph-\d_{\cX/A}\xi.$ Write this therefore, near $X^\sing,$ as $\d_{\cX/A}\al$ where $\al$ is a relative $n-1$ form defined near $X^\sing.$ Since $\ph-\d_{\cX/A}\xi$ has vanishing $(0,n-1)$ component it follows that if we denote by $\al^{0\,n-2}$ the $(0,n-2)$ part of $\al$ then $\db_{\cX/A}\al^{0\,n-2}=0.$ But as $X^\sing$ is of depth $\ge n$ we have $H^{n-1}_{X^\sing}(X,\O_X)=0.$ By an induction on the length of $A$ we can verify also that $H^{n-1}_{X^\sing}(X,\O_\cX)=0.$ If $U$ is a Stein neighbourhood of $X^\sing\sb X$ and if we put $U^\reg:=U\-X^\sing$ then $H^{n-2}(U^\reg,\O_X)=0.$ Making $U$ small enough we can suppose that $\al$ is defined on $U$ and can write then $\al^{0\,n-2}=\db_{\cX/A}\be$ where $\be$ is a relative $(0,n-3)$ form on $U^\reg.$ Since $\chi=\ph^{1\,n-2}-\partial_{\cX/A}\xi$ is the $(1,n-2)$ part of $\ph-\d_{\cX/A}\xi$ it follows now that
\[
\chi|_U=\db_{\cX/A}\al^{1\,n-3}+\partial_{\cX/A}\al^{0\,n-2}=\db_{\cX/A}\al^{1\,n-3}+\partial_{\cX/A}\db_{\cX/A}\be=\db_{\cX/A}(\al^{1\,n-3}-\partial_{\cX/A}\be).
\]
Thus $\chi|_U$ is $\db_{\cX/A}$ cohomologous to zero, which means that $\chi$ defines an element of ${}_cH^{n-2}(X^\reg,\Om^1_{\cX/A}),$ completing the proof.
\end{proof}

We show next that \eq{CY} holds in the present circumstances.
\begin{lem}\l{lem: n0}
Let $X$ be a compact Calabi--Yau $n$-conifold with $X^\sing$ non-empty. Let $\K$ be $\R$ or $\C,$ $A$ an Artin local $\K$-algebra and $\cX/A$ a deformation of $X.$ Then there exists on $X^\reg$ a nowhere-vanishing section of $\Om^n_{\cX/A}.$
\end{lem}
\begin{proof}
By Proposition \ref{prop: hol 1-forms} we have $H^1(X,\O_X)=0.$ We proceed by an induction on $A.$ Let $0\to(\ep)\to A\to B\to0$ be a small extension in $(\Art)_\K.$ Denote by $\cY/B$ the deformation of $X$ induced from $\cX/A.$ There is then an $A$-module sheaf exact sequence $0\to \io_*\Om^n_{X^\reg}\to \io_*(\Om^n_{\cX/A}|_{X^\reg})\to \io_*(\Om^n_{\cY/B}|_{X^\reg})\to 0.$ Since $H^1(X,\io_*\Om^n_{X^\reg})=H^1(X,\O_X)=0$ it follows that the restriction map $\Ga(\Om^n_{\cX/A}|_{X^\reg})\to \Ga(\Om^n_{\cY/B}|_{X^\reg})$ is surjective. Starting with a nowhere-vanishing section $\Om$ of $\Om^n_{X^\reg}$ we get, by the induction on $A,$ a section $\Ph$ of $\Om^n_{\cX/A}|_{X^\reg}$ whose restriction to $\spec\C$ is equal to $\Om.$ Since $\Om$ is nowhere-vanishing it follows that so is $\Ph.$
\end{proof}

We prove now a corollary of Theorem \ref{thm: cH^n-2}.
\begin{cor}\l{cor: cH^n-2}
Let $X$ be a compact Calabi--Yau $n$-conifold with $n\ge3$ and such that for each $x\in X^\sing$ we have $b_{n-2}(C_x^\reg)=0$ or $b_{n-1}(C_x^\reg)=0.$ Notice that by Lemma \ref{lem: n0} the condition \eq{CY} holds so that Definition \ref{dfn: Gross} makes sense. The map $\de|_{\im\ga}:\im\ga\to \bop_{x\in X^\sing}\Ext^2_{\O_{X,x}}(\Om^1_{X,x},\O_{X,x})$ is then injective.
\end{cor}
\begin{proof}
We show first that the following holds:
\begin{equation}\parbox{10cm}{
Let $q\ge1$ be an integer and $U\sb X$ a Stein open neighbourhood of $X^\sing.$ Then for $j=0,\dots,k$ we have $H^q(U,\Om^1_{X_j/A_j})=0.$ 
}\end{equation}
We prove this by induction on $j.$ For $j=0$ the sheaf $\Om^1_X$ is a coherent $\O_X$ module and so $H^q(U,\Om^1_X)=0.$ Suppose next that $H^q(U,\Om^1_{X_{j-1}/A_{j-1}})=0$ for some $j>0.$ Since the kernel of $\ta:\Om^1_{X_{j-1}/A_{j-1}}\to\Om^1_{X_j/A_j}$ is supported on the isolated set $X^\sing$ it follows that $H^q(U,\ker\ta)=H^{q+1}(U,\ker\ta)=0$ and hence that the natural map $H^q(U,\Om^1_{X_{j-1}/A_{j-1}})\to H^q(U,\im\ta)$ is an isomorphism. The induction hypothesis implies therefore that $H^q(U,\im\ta)=0$ and the short exact sequence $0\to \im\ta\to  \Om^1_{X_j/A_j}\to \Om^1_X\to0$ implies in turn that $H^q(U,\Om^1_{X_j/A_j})=0,$ completing thus the induction argument. 

As $n\ge3$ we can apply the result to $q=n-2,n-1;$ that is, for $j=0,\dots,k$ we have $H^q(U,\Om^1_{X_j/A_j})=0.$ Lemma \ref{lem: c} implies therefore that the image of the natural map $H^{n-2}(X,\Om^1_{X_j/A_j})\to H^{n-2}(X^\reg,\Om^1_{X_j/A_j})$ agrees with the image of the natural map $H_c^{n-2}(X^\reg,\Om^1_{X_j/A_j})\to  H^{n-2}(X^\reg,\Om^1_{X_j/A_j}).$ Using this with $j=0,k$ we get a commutative diagram
\begin{equation}\l{k0}\begin{tikzcd}[ sep=small]
\!\!\!\!\!\!\!\!\!\!\!\! H_{X^\sing}^{n-2}(X,\Om^1_{X_k/A_k})\ar[r]\ar[d] &H^{n-2}(X,\Om^1_{X_k/A_k})\ar[r] \ar[d] &\!\!{}_c H^{n-2}(X^\reg,\Om^1_{X_k/A_k})\ar[r]\ar[d]&\!0\\
\!\!\!\!\!\!\!\!\!\!\!\! H_{X^\sing}^{n-2}(X,\Om^1_X)\ar[r,"\si"]&  H^{n-2}(X,\Om^1_X)\ar[r]&\!\! {}_c H^{n-2}(X^\reg,\Om^1_X)\ar[r]& \!0
\end{tikzcd}\end{equation}
whose rows are the local cohomology exact sequences and whose vertical maps are induced from the sheaf homomorphism $\Om^1_{X_k/A_k}\to\Om^1_X.$ Denote by $V$ the cokernel of the middle vertical arrow of \eq{k0}. Since $H^{n-2}(X,\Om^1_X)$ is a $\C$-vector space it follows that so is $V.$ Denote by $\pi:H^{n-2}(X,\Om^1_X)\to V$ the natural projection. Recall from Theorem \ref{thm: cH^n-2} that the right vertical arrow of \eq{k0} is surjective. We see then easily by diagram chase that $\pi\cm\si$ is surjective. 

Recall from \cite[Lemma 2.4(c)]{Gross} that the dual of $\si$ agrees with the $\C$-linear map $\de:\Ext^2_{\O_X}(\Om^1_X,\O_X)\to \Ext^2_{\O_{X,x}}(\Om^1_{X,x},\O_{X,x})$ in \eq{Gross}. The composite map
\e\l{T2T2}\hom_\C(V,\C)\to \hom_\C(H^{n-2}(X,\Om^1_X),\C)=\Ext^2_{\O_X}(\Om^1_X,\O_X)\to
\Ext^2_{\O_{X,x}}(\Om^1_{X,x},\O_{X,x})
\e
is thus dual to the surjection $\pi\cm\si$ and in particular injective.

We show finally that $\hom_\C(V,\C)$ may be identified with the image of $\ga$ in \eq{Gross} and that the map $\hom_\C(V,\C)\to \Ext^2_{\O_X}(\Om^1_X,\O_X)$ from \eq{T2T2} then agrees with the inclusion $\im\ga\to \Ext^2_{\O_X}(\Om^1_X,\O_X).$ Recall from \eq{T1T21} that $V$ is isomorphic to the kernel of the map $H^{n-1}(X,\Om^1_{X_{k-1}/A_{k-1}})\to H^{n-1}(X,\Om^1_{X_k/A_k}).$ The image of the latter arrow of \eq{T1T22} may then be identified with $\hom_{A_k}(V,A_k)\sb \hom_{A_k}(H^{n-2}(X,\Om^1_X),A_k).$ The image of the latter arrow of \eq{T1T23} may in turn be identified with $\hom_\C(V,\C)\sb \hom_\C(H^{n-2}(X,\Om^1_X),\C)=\Ext^2_{\O_X}(\Om^1_X,\O_X).$ The last map thus agrees with $\im\ga\to \Ext^2_{\O_X}(\Om^1_X,\O_X)$ in \eq{T1T2}, which completes the proof.
\end{proof}
\begin{rmk}\l{rmk: n=2}
Corollary \ref{cor: cH^n-2} is true also for $n=1,2.$ For $n=1$ the normal complex space $X$ is non-singular and we have $\Ext^2_{\O_X}(\Om^1_X,\O_X)\cong H^2(X,\Th_X)=0.$ So $\im\ga=0$ and it is obvious that $\de|_{\im\ga}$ is injective. 

For $n=2,$ although the stronger result (Theorem \ref{main thm1}) is known it may be worthwhile to give a direct proof of the current statement (Corollary \ref{cor: cH^n-2}). We show first that $\Om^1_{X_k/A_k}$ is a flat $A_k$ module sheaf. Recall for instance from \cite[Theorem 7.5.1(iv)]{Ish} that every $x\in X^\sing$ is a hypersurface singularity, defined in $\C^3$ by a single equation $f=0.$ By the definition of $\Om^1_X$ there is an exact sequence $\O_X\to \Om^1_{\C^3}\otimes_{\O_{\C^3}}\O_X\to \Om^1_X\to0.$ The first arrow is injective because its kernel vanishes at every point of $X^{\rm reg}.$ There is thus an exact sequence
\e\l{hypersurface sing}
0\to\O_X\to \Om^1_{\C^3}\otimes_{\O_{\C^3}}\O_X\to \Om^1_X\to0.
\e
It is also known that the germ at $x$ of $X_k/A_k$ is an unfolding of $f$ in $\C^3,$ defined by some $F\in\O_{\C^3}\times_\C A_k$ extending $f.$ Generalizing \eq{hypersurface sing} we get an exact sequence
\e\l{unfold}
0\to\O_{X_k}\to \Om^1_{(\C^3\times\spec A)/\spec A}\otimes_{\O_{\C^3\times \spec A}}\O_{X_k}\to \Om^1_{X_k/A_k}\to0.
\e
Now \eq{hypersurface sing} is obtained from \eq{unfold} after tensoring with $\C$ as $A_k$ modules. But then $\mathrm{Tor}_1(\C,\Om^1_{X_k/A_k})$ vanishes and $\Om^1_{X_k/A_k}$ is flat over $A_k.$

So there is an exact sequence $0\to\Om_{X_{k-1}/A_{k-1}}\to \Om^1_{X_k/A_k}\to \Om^1_X\to0.$ On the other hand, there is also an exact sequence $0\to\O_{X_{k-1}/A_{k-1}}\to \O_{X_k/A_k}\to \O_X\to0.$ In particular, passing to the cohomology groups we see that for $p,q\in\Z$ with $p+q=1$ we have
\e\l{orbi0}
\dim_\C H^q(X,\Om^p_{X_k/A_k})\le \dim_\C H^q(X,\Om^p_X)+\dim_\C H^q(X,\Om^p_{X_{k-1}/A_{k-1}}).
\e
Hence it follows by induction on $k$ that
\e\l{orbi1}\dim_\C H^q(X,\Om^p_{X_k/A_k})\le (k+1)\dim_\C H^q(X,\Om^p_X).\e 
On the other hand, from the Hodge spectral sequence $H^q(X,\Om^p_{X_k/A_k})\Rightarrow \gr^p H^1(X,A_k)$ we get
\e\l{orbi2}
\sum_{p+q=1}\dim_\C H^q(X,\Om^p_{X_k/A_k})\ge \dim_\C H^1(X,A_k)=(k+1)\dim_\C H^1(X,\C).
\e
We show now that $X$ has an orbifold K\"ahler form. Recall again from \cite[Theorem 7.5.1(xi)]{Ish} that every singularity $x\in X^\sing$ is of the form $\C^2/G$ with $G<\SU(2)$ a finite subgroup. Take a K\"ahler form on $X^\reg$ and near $x\in X^\sing$ pull it back to $\C^2\-\{0\}.$ By \eq{Fuj2} we can change this to a K\"ahler form on $\C^2$ without changing it at the points far from $x.$ Taking the average with respect to $G$ we can push it down to an orbifold K\"ahler metric. Since compact K\"ahler orbifolds have the Hodge decomposition property \cite{Bail} it follows that the inequalities of \eq{orbi0}--\eq{orbi2} are in fact equalities. The map $H^0(X,\Om^1_{X_k/A_k})\to H^0(X,\Om^1_X)$ is thus surjective. This means the vanishing of the vector space $V$ defined in the proof of Corollary \ref{cor: cH^n-2}. The map $\de$ is therefore injective.  \qed
\end{rmk}

We finally prove Theorem \ref{main thm1}. Let $X$ be a compact Calabi--Yau conifold. By Lemma \ref{lem: n0} the condition \eq{CY} holds and we can thus apply Lemma \ref{lem: concentrate}. Combining it with Corollary \ref{cor: cH^n-2} we complete the proof. \qed.

\section{Relative Harmonic Forms}\l{sect: rel harm}
In this section we prove the facts we shall need for the proof of Theorem \ref{main thm2}.

\begin{dfn}\l{prop: relative Sobolev}
Let $X$ be a compact K\"ahler $n$-conifold, $A$ an Artin local $\R$-algebra and $\cX/A$ a locally trivial deformation of $X.$ Using Corollary \ref{cor: extending conifold metrics} choose on $\cX/A$ a K\"ahler conifold metric, whose K\"ahler form we denote by $\om_{\cX/A}\in \Ga(\La^{11}_{\cX/A}).$ Define $g_{\cX/A}$ by \eq{phps}. For $p\in\Z$ and $\al\in\R$ define an $A$-bilinear map $L^2_\al(\La^p_{\cX/A}|_{X^\reg})\times L^2_{2-\al-2n}(\La^p_{\cX/A}|_{X^\reg})\to A$ by $(\ph,\ps)\mapsto\int_{X^\reg}g_{\cX/A}(\ph,\ps)\om^n_{\cX/A}.$ This is well defined because the relative volume form $\om^n_{\cX/A}$ near each $x\in X^\sing$ may be identified with that of the cone metric at $x.$ For the same reason there is an $A$-module isomorphism $L^2_\al(\La^p_{\cX/A}|_{X^\reg})\cong L^2_\al(\La^p_{X^\reg})\otimes_\R A.$ We give $A$ the topology induced by any norms on it, which is well defined because $A$ is a finite-dimensional $\R$-vector space. Using the topology on $L^2_\al(\La^p_{X^\reg})$ and that of $A$ we give $L^2_\al(\La^p_{X^\reg})\otimes_\R A$ the topology of a Banach space. Using the $A$-module isomorphism $L^2_\al(\La^p_{\cX/A}|_{X^\reg})\cong L^2_\al(\La^p_{X^\reg})\otimes_\R A$ we give $L^2_\al(\La^p_{\cX/A}|_{X^\reg})$ the topology of a Banach space.
\end{dfn}

We study the dual spaces of the weighted $L^2$ spaces for infinitesimal deformations.
\begin{prop}\l{prop: relative duality}
Let $X$ be a compact K\"ahler $n$-conifold, $A$ an Artin local $\R$-algebra and $\cX/A$ a locally trivial deformation of $X.$ Using Corollary \ref{cor: extending conifold metrics} choose on $\cX/A$ a K\"ahler conifold metric. Fix $p\in\Z$ and $\al\in\R.$ Denote by $\hom_A(L^2_\al(\La^p_{\cX/A}|_{X^\reg}),A)$ the space of continuous $A$-module homomorphisms and by $\Xi_A:L^2_\al(\La^p_{\cX/A}|_{X^\reg})\to \hom_A(L^2_{2-\al-n}(\La^p_{\cX/A}|_{X^\reg}),A)$ the $A$-module homomorphism $\ph\mapsto \int_{X^\reg}g_{\cX/A}(\ph,\bullet)\om^n_{\cX/A}.$ The latter is then an $A$-module isomorphism.
\end{prop}
\begin{proof}
For $A=\R$ this is a standard fact; the $p=0$ case is treated in \cite[Lemma 2.8]{J1} and the extension to the general $p$ is straightforward. We treat the general $A$ by an induction on its length. Let $0\to(\ep)\to A\to B\to0$ be a small extension in $(\Art)_\R.$ Define a deformation $\cY/B$ by $\O_\cY:=\O_\cX\otimes_AB.$ There is then a commutative diagram
\begin{equation}\begin{tikzcd}
L^2_\al(\La^p_{\cX/A}|_{X^\reg})\ar[d,"R"]\ar[r,"\Xi_A"]& \hom_A(L^2_{2-\al-n}(\La^p_{\cX/A}|_{X^\reg}),A)\ar[d,"R"]\\
L^2_\al(\La^p_{\cY/B}|_{X^\reg})\ar[r,"\Xi_B"] &\hom_B(L^2_{2-\al-n}(\La^p_{\cY/B}|_{X^\reg}),B)
\end{tikzcd}\end{equation}
where the horizontal maps, both called $R,$ are induced by $A\to B.$ We show first that $\Xi_A$ is injective. Take $\ph\in L^2_\al(\La^p_{\cX/A}|_{X^\reg})$ with $\Xi_A\ph=0.$ Then $\Xi_BR\ph=R\Xi_A\ph=0.$ By the induction hypothesis we have $R\ph=0.$ So $\ph\in  L^2_\al(\La^p_{\cX/A}|_{X^\reg})\otimes_A(\ep)\cong L^2_\al(\La^p_{X^\reg})\otimes_\R(\ep)$ and we can therefore write $\Xi_A\ph=\Xi_\R\ph.$ But we know already that $\Xi_\R$ is injective, which implies $\ph=0$ as we have to prove. 

We show next that $\Xi_A$ is surjective. Take $\ps\in \hom_A(L^2_{2-\al-n}(\La^p_{\cX/A}|_{X^\reg}),A).$ By the induction hypothesis there exists $\ph'\in L^2_\al(\La^p_{\cY/B}|_{X^\reg})$ with $\Xi_B\ph'=R\ps.$ Since $\La^p_{\cX/A}$ is isomorphic to $\La^p_{X^\reg}\otimes_\R A$ it follows that $R:L^2_\al(\La^p_{\cX/A}|_{X^\reg})\to L^2_\al(\La^p_{\cY/B}|_{X^\reg})$ is surjective. Take $\ph\in L^2_\al(\La^p_{\cX/A}|_{X^\reg})$ with $R\ph=\ph'.$ Then $R\Xi_A\ph=\Xi_BR\ph=R\ps.$ So $\Xi_A\ph-\ps$ lies in
\e
\hom_A(L^2_{2-\al-n}(\La^p_{\cX/A}|_{X^\reg}),A)\otimes_A(\ep)\cong
\hom_\C(L^2_{2-\al-n}(\La^p_{X^\reg}),\C)\otimes_\R(\ep).
\e
As we know already that $\Xi_\R$ is surjective, there exists $\chi\in L^2_\al(\La^p_{X^\reg})$ with $\Xi_\R\chi=-\Xi_A\ph+\ps.$ Using the ring homomorphism $\R\to A$ that defines the $\R$-algebra structure of $A$ we get an embedding $\La^p_{X^\reg} \to\La^p_{\cX/A}$ and can then regard $\chi$ as an element of $L^2_\al(\La^p_{\cX/A}|_{X^\reg}).$ Now $\Xi_A(\chi+\ph)=\ps$ as we want to prove.
\end{proof}

We generalize Proposition \ref{prop: non-exceptional Fredholm}.
\begin{prop}\l{prop: relative exceptional}
Let $X$ be a compact K\"ahler $n$-conifold, $A$ an Artin local $\R$-algebra and $\cX/A$ a locally trivial deformation of $X.$ Using Corollary \ref{cor: extending conifold metrics} choose on $\cX/A$ a K\"ahler conifold metric. Then for $\al\in\R$ the operator
\e\l{relative Lap} \De_{\cX/A}:H^k_\al(\La^{pq}_{\cX/A}|_{X^\reg})\to H^{k-2}_{\al-2}(\La^p_{\cX/A}|_{X^\reg})\e
is a Fredholm operator if and only if $\al$ is a non-exceptional value of $\De_X:H^2_\al(\La^p_{X^\reg})\to L^2_{\al-2}(\La^p_{X^\reg}).$
\end{prop}
\begin{proof}
As $\cX/A$ is locally trivial each $x\in X^\sing$ has a punctured neighbourhood on which we can identify $\De_{\cX/A}$ with $\De_X\otimes\id_A: \Ga(\La^p_{\cX/A})\otimes_\R A\to \Ga(\La^p_{\cX/A})\otimes_\R A.$ We denote by $C_x$ the cone of $(X,x)$ and say that a section of $\Ga(\La^p_{C^\reg_x})\otimes_\R A$ is {\it homogeneous of order $\al\in\R$} if it is a finite sum $\sum\ph\otimes \ps$ where $\ph$ is an order $\al$ homogeneous $p$-form on $C_x$ and $\ps$ an element of $A.$ Denote by $\De_x$ the $p$-form Laplacian of $C_x.$ By \cite[Theorem 6.2]{LM} then \eq{relative Lap} is Fredholm if and only if no non-zero order $\al$ homogeneous section $\chi\in \Ga(\La^p_{C^\reg_x})\otimes_\R A$ satisfies $(\De_x\otimes\id_A)\chi=0.$ This is equivalent to saying that $\al$ is a non-exceptional value as stated above.
\end{proof}

We generalize Proposition \ref{prop: Fredholm} to the locally trivial infinitesimal deformations.
\begin{prop}\l{prop: relative Fredholm}
Let $X$ be a compact K\"ahler $n$-conifold. Let $k\ge0$ be an integer and take $A:=\C[t]/(t^{k+1}).$ Let $\cX/A$ be a locally trivial deformation of $X;$ and using Corollary \ref{cor: extending conifold metrics}, choose on $\cX/A$ a K\"ahler conifold metric. Let $\al\in\R$ be a non-exceptional value of $\De_X:H^2_\al(\La^p_{X^\reg})\to L^2_{\al-2}(\La^p_{X^\reg})$ and take $\ph\in L^2_{\al-2}(\La^p_{\cX/A}|_{X^\reg}).$ Then $\ph$ lies in the image of the Fredholm operator $\De_{\cX/A}:H^2_\al(\La^p_{\cX/A}|_{X^\reg})\to L^2_{\al-2}(\La^p_{\cX/A}|_{X^\reg})$ if and only if $\ph\cdot\ps=0$ for every $\ps\in \ker(\De_{\cX/A}|_{X^\reg})_{2-\al-2n}.$
\end{prop}
\begin{proof}
Denote by $(\De_{\cX/A})_\al$ the map $\De_{\cX/A}:H^2_\al(\La^p_{\cX/A}|_{X^\reg})\to L^2_{\al-2}(\La^p_{\cX/A}|_{X^\reg}).$ This is an $A$-module homomorphism and its cokernel $\coker (\De_{\cX/A_k})_\al$ is of course an $A$-module. Since $(\De_{\cX/A})_\al$ is a Fredholm operator it follows that $\coker (\De_{\cX/A_k})_\al$ is a Hausdorff topological vector space of finite dimension, isomorphic to the Euclidean space. Denote by $[\coker (\De_{\cX/A})_\al]^*$ the $A$-module dual to $\coker (\De_{\cX/A})_\al.$ We show that there is an $A$-module isomorphism
\e\l{coker} [\coker (\De_{\cX/A_k})_\al]^*\cong\ker(\De_{\cX/A})_{2-\al-2n}.\e
An element of $[\coker (\De_{\cX/A_k})_\al]^*\cong\ker(\De_{\cX/A})_{2-\al-2n}$ is automatically continuous and is equivalent to a continuous $A$-module homomorphism $L^2_{\al-2}(\La^p_{\cX/A}|_{X^\reg})\to A$ which vanishes on $\im(\De_{\cX/A})_\al.$ Using Proposition \ref{prop: relative duality} regard this as an element $\ph\in L^2_{2-\al-2n}(\La^p_{\cX/A}|_{X^\reg}).$ This should then satisfy $\int_{X^\reg}g_{\cX/A}(\ph,\De_{\cX/A}\ps)\om^n_{\cX/A}=0$ for every $\ps\in L^2_\al(\La^p_{\cX/A}|_{X^\reg}).$ As $\De_{\cX/A}$ is self-adjoint with respect to this pairing, the equation is equivalent to $\De_{\cX/A}\ph=0;$ that is, $\ph\in \ker(\De_{\cX/A})_{2-\al-2n}.$ Hence we get \eq{coker}.

Recall now from hypothesis that $A=\C[t]/(t^{k+1}).$ It is known that every finitely generated $A$-module $M$ is a reflexive $A$-module, which is in fact one of the six equivalent definitions \cite[Theorem 8(1)--(6)]{Kat} of pseudo-Frobenius rings. It is also possible to give a direct proof by an induction on $k,$ as we do now. Put $A_k:=\C[t]/(t^{k+1}).$ The statement is obvious for $k=0$ because $A_0$ is a field. For $k\ge1$ let it hold for $k-1$ in place of $k.$ Let $M$ be a finitely generated $A_k$ module. There is then an $A_k$ module exact sequence $0\to tM\to M\to M/tM\to0.$ Notice that $tM$ is naturally an $A_{k-1}$ module and $M/tM$ an $A_0$ module, to which we can apply the induction hypotheses. We see then that the natural maps $tM\to (tM)^{**}$ and $M/tM\to(M/tM)^{**},$ the maps of taking double duals, are $A_k$ module isomorphisms. These with the map $M\to M^{**}$ form an $A_k$ module commutative diagram 
\begin{equation}\l{refl}\begin{tikzcd}
0\ar[r]& tM\ar[r]\ar[d,"\cong"]& M\ar[r]\ar[d]& M/tM\ar[r]\ar[d,"\cong"]&0\\
0\ar[r]& (tM)^{**}\ar[r]& M^{**}\ar[r] &(M/tM)^{**}\ar[r]&0.
\end{tikzcd}
\end{equation}
Since $A_k$ is (as in Definition \ref{dfn: Gross}) an injective $A_k$ module it follows that the bottom sequence of \eq{refl} is exact. The five lemma implies therefore that the middle vertical map $M\to M^{**}$ is an $A_k$ module isomorphism, completing the induction step.

Since $\coker (\De_{\cX/A})_\al$ is a finitely generated $A$-module we get an $A$-module isomorphism $\coker (\De_{\cX/A})_\al\cong [\coker (\De_{\cX/A})_\al]^{**}.$ Hence we get using \eq{coker} an $A$-module isomorphism $\coker (\De_{\cX/A})_\al\cong [\ker (\De_{\cX/A})_{2-\al-2n}]^*.$ The natural projection $L^2_{\al-2}(\La^p_{\cX/A}|_{X^\reg})\to \coker (\De_{\cX/A})_\al$ may then be identified with the transpose of the inclusion $\ker(\De_{\cX/A})_{2-\al-2n}\sb L^2_{2-\al-2n}(\La^p_{\cX/A}|_{X^\reg}),$ which completes the proof. 
\end{proof}

We generalize also Proposition \ref{prop: L^2 n-forms}.
\begin{prop}\l{prop: relative L^2 n-forms}
Let $X$ be a compact K\"ahler $n$-conifold, $A$ an Artin local $\R$-algebra and $\cX/A$ a locally trivial deformation of $X.$ Using Corollary \ref{cor: extending conifold metrics} choose on $\cX/A$ a K\"ahler conifold metric. Let $\ph\in L^2(\La^n_{\cX/A}|_{X^\reg})$ satisfy $\De_{\cX/A}\ph=0.$ Then $\ph$ is in fact of order $>-n;$ that is, $\ph\in L^2_{\ep-n}(\La^n_{\cX/A}|_{X^\reg})$ for some $\ep>0.$
\end{prop}
\begin{proof}
For $A=\R$ this holds by Proposition \ref{prop: L^2 n-forms}. We treat the general case by an induction on the length of $A.$ Let $0\to(\ep)\to A\to B$ be a small extension in $(\Art)_\R.$ Define a deformation $\cY/B$ by $\O_\cY:=\O_\cX\otimes_AB.$ Denote by $\ps$ the restriction of $\ph$ to $\La^n_{\cY/B}|_{X^\reg}.$ Then $\De_{\cY/B}\ps=0$ and by the induction hypothesis $\ps$ has order $>-n.$ Choosing a cochain complex isomorphism $(\La^n_{\cX/A}|_{X^\reg},\d_{\cX/A})\cong (\La^n_{X^\reg}\otimes_\R A,\d_X\otimes\id_A)$ and an $\R$-vector space splitting of $A\to B,$ lift $\ps$ to a section $\ps'$ of $\La^n_{\cX/A}|_{X^\reg}.$ Then $\ps'$ has order $>-n,$ and $\ph-\ps'$ is an $L^2$ section of $(\La^n_{\cX/A}|_{X^\reg})\otimes_\R(\ep)\cong \La^n_{X^\reg}.$ Applying Proposition \ref{prop: L^2 n-forms} to $\ph-\ps'$ we see that $\ph-\ps'$ has order $>-n.$ 
\end{proof}

We generalize Lemma \ref{lem: dph1} in two steps. The first step is as follows.
\begin{thm}\l{thm: int by parts for p le n-2}
Let $X$ be a compact K\"ahler $n$-conifold. Let $k\ge0$ be an integer; take $A:=\C[t]/(t^{k+1})$ and let $\cX/A$ be a locally trivial deformation of $X.$ Using Corollary \ref{cor: extending conifold metrics} choose on $\cX/A$ a K\"ahler conifold metric and give $X$ the induced conifold metric. Let $p\le n-2$ be an integer. Then the following holds:
\begin{equation}\l{int by parts for p le n-2}\parbox{10cm}{
For every $\al\in(2+p-2n,-p)=:I_p$ and every $\ph\in\Ga_\al(\La^p_{\cX/A}|_{X^\reg})$ with $\De_{\cX/A}\ph=0,$ we have $\d_{\cX/A}\ph=\d^*_{\cX/A}\ph=0$ and $\ph \in\bigcap_{\be\in I_p}\Ga_\be(\La^p_{\cX/A}|_{X^\reg}).$ 
}\end{equation}
\end{thm}
\begin{proof}
We prove \eq{int by parts for p le n-2} by a double induction on $p$ and $A,$ where $p$ is an outer parameter and $A$ an inner parameter. For $p<0$ the statement is automatically true. Suppose that there exists an integer $q\le n-2$ such that \eq{int by parts for p le n-2} holds for every $p\le q-1.$ We prove by the inner induction on $A$ that  \eq{int by parts for p le n-2} holds for $p=q.$ 

For  \eq{int by parts for p le n-2} with $A=\R,$ by Corollary \ref{cor: harmonic1} we have $\ph\in \Ga_{-p}(\La^p_{X^\reg}).$ As $-p\ge2-n$ the integration by parts formula $(\De_X\ph)\cdot \ph=\d_X\ph\cdot\d_X\ph+\d_X^*\ph\cdot\d_X^*\ph$ is valid. Since $\De_X\ph=0$ it follows that $\d_X\ph=\d_X^*\ph=0$ as claimed. 

For the general $A$ we do the inner induction. Let $0\to(\ep)\to A\to B\to0$ be a small extension in $(\Art)_\R.$ Define a deformation $\cY/B$ by $\O_\cY:=\O_\cX\otimes_AB.$ Denote by $R:\Ga(\La^q_{\cX/A}|_{X^\reg})\to \Ga(\La^q_{\cY/B}|_{X^\reg})$ the restriction map. Take any $\be\in I_q$ and let $\ph$ be as in \eq{int by parts for p le n-2} with $p=q.$ We show that there exists $\ph'\in\Ga_\be(\La^q_{\cX/A}|_{X^\reg})$ with $\d_{\cX/A}\ph'=\d_{\cX/A}^*\ph'=0$ and $R\ph'=R\ph.$ By the induction hypothesis we have $\d_{\cY/B}R\ph=\d_{\cY/B}^*R\ph=0$ and $R\ph\in \Ga_\be(\La^q_{\cX/A}|_{X^\reg}).$ Choose $\ps\in\Ga_\be(\La^q_{\cX/A}|_{X^\reg})$ with $\d_{\cX/A}\ps=0$ and $R\ps=R\ph.$ Notice that $\d_{\cX/A}^*\ps\in\Ga_{\be-1}(\La^{q-1}_{\cX/A}|_{X^\reg})$ and that $2-(\be+1)-2n=1-\be-2n\in (1-2n+q,-1-q)\sb (1-2n+q, 1-q)=I_{q-1}.$ By the outer induction hypothesis, if $\th\in \ker(\De_{\cX/A}^{q-1})_{1-\be-2n}$ then $\d_{\cX/A}\th=\d_{\cX/A}^*\th=0$ and $\th\in \bigcap_{\ga\in I_{q-1}}\Ga_\ga(\La^{q-1}_{\cX/A}|_{X^\reg}).$ The order estimate implies that the integration by parts formula $(\d^*_{\cX/A}\ph)\cdot\th=\ph\cdot \d_{\cX/A}\th$ makes sense. Since $\d_{\cX/A}\th=0$ it follows that $\d^*_{\cX/A}\ph\cdot\th=0.$ By Proposition \ref{prop: relative Fredholm} there exists $\chi\in \Ga_\be(\La^{q-1}_{\cX/A}|_{X^\reg})$ with $\De_{\cX/A}\chi=\d^*_{\cX/A}\ph.$ In particular, $\De_{\cX/A}\d^*_{\cX/A}\chi=0.$ As $\d^*_{\cX/A}\chi$ has order $\be-1\in (1+q-2n,-q-1)\sb I_{q-1}\sb I_{q-2}$ it follows from the outer induction hypothesis that $\d_{\cX/A}^*\chi$ is closed and co-closed; that is, $\d_{\cX/A}\d_{\cX/A}^*\chi=0.$ The equation $\De_{\cX/A}\chi=\d^*_{\cX/A}\ps$ implies in turn that $\d^*_{\cX/A}\ps=\d^*_{\cX/A}\d_{\cX/A}\chi.$ Applying the restriction map $R$ we get $\d^*_{\cY/B}R\ph=\d^*_{\cY/B}\d_{\cY/B}R\chi.$ But we know already that $\d^*_{\cY/B}R\ph=0$ so that $\d^*_{\cY/B}\d_{\cY/B}R\chi=0.$ Since $\d_{\cX/A}\d_{\cX/A}^*\chi=0$ it follows that $\d_{\cY/B}\d_{\cY/B}^*R\chi=0.$ This with $\d_{\cX/A}\d_{\cX/A}^*\chi=0$ implies $\De_{\cY/B}R\chi=0.$ By the inner induction hypothesis we have $\d_{\cY/B}^*R\chi=0.$ So $R(\ps-\d_{\cX/A}^*\chi)=R\ps=R\ph;$ that is, $\ph':=\ps-\d_{\cX/A}^*\chi$ is a lift of $R\ph$ we want.

Now $R(\ph-\ph')=0$ or equivalently $\ph-\ph'\in \Ga_\be(\La^q_{\cX/A}|_{X^\reg})\otimes_A(\ep)\cong \Ga_\be(\La^q_{X^\reg})\otimes_\R(\ep).$  So $0=\De_{\cX/A}(\ph-\ph')=\De_X(\ph-\ph').$ As we have shown in the initial step of the inner induction, we have $\d_X(\ph-\ph')=\d_X^*(\ph-\ph')=0$ and $\ph-\ph'\in  \Ga_\be(\La^q_{X^\reg})\otimes_\R(\ep)$ for every $\be\in I_q.$ Or equivalently $\d_{\cX/A}(\ph-\ph')=\d_{\cX/A}^*(\ph-\ph')=0$ and $\ph-\ph'\in  \Ga_\be(\La^q_{\cX/A}|_{X^\reg})\otimes_A(\ep).$ Since $\ph'$ is closed and co-closed and since $\ph'$ is of order $\be$ it follows that $\d_{\cX/A}\ph=\d_{\cX/A}^*(\ph-\ph')=0$ and $\ph\in  \Ga_\be(\La^q_{\cX/A}|_{X^\reg}).$ As $\be\in I_q$ has been arbitrary, we have $\ph\in\bigcap_{\be\in I_q} \Ga_\be(\La^q_{\cX/A}|_{X^\reg}),$ completing the proof.
\end{proof}
\begin{rmk}
The hypothesis $A=\C[t]/(t^{k+1})$ is relevant only to the step of applying Proposition \ref{prop: relative Fredholm}.
\end{rmk}

By the same technique we prove the following corollary, which generalizes Lemma \ref{lem: dph1}. 
\begin{cor}\l{cor: int by parts for n-1}
Let $X$ be a compact Calabi--Yau $n$-conifold. Let $k\ge0$ be an integer; and take $A:=\C[t]/(t^{k+1})$ and let $\cX/A$ be a locally trivial deformation of $X.$ Using Corollary \ref{cor: extending conifold metrics} choose on $\cX/A$ a K\"ahler conifold metric and give $X$ the induced conifold metric. Let $\ep>0$ be so small that $[1-n-\ep,1-n)$ contains no exceptional values of the $n-1$ form Laplacian over $X^\reg.$ Let $\ph\in\Ga_{1-n-\ep}(\La^{n-1}_{\cX/A}|_{X^\reg})$ satisfy $\De_{\cX/A}\ph=0.$ Then $\d_{\cX/A}\ph=\d^*_{\cX/A}\ph=0.$
\end{cor}
\begin{proof}
For $A=\R,$ near each $x\in X^\sing$ write $\ph$ as a countable sum of homogeneous $n-1$ forms on $C_x.$ Making $\ep$ small enough it follows that $\ph$ begins with an order $1-n$ homogeneous $n-1$ form $\ph'.$ By Corollary \ref{cor: order -p} we have $\d\ph'=\d^*\ph'=0.$ So $\d\ph$ and $\d^*\ph$ have order $\de-n$ for some $\de>0.$ Supposing $\ep<\de$ it follows that the integration by parts formula $\ph\cdot\De\ph=\d\ph\cdot\d\ph+\d^*\ph\cdot\d^*\ph$ is valid. As $\De\ph=0$ we have $\d\ph=\d^*\ph=0$ as required.
 
For the general $A$ we proceed by an induction (which is similar to the inner induction process in the proof of Theorem \ref{thm: int by parts for p le n-2}). Let $0\to(\ep)\to A\to B\to0$ be a small extension in $(\Art)_\R.$ Define a deformation $\cY/B$ by $\O_\cY:=\O_\cX\otimes_AB.$ Denote by $R:\Ga(\La^q_{\cX/A}|_{X^\reg})\to \Ga(\La^q_{\cY/B}|_{X^\reg})$ the restriction map. Then $R\ph\in \Ga_{1-n-\ep}(\La^{n-1}_{\cY/B}|_{X^\reg})$ and $\De_{\cY/B}R\ph=0.$ By the induction hypothesis therefore $\d_{\cY/B}R\ph=\d_{\cY/B}^*R\ph=0.$ Choose $\ps\in\Ga_{1-n-\ep}(\La^{n-1}_{\cX/A}|_{X^\reg})$ with $\d_{\cX/A}\ps=0$ and $R\ps=R\ph.$ Notice that $\d_{\cX/A}^*\ps\in \Ga_{-n-\ep}(\La^{n-2}_{\cX/A}|_{X^\reg}).$ Putting $\al:=2-n-\ep$ we have $2-\al-2n=-n+\ep.$ By Theorem \ref{thm: int by parts for p le n-2}, if $\th\in \ker(\De_{\cX/A}^{n-2})_{-n+\ep}$ then $\d_{\cX/A}\th=0.$ So $(\d^*_{\cX/A}\ps)\cdot\th=\ps\cdot\d_{\cX/A}\th=0.$ By Proposition \ref{prop: relative Fredholm} therefore there exists $\chi\in \Ga_\al(\La^{n-2}_{\cX/A}|_{X^\reg})$ with $\De_{\cX/A}\chi=\d^*_{\cX/A}\ps.$ In particular, $\De_{\cX/A}\d^*_{\cX/A}\chi=0.$ But $\d^*_{\cX/A}\chi\in \Ga_{\al-1}(\La^{n-2}_{\cX/A}|_{X^\reg})$ with $\al-1=1-n-\ep\in I_{n-3}$ in the notation of Theorem \ref{thm: int by parts for p le n-2}. So $\d^*_{\cX/A}\chi$ is closed; that is, $0=\d_{\cX/A}\d^*_{\cX/A}=\d^*_{\cX/A}(\ps-\d_{\cX/A}\chi).$ Now $\ph':=\ps-\d_{\cX/A}\chi$ is a relative closed and co-closed form whose image under $R$ is equal to that of $\ph.$ We have  $R(\ph'-\ph)=0$ and can write $\ph'-\ph\in\Ga(\La^{n-1}_{\cX/A}|_{X^\reg})\otimes_A(\ep)\cong\Ga(\La^{n-1}_{X^\reg})\otimes_\R(\ep).$ We have also $\De_X(\ph'-\ph)=\De_{\cX/A}(\ph'-\ph)=0.$ The result in the $A=\R$ case implies then $\d_X(\ph'-\ph)=\d_X(\ph-\ph')=0$ and accordingly $\d_{\cX/A}(\ph'-\ph)=\d_{\cX/A}(\ph-\ph')=0.$ Since $\d_{\cX/A}\ph'=\d_{\cX/A}\ph'=0$ it follows that $\d_{\cX/A}\ph=\d_{\cX/A}^*\ph=0,$ completing the proof.
\end{proof}

We generalize now Theorem \ref{thm: Lock}.
\begin{prop}\l{prop: relative Lockhart}
Let $X$ be a compact K\"ahler $n$-conifold, $A$ an Artin local $\R$-algebra and $\cX/A$ a locally trivial deformation of $X.$ Using Corollary \ref{cor: extending conifold metrics} choose on $\cX/A$ a K\"ahler conifold metric. Then there is an injective $A$-module homomorphism $\ker(\d_{\cX/A}+\d_{\cX/A}^*)^n_{-n}\cong {}_cH^n(X^\reg,A).$
\end{prop}
\begin{proof}
Use the cochain complex isomorphism $\La^n_{\cX/A}|_{X^\reg}\cong\La^n_{X^\reg}\otimes_\R A$ and note that the K\"ahler conifold metric is $q$-bounded in the sense of \cite[Theorem 7.4]{Lock}. This result then extends to relative forms; that is, 
\e\l{Lock inj}
L^2(\La^n_{\cX/A}|_{X^\reg})\cap 
\d_{\cX/A}[\Ga(\La^{n-1}_{\cX/A}|_{X^\reg})]\sb 
\ov{\d [\Ga_c(\La^{n-1}_{\cX/A}|_{X^\reg})]}
\e
where the bar on the right-hand side is to take the closure in $L^2(\La^n_{\cX/A}|_{X^\reg}).$ Consider now the natural $A$-module homomorphism $\ker(\d_{\cX/A}+\d_{\cX/A}^*)^n_{-n}\to {}_c H^n(X^\reg,A)$ which takes the $\d_{\cX/A}$ cohomology class. By \eq{Lock inj} this map is injective. 
\end{proof}
\begin{rmk}
The map $\ker(\d_{\cX/A}+\d_{\cX/A}^*)^n_{-n}\to {}_c H^n(X^\reg,A)$ is in fact surjective too. For $A=\R$ this follows from the Poincar\'e duality property \cite[Chapter IV, Theorem 17']{deR} and Kodaira's decomposition theorem \cite[Chapter V, Theorem 24]{deR}. We can show by an induction on $A$ that these results extend to the infinitesimal deformation $\cX/A.$ We omit the details because we shall not need them.
\end{rmk}

We end with a remark which applies to the results above in the present section. In these statements the K\"ahler condition has not played essential r\^oles in the proofs. They will extend to locally trivial deformations of compact Riemannian conifolds once we have made the right definition of such objects. We have not done that just because the current statements will do for our purpose.

\section{Proof of Theorem \ref{main thm2}}\l{sect: proof2}
We prove a lemma about relative tangent sheaves.
\begin{lem}\l{lem: rel tan}
Let $X$ be a normal complex space, $A$ an Artin local $\C$-algebra and $\cX/A$ a deformation of $X.$ Denote by $\Th_{\cX/A}$ the $\O_\cX$ module dual to $\Om^1_{\cX/A}.$ Then the following three statements hold: {\bf(i)} the natural $\O_\cX$ module homomorphism $\Th_{\cX/A}\to \io_*(\Th_{\cX/A}|_{X^\reg})$ is an isomorphism; {\bf(ii)} if $U\sb X$ is a Stein neighbourhood of $X^\sing$ then for $q=1,2,3,\dots$ we have $H^q(U,\Th_{\cX/A})=0;$ and 
{\bf(iii)} $H^1_{X^\sing}(X,\Th_{\cX/A})=0.$
\end{lem}
\begin{proof}
We prove these by an induction on the length of $A.$ For $A=\C$ it is a well-known property of the reflexive sheaf $\Th_X.$ Suppose now that $0\to(\ep)\to A\to B\to0$ is a small extension in $(\Art)_\C.$ Define a deformation $\cY/B$ by $\O_\cY:=\O_\cX\otimes_AB.$ As $\Th_{\cX/A}$ is flat over $X^\reg,$ tensoring $\Th_{\cX/A}|_{X^\reg}$ with the small extension sequence we get a short exact sequence $0\to \Th_{X^\reg}\to \Th_{\cX/A}|_{X^\reg}\to \Th_{\cY/B}|_{X^\reg}\to0.$ Pushing forward these by $\io_*$ and using the isomorphism $\io_*\Th_{X^\reg}\cong\Th_X$ we get a short exact sequence
\e\l{Th ex} 0\to \Th_X\to \io_*(\Th_{\cX/A}|_{X^\reg})\to\io_*(\Th_{\cY/B}|_{X^\reg})\to0.\e
On the other hand, using the natural transformation $\id\to\io_*\io^*$ we get a commutative diagram
\begin{equation}\l{Th}\begin{tikzcd}
& \Th_X\ar[r]\ar[d,"\id"]& \Th_{\cX/A}\ar[r]\ar[d,"\al"] &\Th_{\cY/B}\ar[r]\ar[d,"\be"]&0\\
0\ar[r]& \Th_X\ar[r]& \io_*(\Th_{\cX/A}|_{X^\reg})\ar[r] &\io_*(\Th_{\cY/B}|_{X^\reg})\ar[r]&0.
\end{tikzcd}\end{equation}
By the induction hypothesis the rightmost vertical map $\be$ is an isomorphism. Although the top left part is missing in \eq{Th} we can show directly by diagram chase that the five lemma applies to the current circumstances; that is, $\al$ is an isomorphism, which proves (i). 

Now \eq{Th ex} becomes $0\to \Th_X\to \Th_{\cX/A}\to \Th_{\cY/B}\to0.$ Let $U\sb X$ be a Stein neighbourhood of $X^\sing.$ Then for $q=1,2,3,\dots$ there is an exact sequence $H^q(U,\Th_X)\to H^q(U,\Th_{\cX/A})\to H^q(U,\Th_{\cY/B}).$ But $H^1(U,\Th_X)=0$ and by the induction hypothesis $H^1(U,\Th_{\cY/B})=0.$ So $H(U,\Th_{\cX/A})=0$ as in (ii). 

From the short exact sequence $0\to \Th_X\to \Th_{\cX/A}\to \Th_{\cY/B}\to0$ we get also an exact sequence $H^1_{X^\sing} (X,\Th_X)\to H^1_{X^\sing}(X,\Th_{\cX/A})\to H^1_{X^\sing}(X,\Th_{\cY/B}).$ But as $\Th_X$ is reflexive, the leftmost term $H^1_{X^\sing} (X,\Th_X)$ vanishes; and by the induction hypothesis, the rightmost term $H^1_{X^\sing} (X,\Th_{\cY/B})$ vanishes. Accordingly so does the middle term, which proves (iii).
\end{proof}
We generalize Corollary \ref{cor: c} as follows.
\begin{cor}\l{cor: tangent sheaf}
In the circumstances of Lemma \ref{lem: rel tan} there exists an isomorphism ${}_c H^1(X^\reg,\Th_{\cX/A})\cong H^1(X,\Th_{\cX/A}).$
\end{cor}
\begin{proof}
Lemmas \ref{lem: c} and \ref{lem: rel tan}(ii) imply that ${}_c H^1(X^\reg,\Th_{\cX/A})$ agrees with the image of the natural map $ H^1(X,\Th_{\cX/A})\to H^1(X^\reg,\Th_{\cX/A}).$ Lemma \ref{lem: rel tan}(iii) implies that the latter map is injective, from which we get the isomorphism we want.  
\end{proof}

We generalize Lemma \ref{lem: hol improve}.
\begin{lem}\l{lem: relative hol improve}
Let $n\ge3$ be an integer and $X$ a compact Calabi--Yau $n$-conifold with $X^\sing$ non-empty. Let $A$ be an Artin local $\R$-algebra and $\cX/A$ a locally trivial deformation of $X.$ Using Corollary \ref{cor: extending conifold metrics} choose on $\cX/A$ a K\"ahler conifold metric and give $X$ the induced conifold metric. 
Let $\ph$ be a section of $\La^{n1}_{\cX/A}|_{X^\reg}$ of order $>-n-1$ with $\db_{\cX/A}\ph=\db_{\cX/A}^*\ph=0.$ Then there exists $\chi\in\Ga(\La^{n0}_{\cX/A}|_{X^\reg})$ of order $>-n$ and with $\db_{\cX/A}\chi=\ph.$
\end{lem}
\begin{proof}
We show by an induction on $A$ that $H^1(X,\O_{\cX/A})=0.$ Let $0\to(\ep)\to A\to B\to0$ be a small extension in $(\Art)_\R.$ Define a deformation $\cY/B$ by $\O_\cY:=\O_\cX\otimes_AB.$ There are then exact sequences $0\to\O_X\to\O_\cX\to \O_\cY\to0$ and $0=H^1(X,\O_X)\to H^1(X,\O_\cX)\to H^1(X,\O_\cY).$ By the induction hypothesis we have $H^1(X,\O_\cY)=0.$ So  $H^1(X,\O_\cX)=0.$

Since $X$ is of depth $\ge3$ it follows that $H^1_{X^\sing}(X,\O_X)=H^2_{X^\sing}(X,\O_X)=0.$ Again by an induction on $A$ we can show that $H^1_{X^\sing}(X,\O_\cX)=H^2_{X^\sing}(X,\O_\cX)=0.$ This with $H^1(X,\O_\cX)=0$ implies that $H^1(X^\reg,\O_\cX)=0.$ So we can write $\ph=\db_{\cX/A}\chi$ with $\chi$ an $(n,0)$ form on $X^\reg.$ Then $\db_{\cX/A}^*\db_{\cX/A}\chi=0;$ or equivalently, $\De_{\cX/A}\chi=0.$ Since $\cX/A$ is locally trivial it follows that for each $x\in X^\sing$ there exists a punctured neighbourhood $U$ of it on which $\La^{n0}_{\cX/A}|_U\cong \La^{n0}_U\otimes_\R A.$ We can also identify $\De_{\cX/A}|_U$ with $\De_U\otimes \id_A$ where $\De_U$ is the Laplacian over $U\sb X.$ After choosing an $\R$-vector space basis of $A$ we can then regard $\chi$ as a finite system of harmonic $(n,0)$ forms. To these we can apply the same argument as in the proof of Lemma \ref{lem: hol improve}, which shows that $\chi$ has order $>-n$ as we want to prove.
\end{proof}

Lemma \ref{lem: p exact} generalizes readily as follows. We leave the proof to the reader as an exercise.
\begin{prop}\l{prop: relative exact}
Let $X$ be a compact K\"ahler $n$-conifold, $A$ an Artin local $\R$-algebra and $\cX/A$ a locally trivial deformation of $X.$ Using Corollary \ref{cor: extending conifold metrics} choose on $\cX/A$ a K\"ahler conifold metric. Let $p$ be an integer and $\ph$ a section of order $>-p$ of $\La^p_{\cX/A}|_{X^\reg}.$ Then every $x\in X^\sing$ has a punctured neighbourhood $U^\reg$ on which $\ph$ is $\d_{\cX/A}$ exact. \qed
\end{prop}

We prove an integration by parts formula for relative harmonic forms.
\begin{lem}\l{lem: harm cl}
Let $X$ be a compact Calabi--Yau $n$-conifold, $A$ an Artin local $\R$-algebra and $\cX/A$ a deformation of $X.$ Using Corollary \ref{cor: extending conifold metrics} choose on $\cX/A$ a K\"ahler conifold metric and give $X$ the induced conifold metric. Let $p,q\in\Z$ and $\al\in\R$ be such that for every $\chi\in \ker\De^{pq}_\al$ we have $\d\chi=0.$ Then for every $\ph\in \ker(\De_{\cX/A}^{pq})_\al$ we have $\d_{\cX/A}\ph=0.$
\end{lem}
\begin{proof}
We prove this by an induction on the length of $A.$ Let $0\to(\ep)\to A\to B\to0$ be a small extension in $(\Art)_\R.$ Define a deformation $\cY/B$ by $\O_\cY:=\O_\cX\otimes_AB.$ Denote by $\ps$ the restriction to $\cY/B$ of $\ph.$ Then $\ps$ is of order $\al$ and satisfies $\De_{\cY/B}\ps=0.$ By the induction hypothesis we have $\d_{\cY/B}\ph=0.$ Choosing a cochain complex isomorphism $(\La^{p+q}_{\cX/A}|_{X^\reg},\d_{\cX/A})\cong (\La^{p+q}_{X^\reg}\otimes_\R A,\d_X\otimes\id_A)$ and an $\R$-vector space splitting of $A\to B,$ lift $\ps$ to a section $\ps'$ of $\La^{p+q}_{\cX/A}|_{X^\reg}.$ Then $\ps'$ has order $\al$ and satisfies $\d_{\cX/A}\ps'=0.$ Notice that $\ph-\ps'$ is a section of $(\La^{p+q}_{\cX/A}|_{X^\reg})\otimes_\R(\ep)\cong \La^{p+q}_{X^\reg}$ and satisfies $\De_X(\ph-\ps')=0.$ By hypothesis we have $\d_X(\ph-\ps')=0.$ Regarding $\ph-\ps'$ as a section of $(\La^{p+q}_{\cX/A}|_{X^\reg})\otimes_\R(\ep)$ we have $\d_{\cX/A}(\ph-\ps')=0.$ So $\d_{\cX/A}\ph=0.$
\end{proof}
\begin{rmk}
Suppose that every $\chi\in \ker\De^{pq}_\al$ satisfies also $\d_X^*\chi=0.$ It will then be reasonable to ask whether $\d^*_{\cX/A}\ph=0.$ For our purpose however we shall not need an answer to this question and therefore not discuss it further. 
\end{rmk}

We generalize now Theorem \ref{thm: inj}.
\begin{thm}\l{thm: inj A}
Let $n\ge3$ be an integer and $X$ a compact Calabi--Yau $n$-conifold with $X^\sing$ non-empty. 
Let $k\ge0$ be an integer and take $A:=\C[t]/(t^{k+1}).$ Let $\cX/A$ be a locally trivial deformation of $X;$ and using Corollary \ref{cor: extending conifold metrics}, choose on $\cX/A$ a K\"ahler conifold metric. Then there is an injective $A$-module homomorphism $\ker(\db_{\cX/A}+\db_{\cX/A}^*)^{n-1\,1}_{-n}\to\gr^{n-1}{}_c H^n(X^\reg,A).$
\end{thm}
\begin{proof}
We show first that for every $\ph\in \ker(\db_{\cX/A}+\db_{\cX/A}^*)^{n-1\,1}_{-n}$ there exists on $X^\reg$ a relative $(n,0)$ form $\ps$ of order $>-n$ such that $\db_{\cX/A}\ps=\bd_{\cX/A}\ph.$ Let $\ep\in(0,1)$ be so small that $\ph$ has order $\ep-n.$ Then $\db_{\cX/A}^*\bd_{\cX/A}\ph$ is a relative $(n,0)$ form of order $\ep-n-2.$ Put $\al:=\ep-n$ and take $\chi\in \ker(\De_{\cX/A})^{n-1\,1}_{2-\al-2n}=\ker(\De_{\cX/A})^{n-1\,1}_{2-\ep-n}.$ 
Lemma \ref{lem: harm cl} implies then that $\d_{\cX/A}\chi=0.$ As $\chi\in\Ga(\La^{n-1\,1}_{\cX/A})$ we have $\db_{\cX/A}\chi=0$ and accordingly $\chi\cdot\db_{\cX/A}^*\bd_{\cX/A}\ph=0.$ So by Proposition \ref{prop: relative Fredholm} there exists an $(n,0)$ form $\ta$ of order $\al=\ep-n$ and with $\frac12\De_{\cX/A}\th=\db_{\cX/A}^*\bd_{\cX/A}\ph.$ Then $\bd_{\cX/A}\ph-\db_{\cX/A}\ta$ is $\db_{\cX/A}$ closed and co-closed. By Lemma \ref{lem: relative hol improve} there exists $\si\in\Ga(\La^{n0}_{\cX/A}|_{X^\reg})$ of order $>-n$ and with $\db_{\cX/A}\si=\ph.$ Put $\ps:=\si+\ta,$ which is a section of $\chi\in\Ga(\La^{n0}_{\cX/A}|_{X^\reg})$ of order $>-n.$ Moreover, $\d_{\cX/A}(\ph-\ps)=\db_{\cX/A}\ph+\bd_{\cX/A}\ph-\db_{\cX/A}\ps=0.$ 

By Proposition \ref{prop: relative exact} the de Rham cohomology class $[\ph-\ps]$ lies in ${}_cH^n(X^\reg,A).$ Recalling again that $\ph$ is a relative $(n-1,1)$ form and $\ps$ a relative $(n,0)$ form, we can also define $[\ph-\ps]\in \gr^{n-1}{}_cH^n(X^\reg,A).$ Notice that $[\ph-\ps]\in {}_cH^n(X^\reg,A)$ may depend on the choice of $\ps$ but that $[\ph-\ps]\in \gr^{n-1}{}_cH^n(X^\reg,A)$ does not. The map $\ker(\db_{\cX/A}+\db_{\cX/A}^*)^{n-1\,1}_{-n}\to\gr^{n-1}{}_c H^n(X^\reg,A)$ is thus well defined. This is of course an $A$-module homomorphism. We show now that it is injective.

Let $\ph\in \ker(\db_{\cX/A}+\db_{\cX/A}^*)^{n-1\,1}_{-n}$ be an element of the kernel of the map $\ker(\db_{\cX/A}+\db_{\cX/A}^*)^{n-1\,1}_{-n}\to\gr^{n-1}{}_c H^n(X^\reg,\Om^\bt_{\cX/A});$ that is, there exists an $(n,0)$ form $\ps$ of order $>-n,$ with $\bd_{\cX/A}\ps=\db_{\cX/A}\ph$ and with $[\ph-\ps]=0\in  \gr^{n-1}{}_c H^n(X^\reg,A).$ Then there exists another $(n,0)$ form $\ps'$ with $\db_{\cX/A}\ps'=0$ and $[\ph-\ps-\ps']=0\in H^n(X^\reg,A).$

Recall from Lemma \ref{lem: n0} that there exists a nowhere-vanishing section $\Ph$ of $\Om^n_{\cX/A}|_{X^\reg}.$ We show that this is $L^2.$ Since $\cX/A$ is locally trivial it follows that for each $x\in X^\sing$ there exists a punctured neighbourhood $U$ of it on which $\Om^n_{\cX/A}|_U\cong \Om^n_U\otimes_\R A.$ So after choosing an $\R$-vector space basis of $A$ we can regard $\Ph$ as a finite system of sections of $\Om^n_U.$ But these are automatically $L^2$ and accordingly so is $\Ph.$

Write $\ps'=f\Ph$ with $f$ some holomorphic function $X^\reg\to A.$ Since $H^1_{X^\sing}(X,\O_\cX)=0$ it follows that $f$ extends to $X^\sing.$ In particular, $f$ is bounded and $f\Ph=\ps'$ is $L^2.$ Proposition \ref{prop: relative L^2 n-forms} implies that $\ps'$ has order $>-n.$  

Thus $\ph-\ps-\ps'$ has order $>-n;$ and accordingly, $\d_{\cX/A}^*(\ph-\ps-\ps')$ has order $>-n-1.$ Take $\ep>0$ and put $\be:=1+\ep-n.$ Take $\chi\in\ker(\De_{\cX/A}^{n-1})_{2-\be-n}=\ker(\De_{\cX/A}^{n-1})_{1-n-\ep}.$ Letting $\ep>0$ be small enough and using Corollary \ref{cor: int by parts for n-1} we see that $\d_{\cX/A}\chi=\d_{\cX/A}^*\chi=0$ and that $\chi\cdot\d_{\cX/A}^*(\ph-\ps-\ps')=0.$ By Proposition \ref{prop: relative Fredholm} there exists on $X^\reg$ a relative $n-1$ form $\th$ of order $\be=1+\ep-n$ and with $\De_{\cX/A}\th=\d_{\cX/A}^*(\ph-\ps-\ps').$ Then $\d_{\cX/A}\d_{\cX/A}^*\th=\d^*(\ph-\ps-\ps'-\d_{\cX/A}\th).$ Applying $\d_{\cX/A}^*$ we get $\d_{\cX/A}^*\d_{\cX/A}\d_{\cX/A}^*\th=0;$ that is, $\d_{\cX/A}^*\th$ is a relative harmonic $n-2$ form of order $\ep-n.$ By Theorem \ref{thm: int by parts for p le n-2} however $\d_{\cX/A}^*\th$ is closed and co-closed; that is, $\d_{\cX/A}\d_{\cX/A}^*\th=0.$ Or equivalently $\d_{\cX/A}^*(\ph-\ps-\ps'-\d_{\cX/A}\th)=0.$ Thus $\ph-\ps-\ps'-\d_{\cX/A}\th\in \ker(\d_{\cX/A}+\d_{\cX/A}^*)^n_{-n}.$ Since $[\ph-\ps-\ps']=0\in {}_cH^n(X^\reg,A)\cong \ker(\d_{\cX/A}+\d_{\cX/A}^*)^n_{-n}$ it follows by Theorem \ref{prop: relative Lockhart} that $\ph-\ps-\ps'-\d_{\cX/A}\th=0;$ that is, $\ph-\ps-\ps'=\d_{\cX/A}\th.$ We now generalize Lemma \ref{lem: p+q=n-1} as follows.

\begin{lem}\l{lem: relative p+q=n-1}
Let $p,q\in\Z$ be such that $p+q=n-1.$ Let $\chi$ be a relative $(p,q)$ form of order $>1-n$ and with $\db_{\cX/A}\chi=0.$ Then there exists a relative $(p,q-1)$ form $\ze$ of order $2-n$ and such that $\chi-\db\ze$ is $\d_{\cX/A}$ closed and co-closed.
\end{lem}

If we use Proposition \ref{prop: relative Fredholm} and Theorem \ref{thm: int by parts for p le n-2} the proof of Lemma \ref{lem: relative p+q=n-1} will be similar to that of Lemma \ref{lem: p+q=n-1}, which we leave to the reader as an exercise. For $(p,q)$ with $p+q=n-1$ denote by $\th^{pq}$ the $(p,q)$ component of the $n-1$ form $\th.$ So $\th=\th^{n-1\,0}+\th^{n-2\,1}.$ Since $\d\th=\ph-\ps-\ps'$ has only $(n,0)$ and $(n-1,1)$ components it follows that $\db_{\cX/A}\th^{n-2\,1}=0.$ Applying Lemma \ref{lem: relative p+q=n-1} to $\th^{n-2\,1}$ we find some $(n-2,0)$ form $\ze$ of order $2-n$ such that $\th^{n-2\,1}-\db_{\cX/A}\ze$ is $\d_{\cX/A}$ closed. In particular, $\partial_{\cX/A}\th^{n-2\,1}= \partial_{\cX/A}\db_{\cX/A}\ze=-\db_{\cX/A}\partial_{\cX/A}\ze.$ Put $\et:=\th^{n-1\,0}-\db_{\cX/A}\ze.$ Then $\db_{\cX/A}\et=\db_{\cX/A}\th^{n-1\,0}\th+\partial_{\cX/A}\th^{n-2\,1},$ which is exactly the $(n-1,1)$ component of $\d\th=\ph-\ps-\ps',$ that is, $\ph.$ Since $\et$ is of order $1-n$  and $\ph$ of order $>-n$ it follows that the integration by parts formula $\ph\cdot\ph=\db_{\cX/A}^*\ph\cdot\et=0$ makes sense. But by the initial hypothesis we have $\db_{\cX/A}^*\ph=0$ and so $\ph\cdot\ph=0.$ We can then verify by an induction on $A$ that the inner product is non-degenerate or more precisely that $\ph=0.$ The map $\ker(\db+\db^*)^{n-1\,1}_{-n}\to\gr^{n-1}{}_c H^n(X^\reg,\Om^\bt_{X^\reg})$ is thus injective.
\end{proof}

We next generalize Corollary \ref{cor: inj}.
\begin{cor}\l{cor: inj A}
Let $X$ be a compact Calabi--Yau $n$-conifold. Let $k\ge0$ be an integer and take $A:=\C[t]/(t^{k+1}).$ Let $\cX/A$ be a locally trivial deformation of $X;$ and using Corollary \ref{cor: extending conifold metrics}, choose on $\cX/A$ a K\"ahler conifold metric. The natural projection $\ker(\db_{\cX/A}+\db^*_{\cX/A})^{n-1\,1}_{-n}\to H^1(X^\reg,\Om^{n-1}_{\cX/A})$ is then injective.
\end{cor}
\begin{proof}
The proof is similar to that of Corollary \ref{cor: inj}, which we leave to the reader as an exercise.
\end{proof}

We generalize also Lemma \ref{lem: surj}. 
\begin{lem}\l{lem: surj A}
Let $n\ge3$ be an integer and $X$ a compact Calabi--Yau $n$-conifold with $X^\sing$ non-empty. 
Let $k\ge0$ be an integer and take $A:=\C[t]/(t^{k+1}).$ Let $\cX/A$ be a locally trivial deformation of $X;$ and using Corollary \ref{cor: extending conifold metrics}, choose on $\cX/A$ a K\"ahler conifold metric. Then ${}_c H^1(X^\reg,\Om^{n-1}_{X^\reg})$ lies in the image of the natural projection $\ker(\db_{\cX/A}+\db^*_{\cX/A})^{n-1\,1}_{-n}\to H^1(X^\reg,\Om^{n-1}_{\cX/A}).$
\end{lem}
\begin{proof}
If we use Corollary \ref{cor: int by parts for n-1}, Proposition \ref{prop: relative Fredholm} and Theorem \ref{thm: int by parts for p le n-2} then the proof of this lemma will be similar to that of Lemma \ref{lem: surj}, which we leave to the reader as an exercise.
\end{proof}

We finally generalize Theorem \ref{thm: harm n-1 1 version2}.
\begin{thm}\l{thm: harm n-1 1 version2A}
Let $n\ge3$ be an integer and $X$ a compact Calabi--Yau $n$-conifold with $X^\sing$ non-empty and $H^2_{X^\sing}(X,\Om^{n-2}_X)=0.$ 
Let $k\ge0$ be an integer and take $A:=\C[t]/(t^{k+1}).$ Let $\cX/A$ be a locally trivial deformation of $X;$ and using Corollary \ref{cor: extending conifold metrics}, choose on $\cX/A$ a K\"ahler conifold metric. Then there exist $A$-module isomorphisms
\e H^1(X,\Th_{\cX/A})\cong \gr^{n-1}{}_c H^n(X^\reg,A)\cong \ker(\db_{\cX/A}+\db_{\cX/A}^*)^{n-1\,1}_{-n}.\e
\end{thm}
\begin{proof}
By hypothesis every $x\in X^\sing$ has a punctured neighbourhood $U\sb X^\reg$ such that $H^1(U,\Om^{n-2}_X)\cong H^2_x(X,\Om^{n-2}_X)=0.$ We can then verify by an induction on $A$ that $H^1(U,\Om^{n-2}_{\cX/A})=0.$ If we use Corollary \ref{cor: inj A} and Lemma \ref{lem: surj A} the rest of the proof will be similar to that of Theorem \ref{thm: harm n-1 1 version2}, which we leave to the reader as an exercise.
\end{proof}

Using Theorem \ref{thm: harm n-1 1 version2A} we prove
\begin{thm}\l{thm: harm n-1 1 version3}
Let $n\ge3$ be an integer and $X$ a compact Calabi--Yau $n$-conifold with $X^\sing$ non-empty and $H^2_{X^\sing}(X,\Om^{n-2}_X)=0.$ Let $k\ge1$ be an integer and put $A_k:=\C[t]/(t^{k+1}).$ Let $X_k/A_k$ a locally trivial deformation of $X,$ and $X_{k-1}/A_{k-1}$ the deformation of $X$ defined by $\O_{X_{k-1}}:=\O_{X_k}\otimes_{A_k}A_{k-1}.$ The natural map $H^1(X,\Th_{X_k/A_k})\to H^1(X,\Th_{X_{k-1}/A_{k-1}})$ is then surjective.
\end{thm}
\begin{proof}
We show by an induction on $k$ that
\e\l{dim A} \dim_\C H^1(X,\Th_{X_k/A_k})\le (k+1)\dim_\C H^1(X,\io_*\Th_X).\e
This holds automatically for $k=0.$ The small extension $0\to(t^k)\to A_k\to A_{k-1}\to0$ induces an $A_k$ module sheaf exact sequence $0\to\io_*\Th_{X^\reg}\to \io_*(\Th_{X_k/A_k}|_{X^\reg})\to \io_*(\Th_{X_{k-1}/A_{k-1}}|_{X^\reg})\to0.$ This is by Lemma \ref{lem: rel tan}(i) equivalent to an exact sequence $0\to\Th_X\to \Th_{X_k/A_k}\to \Th_{X_{k-1}/A_{k-1}}\to0.$ Passing to the cohomology groups we get an $A_k$ module exact sequence $H^1(X,\Th_{X^\reg})\to H^1(X,\Th_{X_k/A_k})\to H^1(X,\Th_{X_{k-1}/A_{k-1}}).$ The latter implies
\e\l{dim A1} \dim_\C H^1(X,\Th_{X_k/A_k})\le  \dim_\C H^1(X,\Th_{X_{k-1}/A_{k-1}})+\dim_\C H^1(X,\Th_X).\e
By the induction hypothesis we have $\dim_\C H^1(X,\Th_{X_{k-1}/A_{k-1}})\le k \dim_\C H^1(X,\Th_X),$ which with \eq{dim A1} implies \eq{dim A}.

We show also by an induction on $k$ that
\e\l{dim A3} \dim_\C \ker(\db_{X_k/A_k}+\db_{X_k/A_k}^*)^{n-1\,1}_{-n}\ge (k+1)\dim_\C \ker(\db_X+\db_X^*)^{n-1\,1}_{-n}.\e
This holds again automatically for $k=0.$ We show that the restriction map $R:L^2(\La^{n-1\,1}_{X_k/A_k})\to L^2(\La^{n-1\,1}_{X_{k-1}/A_{k-1}})$ maps $\ker(\db_{X_k/A_k}+\db_{X_k/A_k}^*)^{n-1\,1}_{-n}$ surjectively to $\ker(\db_{X_{k-1}/A_{k-1}}+\db_{X_{k-1}/A_{k-1}}^*)^{n-1\,1}_{-n}.$ Let $\ph \in \ker(\db_{X_{k-1}/A_{k-1}}+\db_{X_{k-1}/A_{k-1}}^*)^{n-1\,1}_{-n}$ be any element and regard it as a section of $\La^n_{X_{k-1}/A_{k-1}}|_{X^\reg}\cong\La^n_{X^\reg}\otimes_\R A_{k-1}.$ Using a splitting of $A_k\to A_{k-1}$ lift $\ph$ to a section $\ps$ of $\La^n_{X^\reg}\otimes_\R A_k\cong \La^n_{X_k/A_k}|_{X^\reg}.$ Then $\ps$ is $L^2$ and we can project it to an element $\chi\in\ker(\db_{X_k/A_k}+\db_{X_k/A_k}^*)^{n-1\,1}_{-n}.$ 
We look now at the $L^2$ orthogonal decomposition
\e\l{decomp k}
\begin{split}
\!\!\!\!\!\!\!\!\!\!\! L^2(\La^{n-1\,1}_{X_k/A_k}|_{X^\reg})=\ker(\db_{X_k/A_k}+\db_{X_k/A_k}^*)^{n-1\,1}_{-n}\oplus \ov{\db_{X_k/A_k}(\Ga_c(\La^{n-1\,0}_{X_k/A_k}|_{X^\reg}))}\\
\oplus \ov{\db_{X_k/A_k}^*(\Ga_c(\La^{n-1\,2}_{X_k/A_k}|_{X^\reg}))}
\end{split}\e
and its version with $k-1$ in place of $k.$ The projection $A_k\to A_{k-1}$ induces the four restriction maps from the respective four terms on \eq{decomp k} to their versions with $k-1$ in place of $k.$ These moreover commute with the equality in \eq{decomp k} and its version with $k-1$ in place of $k.$
So $R\chi=\ph$ and the restriction map $\ker(\db_{X_k/A_k}+\db_{X_k/A_k}^*)^{n-1\,1}_{-n}\to\ker(\db_{X_{k-1}/A_{k-1}}+\db_{X_{k-1}/A_{k-1}}^*)^{n-1\,1}_{-n}$ is surjective. As the kernel of this map contains $\ker(\db_X+\db_X^*)^{n-1\,1}_{-n}$ it follows from the induction that \eq{dim A3} holds.

Combining \eq{dim A3}, \eq{dim A} and Theorem \ref{thm: harm n-1 1 version2A} we see that the all the inequalities in \eq{dim A}--\eq{dim A3} are in fact equalities. In particular, the map $H^1(X,\Th_{X_k/A_k})\to H^1(X,\Th_{X_{k-1}/A_{k-1}})$ is surjective as we have to prove.
\end{proof}
\begin{rmk}
For $n=2,$ as mentioned in the introduction it is known that the locally trivial deformations are unobstructed. But we can also verify directly the following statement: let $X$ be a compact Calabi--Yau $2$-conifold, $0\to(\ep)\to A\to B\to0$ a small extension in $(\Art)_\R,$ $\cX/A$ a locally trivial deformation of $X,$ and $\cY/B$ the deformation of $X$ defined by $\O_\cY:=\O_\cX\otimes_A B;$ then the natural map $H^1(X,\Th_{\cX/A})\to H^1(X,\Th_{\cY/B})$ is surjective. This may be done using the degenerate spectral sequence as in Remark \ref{rmk: n=2}.
\end{rmk}

By Theorem \ref{thm: harm n-1 1 version3} we can use the $T^1$ lift method as in Example \ref{ex: loc triv def}, which proves Theorem \ref{main thm2}. \qed

\section{Remarks on Theorem \ref{main thm2}}\l{sect: remarks}
The hypothesis $H^2_{X^\sing}(X,\Om^{n-2})=0$ in Theorem \ref{main thm2} is equivalent to saying that for each $x\in X^\sing,$ if $U\sb X$ is a Stein neighbourhood of $x$ then we should have $H^1(U\-\{x\},\Om^{n-2}_U)=0.$ We give in Proposition \ref{prop: toric} below a simple example for which this holds. We begin by recalling the standard facts we will use about Fano manifolds.

\begin{dfn}\l{dfn: Fano}
A {\it Fano} manifold is a compact connected complex manifold $Y$ whose anti-canonical bundle $K_Y^*$ is ample. For a Fano manifold $Y$ we define its {\it index} $\ind Y\in\{1,2,3,\dots\}.$ Recall (for instance from \cite[Theorem 3.6.9]{BG}) that $Y$ is simply connected. By the universal coefficient theorem there is a group isomorphism $H^2(Y,\Z)\cong\hom(H_2(Y,\Z),\Z).$ The first Chern class $c_1(K_Y)$ defines now a non-zero homomorphism $H_2(Y,\Z)\to \Z$ whose image may be written uniquely as $k\Z$ with $k>0.$ We then set $\ind Y:=k.$ In particular, we can define an element $\frac1{\ind Y}c_1(K_Y)\in H^2(Y,\Z).$

By the Kodaira vanishing theorem we have $H^1(Y,\O_Y)=H^2(Y,\O_Y)=0$ and so the Picard group of $Y$ is isomorphic to $H^2(Y,\Z).$ Let $L\to Y$ be a holomophic line bundle with $c_1(L)= \frac1{\ind Y}c_1(K_Y)\in H^2(Y,\Z).$ Recall from the proof of {\rm\cite[Theorem 3.1]{vC}} that there exist a normal Stein space $C$ with a unique singular point $\vx\in C$ and a complex space morphism $L\to C$ which maps the zero-section $Y\sb L$ onto $\vx$ and maps $L\-Y$ biholomorphically onto $C^\reg.$ Choose now on the holomorphic line bundle $L\to Y$ a Hermitian metric $h$ whose curvature $(1,1)$ form represents $c_1(L).$ Denote by $r:L\to[0,\iy)$ the distance relative to $h$ from the zero-section $Y\sb L,$ and by $A$ the connection $1$-form relative to $h$ on the principal $\C^*$ bundle $L\-Y\to Y.$ Restricting $r$ to $L\-Y\cong C^\reg$ we get on $C^\reg$ a $2$-form $-\frac12\d(r^2A).$ This defines then on $C^\reg$ a K\"ahler cone structure with contact $1$-form $-A.$ The link $C^\lk$ has the corresponding Sasakian structure.
\end{dfn}
We prove that $(C,\vx)$ above has the following properties.
\begin{lem}
Let $Y$ be a Fano manifold and $(L,Y)\to(C,\vx)$ as in Definition \ref{dfn: Fano}. Then $C^\reg\cong L\-Y$ is simply connected; its canonical bundle $K_{C^\reg}\to C^\reg$ is a trivial holomorphic line bundle; and $(C,\vx)$ is a rational singularity. 
\end{lem}
\begin{proof}
We use the homotopy exact sequence $\pi_2(L\-Y)\to \pi_2Y\to \pi_1\C^*\to \pi_1(L\-Y)\to \pi_1Y=\{1\}.$ If we choose a point $y\in Y$ and denote its fibre by $F\sb L\-Y$ then there is a group isomorphism $\pi_2Y=\pi_2(Y,y)\cong \pi_2(L\-Y,F).$ So every element of $\pi_2(Y,y)$ may be represented by a map $\al:(D^2,\partial D^2)\to (L\-Y,F).$ The map $\pi_2(Y,y)\to\pi_1F$ is then defined by restricting $\al$ to $\partial D^2.$ On the other hand, define the $1$-form $A$ as in Definition \ref{dfn: Fano}; and restricting it to the fibre $F\cong\C^*$ we get an isomorphism $\pi_1F\cong\Z.$ The composite map $\pi_2Y\to\pi_1F\cong\Z$ is thus defined by integrating $A$ over each $\al.$ By Stokes' theorem the map $\pi_2Y\to \Z$ is the pairing with $c_1(L).$ Since $Y$ is simply connected we get, by the Hurewicz and universal coefficient theorems, a group isomorphism $H^2(Y,\Z)\cong \hom(\pi_2Y,\Z).$ The map $\pi_2Y\to \Z$ thus contains all the information of $c_1(L)\in H^2(Y,\Z).$ The definition of $\ind Y$ implies now that the group homomorphism $\pi_2Y\to\Z$ is surjective. Returning to the homotopy exact sequence above we see that $\pi_1(L\-Y)$ is a trivial group or equivalently that $L\-Y$ is simply connected.

Notice that the cohomology class $c_1(K_Y^*)=-\ind(Y) c_1(L)$ may be represented by the closed $2$-form $-\ind (Y)\d A.$ Denote by $\pi: C^\lk\to Y$ the composite map $C^\lk\sb C^\reg\cong L\-Y\sb L \to Y$ ending with the bundle projection. We call $\pi^*c_1(K_Y^*)$ the {\it basic Chern class of the Sasakian $C^\lk.$} This may be represented by the $2$-form $-\ind (Y)\d A|_{C^\lk}$ and thus satisfies the condition (i) of \cite[Proposition 2.5]{vC}. The proof of the same proposition shows that there exists a function $f:Y\to(0,\iy),$ identified with its pull-back to $C^\reg,$ such that the holomorphic line bundle $K_{C^\reg}\to C^\reg$ is flat with respect to the conformal change $fg$ of the K\"ahler cone metric $g$ of $C^\reg.$ Since $C^\reg$ is simply connected we get a nowhere-vanishing parallel section $\Om$ of $K_{C^\reg}.$ This is in particular holomorphic and so $K_{C^\reg}$ is holomorphically trivial. Also $\Om$ has constant pointwise norm with respect to $fg,$ and $f$ is the pull-back of a function on $Y.$ These imply that $\Om$ is $L^2$ on a punctured neighbourhood of $\vx\in C.$ Hence it follows as in Remark \ref{rmk: L^2} that $(C,\vx)$ is rational.
\end{proof}

We turn now to the case in which the Fano manifolds are toric.
\begin{prop}\l{prop: toric}
Let $n\ge5$ be an integer and $Y$ a toric Fano manifold of dimension $n-1.$ Let $(L,Y)\to (C,\vx)$ be as in Definition \ref{dfn: Fano}, and $U\sb C$ a Stein neighbourhood of $\vx.$ Then $H^1(U\-\{\vx\},\Om^{n-2}_C)=0$ and $H^1(U\-\{\vx\},\Om^{n-1}_C)=0.$
\end{prop}
\begin{proof}
For $k\in\Z$ denote by $L^{-k}$ the invertible sheaf on $Y$ corresponding to the line bundle $(L^*)^{\otimes k}.$ We define on $Y$ an $\O_L|_Y$ module sheaf exact sequence 
$0\to \Om^1_Y\otimes_{\O_Y} (\O_L|_Y)\to \Om^1_L|_Y \to L^{-1}\otimes_{\O_Y}(\O_L|_Y)\to0.$ The injective map $\Om^1_Y\otimes_{\O_Y} (\O_L|_Y)\to \Om^1_L|_Y$ is defined by $\d z\otimes w\mapsto w\d z$ where $z$ is a local function on $Y,$ and $w$ a local function on $L.$ The surjective map $\Om^1_L|_Y \to L^{-1}\otimes_{\O_Y}\O_L|_Y$ is defined as follows. Let $\om$ be a $1$-form defined near $y\in Y.$ For each point $x\in L$ in the fibre over $y$ choose in the same fibre a path from $y$ to $x$ and integrate $\om$ over it. Then we get an element of $L^{-1}\otimes_{\O_Y}\O_L|_Y.$ This map $\Om^1_L|_Y \to L^{-1}\otimes_{\O_Y}\O_L|_Y$ has a splitting $L^{-1}\otimes_{\O_Y}\io^{-1}\O_L\to \Om^1_L$ (which differentiates the elements of $L^{-1}\otimes_{\O_Y}\io^{-1}\O_L$). So there is an isomorphism $\Om^1_L|_Y\cong (\Om^1_Y\otimes_{\O_Y}\O_L|_Y)\oplus (L^{-1}\otimes_{\O_Y}\O_L|_Y).$ Hence we get an isomorphism 
\e\l{Om^p}\Om^p_L|_Y\cong (\Om^p_Y\otimes_{\O_Y}\O_L|_Y)\oplus (\Om^{p-1}_Y\otimes_{\O_Y} L^{-1}\otimes_{\O_Y}\O_L|_Y). \e
Denote by $\w$ the functor of taking completions on $Y\sb L.$ Then $\O_L|_Y^\w$ is an $\O_Y$ algebra sheaf on $Y$ isomorphic to $\prod_{k=0}^\iy L^{-k}.$  From \eq{Om^p} we get an $\O_Y$ module sheaf isomorphism
\ea\l{line bundle}
\Om^p_L|_Y^\w\cong (\Om^p_Y\otimes_{\O_Y}\O_L|_Y^\w)\oplus (\Om^{p-1}_Y\otimes_{\O_Y} L^{-1}\otimes_{\O_Y}\O_L|_Y^\w)\\
\cong\prod_{k=0}^\iy(\Om^p_Y\otimes_{\O_Y} L^{-k})\oplus\prod_{k=1}^\iy (\Om^p_Y\otimes_{\O_Y} L^{-k}).
\ea 
Recall (for instance from \cite[Theorem 9.3.1]{CLS}) that the toric variety $Y$ has the Bott vanishing property, which in the current circumstances implies that for $q\ge1$ we have $H^q(Y,\Om^p_Y\otimes L^{-k})=0.$ Denote by $V\sb L$ the pull-back of $U\sb C$ by the contraction map $(L,Y)\to(C,\vx).$ Then $V$ is a $1$-convex space; and by \cite[Chapter V, Theorem 2.1.7]{GPR}, for $q\ge1$ we have $H^q(V,\Om^p_V)\cong H^q(Y,\Om^p_L|_Y^\w)\cong H^q(Y,\Om^p_Y).$ As $Y$ is toric, by \cite[Theorem 9.4.11]{CLS} we have $H^q(Y,\Om^p_Y)=0$ for $p\ne q.$ Since $n\ge5$ it follows that $H^1(Y,\Om^{n-2}_Y)=H^{n-2}(Y,\Om^2_Y)=0.$ By the duality theorem \cite{Kar} we have also $H^2_Y(V,\Om^{n-2}_L)=0.$ The long exact sequence $0=H^1(V,\Om^{n-2}_L)\cong H^1(V\-Y,\Om^{n-2}_L)\to H^2_Y(V,\Om^{n-2}_L)=0$ implies now $H^1(V\-Y,\Om^{n-2}_L).$ Identifying $V\-Y$ with $U\-\{\vx\}$ we see that $H^1_\vx(U\-\{\vx\},\Om^{n-2}_C)=0.$ In the same way we can verify also that $H^1(U\-\{\vx\},\Om^{n-1}_C)=0.$ 
\end{proof}
\begin{rmk}
The algebraic scheme version of \eq{line bundle} is proved in \cite[Lemma A.1]{SVV}.
\end{rmk}

As in \cite{Schless2} the vector space $H^1(U\-\{\vx\},\Th_C)\cong H^1(U\-\{\vx\},\Om^{n-1}_C)=0$ parametrizes the infinitesimal deformations of $(C,\vx)$ which must therefore be trivial. In particular, the deformations of $(C,\vx)$ are unobstructed. If every singularity of a compact Calabi--Yau conifold $X$ is of this kind then the locally trivial deformations of $X$ are equivalent to the full deformations of $X.$ The difference from Theorem \ref{main thm1} is that we do not need the hypothesis that either $b_{n-1}(C^\reg)=0$ or $b_{n-2}(C^\reg)=0.$

It does not seem easy to find examples of $(C,\vx)$ with $H^1(U\-\{\vx\},\Om^{n-2}_C)=0.$ We end by showing that even the ordinary double points do not satisfy this condition. 
\begin{prop}
Let $n\ge3$ be an integer and $(C,\vx)\sb\C^{n+1}$ the isolated singularity defined by $z_1^2+\dots+z_{n+1}^2=0$ using the coordinates $z_1,\dots,z_{n+1}$ of $\C^{n+1}.$
Let $U\sb C$ be a Stein neighbourhood of $\vx.$ Then $H^1(U\-\{\vx\},\Om^{n-2}_C)$ is non-zero. 
\end{prop}
\begin{proof}
Denote by $Y\sb\P^n$ the Fano manifold defined by the same homogeneous polynomial. In Definition \ref{dfn: Fano} then we can take $L$ to be the line bundle corresponding to $\O_Y(-1):=\O_{\P^n}(-1)|_Y.$ Contracting the zero-section $Y\sb L$ we reproduce $(C,\vx).$  Suppose $n=3.$ Then $Y$ is isomorphic to $\P^1\times\P^1$ and in particular toric. The proof of Proposition \ref{prop: toric} shows that there are isomorphisms $H^1_Y(L,\Om^1_L)\cong H^2(Y,\Om^2_Y)\cong \C,$ $H^1(L,\Om^1_L)\cong H^1(Y,\Om^1_Y)\cong \C^2$ and $H^2_Y(L,\Om^1_L)\cong H^1(Y,\Om^2_Y)=0.$ The exact sequence $H^1_Y(L,\Om^1_L)\to H^1(L,\Om^1_L)\to H^1(U\-\{\vx\},\Om^1_C)\to H^2_Y(L,\Om^1_L)$ implies therefore $H^1(U\-\{\vx\},\Om^1_C)\ne0.$ 

For $n\ge4$ we do different computation. We show first that $H^1(L,\Om^{n-2}_L)$ is non-zero. For $p,k\in\Z$ put $\Om^p_{\P^n}(k):=\Om^p_{\P^n}\otimes \O_{\P^n}(k)$ and $\Om_Y^p(k)\cong\Om^p_Y\otimes \O_Y(k).$ By \cite[Proposition 14.4]{Bott} we have 
\begin{equation}\l{Bott} \parbox{10cm}{
$H^q(\mathbb P^n, \Om^p_{\mathbb P^n} (k))=0$ unless {\bf(i)} $k=0$ and $p=q,$ {\bf(ii)} $q=0$ and $k>p,$ or {\bf(iii)} $q=n$ and $k<p-n.$
}\end{equation}
In particular, $H^1(\P^n,\Om^{n-1}_{\P^n}(n-3))=0$ and $H^0(\P^n,\Om^{n-1}_{\P^n}(n-1))=0.$ The exact sequence $0\to \Om^{n-1}_{\P^n}(n-3)\to \Om^{n-1}_{\P^n}(n-1)\to \Om^{n-1}_{\P^n}(n-1)|_Y\to0$ implies then $H^0(Y,\Om^{n-1}_{\P^n}(n-3))=0.$ The exact sequence $0\to \Om^{n-2}_Y(n-3)\to \Om^{n-1}_{\P^n}(n-1)|_Y\to \Om^{n-1}_Y(n-1)\to 0$ induces therefore an injective map $H^0(Y, \Om^{n-1}_Y(n-1))\to H^1(Y,\Om^{n-2}_Y(n-3)).$
On the other hand, there are isomorphisms $\Om^{n-1}_Y\cong \Om^n_{\P^n}(2)|_Y\cong \O_Y(1-n)$ and $H^0(Y, \Om^{n-1}_Y(n-1))\cong H^0(Y,\O_Y)\cong \C\ne0.$ This may however be regarded as a subspace of $H^1(Y,\Om^{n-2}_Y(n-3))$ so the latter is non-zero. The proof of Proposition \ref{prop: toric} shows that
\e H^1(L,\Om^{n-2}_L)\cong \Bigl(\prod_{k=0}^\iy H^1(L,\Om^{n-2}_L(k))\Bigr)\oplus \Bigl(\prod_{k=1}^\iy H^1(L,\Om^{n-3}_L(k))\Bigr)\e
which is non-zero too. 

We show next that $H^{n-1}(L,\Om^2_L)=0.$ For $k\ge0$ it follows from \eq{Bott} with $n\ge4$ that $H^{n-1}(\P^n,\Om^2_{\P^n}(k))=H^n(\P^n,\Om^2_{\P^n}(k-2))=0.$ The exact sequence $0\to \Om^2_{\P^n}(k-2)\to\Om^2_{\P^n}(k)\to\Om^2_{\P^n}(k)|_Y\to0$ implies therefore $H^{n-1}(Y,\Om^2_{\P^n}(k))=0.$ As $Y$ is of dimension $n-1$ we have also $H^n(Y,\Om^1_{\P^n}(k-2))=0.$ The exact sequence $0\to\Om^1_Y(k-2)\to\Om^2_{\P^n}(k)|_Y\to \Om^2_Y(k)\to0$ implies now $H^{n-1}(Y,\Om^2_Y(k))=0.$ In the same way we can verify that for $k\ge1$ we have $H^{n-1}(Y,\Om^1_Y(k))=0.$ The proof of Proposition \ref{prop: toric} shows that
\e H^{n-1}(L,\Om^2_L)\cong \Bigl(\prod_{k=0}^\iy H^{n-1}(L,\Om^2_L(k))\Bigr)\oplus \Bigl(\prod_{k=1}^\iy H^{n-1}(L,\Om^1_L(k))\Bigr)\e
which vanishes. By the duality theorem \cite{Kar} therefore $H^1_Y(L,\Om^{n-2}_L)=0$ and there is accordingly an exact sequence $0\to H^1(L,\Om^{n-2}_L)\to H^1(U\-\{\vx\},\Om^{n-2}_L).$ Hence we find that $H^1(U\-\{\vx\},\Om^{n-2}_L)\ne0.$
\end{proof}

Institute of Mathematical Sciences, ShanghaiTech University, 393 Middle Huaxia Road, Pudong New District, Shanghai, China 

e-mail address: yosukeimagi@shanghaitech.edu.cn

\end{document}